\documentclass[11pt]{amsart}

\usepackage{subcaption}
\usepackage{subfiles}
\usepackage{comment}
\usepackage{float}
\usepackage{marginnote}
\usepackage{tabu}
\usepackage{euscript}
\usepackage{mathdots}
\usepackage[dvipsnames]{xcolor}
\usepackage{graphicx}			
\usepackage{amssymb}
\usepackage{mathrsfs}
\usepackage{amsmath}
\usepackage{amsthm}
\usepackage{stmaryrd}
\usepackage{tikz}
\usepackage{tikz-cd}
\usepackage{accents}
\usepackage{upgreek}
\usepackage{enumerate}
\usepackage{bm}
\usepackage{mathtools}
\usetikzlibrary{patterns}
\usepackage[all]{xy}
\usepackage{caption}
\usepackage{url}
\usepackage{float}
\usepackage{todonotes} 
\usepackage{colonequals}
\usepackage{bbm}
\usepackage{longtable}
\usepackage[full]{textcomp}
\usepackage[cal=cm]{mathalfa}
\usepackage{xparse}
\usepackage{comment}
\usepackage{cite}
\usepackage{enumitem}

\makeatletter
\DeclareRobustCommand{\cev}[1]{%
  \mathpalette\do@cev{#1}%
}
\newcommand{\do@cev}[2]{%
  \fix@cev{#1}{+}%
  \reflectbox{$\m@th#1\vec{\reflectbox{$\fix@cev{#1}{-}\m@th#1#2\fix@cev{#1}{+}$}}$}%
  \fix@cev{#1}{-}%
}
\newcommand{\fix@cev}[2]{%
  \ifx#1\displaystyle
    \mkern#23mu
  \else
    \ifx#1\textstyle
      \mkern#23mu
    \else
      \ifx#1\scriptstyle
        \mkern#22mu
      \else
        \mkern#22mu
      \fi
    \fi
  \fi
}

\usetikzlibrary{calc}
\usetikzlibrary{fadings}
\usetikzlibrary{decorations.pathmorphing}
\usetikzlibrary{decorations.pathreplacing}
\usepackage{tikz,tikz-cd,tikz-3dplot}
\usepackage{pgfplots}
\usetikzlibrary{arrows,shadows,positioning, calc, decorations.markings, 
	hobby,quotes,angles,decorations.pathreplacing,intersections,shapes}
\usepgflibrary{shapes.geometric}
\usetikzlibrary{fillbetween,backgrounds}

\usepackage[margin=0.9in]{geometry}

\tikzset{
	commutative diagrams/.cd, 
	arrow style=tikz, 
	diagrams={>=stealth}
}
\tikzset{
	arrow/.pic={\path[tips,every arrow/.try,->,>=#1] (0,0) -- +(0,4pt);},
	pics/arrow/.default={triangle 90}
}
\tikzset{->-/.style={decoration={
			markings,
			mark=at position .6 with {\arrow{latex}}},postaction={decorate}}
}
\tikzset{
	c/.style={every coordinate/.try}
}

\theoremstyle{theorem}
\newtheorem{innercustomthm}{Theorem}
\newenvironment{customthm}[1]
{\renewcommand\theinnercustomthm{#1}\innercustomthm}
{\endinnercustomthm}

\theoremstyle{theorem}

\theoremstyle{theorem}

\makeatletter
\def\@tocline#1#2#3#4#5#6#7{\relax
	\ifnum #1>\c@tocdepth 
	\else
	\par \addpenalty\@secpenalty\addvspace{#2}%
	\begingroup \hyphenpenalty\@M
	\@ifempty{#4}{%
		\@tempdima\csname r@tocindent\number#1\endcsname\relax
	}{%
		\@tempdima#4\relax
	}%
	\parindent\z@ \leftskip#3\relax \advance\leftskip\@tempdima\relax
	\rightskip\@pnumwidth plus4em \parfillskip-\@pnumwidth
	#5\leavevmode\hskip-\@tempdima
	\ifcase #1
	\or\or \hskip 1em \or \hskip 2em \else \hskip 3em \fi%
	#6\nobreak\relax
	\dotfill\hbox to\@pnumwidth{\@tocpagenum{#7}}\par
	\nobreak
	\endgroup
	\fi}
\makeatother

\newcounter{marginnote}
\DeclareMathAlphabet{\mathpzc}{OT1}{pzc}{m}{it}

\usepackage[backref=page]{hyperref}
\hypersetup{
  colorlinks   = true,          
  urlcolor     = PineGreen,          
  linkcolor    = PineGreen,          
  citecolor   = PineGreen             
}

\usepackage{cleveref}

\theoremstyle{theorem}
\newtheorem{theorem}{Theorem}[section]

\newtheorem{lemma}[theorem]{Lemma}
\newtheorem{proposition}[theorem]{Proposition}
\newtheorem*{problem*}{Problem}

\theoremstyle{definition}

\newtheorem{remark}[theorem]{Remark}

\newtheorem*{runningexample*}{Running example}

\newtheorem*{aside*}{Aside}

\newtheorem{construction}[theorem]{Construction}

\newtheorem{definition}[theorem]{Definition}
\newtheorem{example}[theorem]{Example}

\newtheorem{notation}[theorem]{Notation}
\newtheorem{proposition-definition}[theorem]{Proposition-Definition}

\newcommand{\xdashleftrightarrow}[2][]{\ext@arrow 3359\leftrightarrowfill@@{#1}{#2}}

\DeclareMathOperator{\ev}{ev}

\newcommand{\tw}{\mathsf{tw}}

\newcommand{\Tsf}{\mathsf{T}}
\newcommand{\TropMap}{\Trop_\Lambda(\Rplus^k)}

\newcommand{\MTropOrb}{\mathsf{M}^{\mathsf{tw}}}
\newcommand{\MTrop}{\mathsf{M}}

\newcommand{\Msf}{\mathsf{M}}

\newcommand{\Rplus}{\RR_{\geqslant 0}}

\newcommand{\RR}{\mathbb{R}}
\newcommand{\EE}{\mathbf{E}}

\newcommand{\Vcal}{\mathcal{V}}

\newcommand{\sqC}{\scalebox{0.8}[1.2]{$\sqsubset$}}

\newcommand{\Trop}{\mathsf{Trop}}

\newcommand{\Gm}{\mathbb{G}_{\operatorname{m}}}

\newcommand{\ol}[1]{\overline{#1}}

\newcommand{\bcd}{\begin{center}\begin{tikzcd}}
		\newcommand{\ecd}{\end{tikzcd}\end{center}}
\newcommand{\Tlog}{T}

\newcommand{\Aaff}{\mathbb{A}}
\newcommand{\Aone}{\mathbb{A}^{\! 1}}

\newcommand{\PP}{\mathbb{P}}
\newcommand{\OO}{\mathcal{O}}
\newcommand{\N}{\mathbb{N}}
\newcommand{\Z}{\mathbb{Z}}

\newcommand{\refined}{\mathsf{ref}}

\newcommand{\Punct}{\mathsf{Punct}}
\newcommand{\Orb}{\mathsf{Orb}}
\newcommand{\PunctOrb}{\mathsf{PunctOrb}}
\newcommand{\NPunct}{\mathsf{NPunct}}
\newcommand{\NOrb}{\mathsf{NOrb}}
\newcommand{\NPunctOrb}{\mathsf{NPunctOrb}}
\newcommand{\Log}{\mathsf{Log}}
\newcommand{\NLog}{\mathsf{NLog}}

\newcommand{\R}{\mathbb{R}}

\newcommand{\virt}{\mathsf{vir}}

\newcommand{\kfield}{\Bbbk}

\newcommand{\Lcal}{\mathcal{L}}

\newcommand{\Acal}{\mathcal{A}}
\newcommand{\Dcal}{\mathcal{D}}
\newcommand{\NAcal}{\mathcal{N\!A}}
\newcommand{\NVcal}{\mathcal{NV}}

\newcommand{\Bcal}{\mathcal{B}}

\newcommand{\Ccal}{\mathcal{C}}
\newcommand{\Mfrak}{\mathfrak{M}}

\newcommand{\Rder}{\mathbf{R}^\bullet}

\newcommand{\Kup}{\mathsf{K}}

\newcommand{\ptrop}{\mathsf{p}}

\makeatletter
\let\save@mathaccent\mathaccent
\newcommand*\if@single[3]{%
  \setbox0\hbox{${\mathaccent"0362{#1}}^H$}%
  \setbox2\hbox{${\mathaccent"0362{\kern0pt#1}}^H$}%
  \ifdim\ht0=\ht2 #3\else #2\fi
  }
\newcommand*\rel@kern[1]{\kern#1\dimexpr\macc@kerna}
\newcommand*\widebar[1]{\@ifnextchar^{{\wide@bar{#1}{0}}}{\wide@bar{#1}{1}}}
\newcommand*\wide@bar[2]{\if@single{#1}{\wide@bar@{#1}{#2}{1}}{\wide@bar@{#1}{#2}{2}}}
\newcommand*\wide@bar@[3]{%
  \begingroup
  \def\mathaccent##1##2{%
    \let\mathaccent\save@mathaccent
    \if#32 \let\macc@nucleus\first@char \fi
    \setbox\z@\hbox{$\macc@style{\macc@nucleus}_{}$}%
    \setbox\tw@\hbox{$\macc@style{\macc@nucleus}{}_{}$}%
    \dimen@\wd\tw@
    \advance\dimen@-\wd\z@
    \divide\dimen@ 3
    \@tempdima\wd\tw@
    \advance\@tempdima-\scriptspace
    \divide\@tempdima 10
    \advance\dimen@-\@tempdima
    \ifdim\dimen@>\z@ \dimen@0pt\fi
    \rel@kern{0.6}\kern-\dimen@
    \if#31
      \overline{\rel@kern{-0.6}\kern\dimen@\macc@nucleus\rel@kern{0.4}\kern\dimen@}%
      \advance\dimen@0.4\dimexpr\macc@kerna
      \let\final@kern#2%
      \ifdim\dimen@<\z@ \let\final@kern1\fi
      \if\final@kern1 \kern-\dimen@\fi
    \else
      \overline{\rel@kern{-0.6}\kern\dimen@#1}%
    \fi
  }%
  \macc@depth\@ne
  \let\math@bgroup\@empty \let\math@egroup\macc@set@skewchar
  \mathsurround\z@ \frozen@everymath{\mathgroup\macc@group\relax}%
  \macc@set@skewchar\relax
  \let\mathaccentV\macc@nested@a
  \if#31
    \macc@nested@a\relax111{#1}%
  \else
    \def\gobble@till@marker##1\endmarker{}%
    \futurelet\first@char\gobble@till@marker#1\endmarker
    \ifcat\noexpand\first@char A\else
      \def\first@char{}%
    \fi
    \macc@nested@a\relax111{\first@char}%
  \fi
  \endgroup
}
\makeatother

\newcommand{\Mbar}{\widebar{M}}

\newcommand{\FF}{\mathbb{F}}

\newcommand{\Spec}{\operatorname{Spec}}
\newcommand{\acts}{\curvearrowright}

\newcommand{\ftrop}{\mathsf{f}}

\crefname{equation}{eq.}{eqs.}
\crefname{eqnarray}{eq.}{eqs.}
\crefname{conjecture}{conjecture}{conjectures}
\crefname{lemma}{lemma}{lemmas}
\crefname{theorem}{theorem}{theorems}
\crefname{claim}{claim}{claims}
\crefname{remark}{remark}{remarks}
\crefname{proposition}{proposition}{propositions}
\crefname{section}{section}{sections}
\crefname{appendix}{appendix}{appendices}
\crefname{corollary}{corollary}{corollaries}
\crefname{figure}{figure}{figures}
\crefname{table}{table}{tables}
\crefname{example}{example}{examples}
\crefname{assumption}{assumption}{assumptions}
\crefname{definition}{definition}{definitions}
\crefname{innercustomthm}{theorem}{theorems}
\crefname{innercustomconj}{conjecture}{conjectures}

\setlist[enumerate,1]{label=(\roman*),itemsep=0.9ex}
\setlist[itemize]{itemsep=0.9ex}

\begin{document}

\title[Negative contacts]{Logarithmic negative tangency and root stacks}
\author{Luca Battistella, Navid Nabijou, Dhruv Ranganathan}

\begin{abstract}
We study stable maps to normal crossings pairs with possibly negative tangency orders. There are two independent models: punctured Gromov--Witten theory of pairs and orbifold Gromov--Witten theory of root stacks with extremal ages. Exploiting the tropical structure of the punctured mapping space, we define and study a new virtual class for the punctured theory. This arises as a refined intersection product on the Artin fan, and produces a distinguished sector of punctured Gromov--Witten invariants. Restricting to genus zero, we show that these invariants coincide with the orbifold invariants, first for smooth pairs, and then for normal crossings pairs after passing to a sufficiently refined blowup. This builds on previous work to provide a complete picture of the logarithmic-orbifold comparison in genus zero, which is compatible with splitting and thus allows for the wholesale importation of orbifold techniques, including boundary recursion and torus localisation. Contemporaneous work of Johnston uses the comparison to give a new proof of the associativity of the Gross--Siebert intrinsic mirror ring.\vspace{-0.5cm}
\end{abstract}

\maketitle
\tableofcontents

\, \vspace{-1cm}
\section*{Introduction} \label{sec: introduction}
\subsection{Outline} The subject of this paper is the enumerative geometry of maps
\begin{equation} \label{eqn: map curve to target introduction}
(C|p_1+\cdots+p_n)\to (X|D = D_1+\cdots+D_k)
\end{equation}
from a rational curve $C$ with specified, possibly negative, tangency orders
\[ \upalpha_{ij} \in \Z \]
at each point $p_i$ with respect to each divisor $D_j$. Rational curves with negative tangency have been the subject of significant recent interest due to their role in quantum cohomology and mirror symmetry~\cite{PuncturedMaps,GS19,GrossSiebertIntrinsic,FWY20}. There are two distinct models:
\begin{itemize}
\item \textbf{The punctured path.} The map is enhanced to a punctured logarithmic map. Locally near a point $p_i$ with negative tangency $\upalpha_{ij} < 0$, the curve must lift to the total space of the normal bundle of $D_j$ and meet the zero section at $p_i$ with tangency $-\upalpha_{ij}$. The lifting condition can be efficiently phrased in terms of logarithmic structures. This theory was developed by Abramovich--Chen--Gross--Siebert~\cite{AbramovichChenLog,ChenLog,GrossSiebertLog,PuncturedMaps}.
\item \textbf{The root route.} The target $X$ is enhanced to an orbifold, endowing each $D_j$ with a root of order $r_j$. Tangency is imposed by enhancing $C$ to an orbifold and imposing that the age of $p_i$ with respect to $D_j$ is $\upalpha_{ij}/r_j$ (mod $1$). This approach was pioneered by Cadman \cite{Cadman} and developed in stages by numerous authors \cite{CadmanChen,AbramovichCadmanWise,TsengYouHigherGenus,TsengYouSNC}. In the presence of negative tangency, the theory was developed by Fan--Wu--You~\cite{FWY20,FWY21}.
\end{itemize}
Negative tangencies are an inescapable feature of both theories. Given a logarithmic or orbifold map \eqref{eqn: map curve to target introduction} with positive tangencies, restricting to a subcurve of $C$ will in general produce a map with negative tangencies. This restriction operation is crucial in applying standard techniques in Gromov--Witten theory, such as boundary recursion and torus localisation.

The two theories have complementary geometric features, and are amenable to different techniques. The logarithmic theory benefits from close links to tropical curve counting and scattering diagrams, whereas the orbifold theory enjoys the standard arsenal of techniques in Gromov--Witten theory, including boundary recursion and torus localisation. This counterpoint has prompted significant interest in a comparison incorporating negative tangencies, extending and completing earlier work in the positive tangency setting \cite{AbramovichCadmanWise,TsengYouHigherGenus,BNR2}.

The main contributions of the present paper are as follows:
\begin{enumerate}
\item \textbf{Refined punctured theory.} We introduce a new virtual fundamental class on the space of punctured logarithmic maps, called the refined virtual class. This is defined in all genera. It is a global class of pure dimension, obtained without restricting to a particular stratum of the moduli space. It can be calculated in terms of strata and tautological classes. 
\item \textbf{Punctured-orbifold comparison.} We prove that for maps from rational curves to smooth pairs, the refined punctured theory coincides with the orbifold theory. Employing the rank reduction bootstrap, we extend this comparison result to normal crossings pairs, after passing to a sufficiently refined blowup of the target pair, depending only on the numerical data.
\end{enumerate}

The refined punctured theory arises naturally when pulling back strata classes from the moduli space of curves, and is independent of any ideas coming from orbifold geometry. It seems likely to find wider applications.

Our work builds on and connects to that of many others. The construction of the refined punctured theory gives a conceptual account of the classes introduced by Fan--Wu--You for smooth pairs via localisation residues~\cite{FWY20}. Our conceptual definition applies to all simple normal crossings pairs, providing a generalisation which is not accessible via the localisation approach. In genus zero the two definitions coincide by comparing both to the orbifold theory. Based on basic examples and the ingredients that go into the formula that defines the higher genus theory of~\cite{FWY21} we would guess this is also true in higher genus, but we do not pursue this here. 

As mentioned above, the punctured-orbifold comparison permits the interlacing of logarithmic and orbifold techniques. This sets it apart from results comparing different models \emph{within} logarithmic Gromov--Witten theory, which are significant foundationally but rarely provide access to new techniques \cite{AbramovichMarcusWise,RangExpansions}. Applications have already been found. The positive tangency correspondence has been used to establish a reconstruction principle for genus zero logarithmic Gromov--Witten invariants \cite[Corollary~Y]{BNR2}. Contemporaneous work of Johnston \cite{JohnstonFrobenius} uses the punctured-orbifold comparison of the present paper to deduce associativity of the intrinsic mirror ring, a landmark result in the Gross--Siebert program~\cite{GS19}. 

Another application is to torus localisation. Starting with a logarithmic invariant, we convert it to an orbifold invariant and apply localisation, thus obtaining an in-principle localisation schema in genus zero. This provides a structural and calculational yardstick for the pursuit of localisation entirely within logarithmic Gromov--Witten theory. Moreover, the isomorphism $\upomega$ established in \Cref{thm: main smooth pairs introduction} endows (for the first time) the space of stable logarithmic maps with an absolute perfect obstruction theory. The lack of an absolute perfect obstruction theory has been a fundamental obstacle to implementing torus localisation in the logarithmic setting \cite{GraberMFOReport}. With this new obstruction theory, the space carries all of the formal properties required for applying the results of \cite{GraberPandharipande}.

\subsection{The key construction: the refined punctured theory} \label{sec: refined punctured theory introduction} The moduli space of punctured maps differs fundamentally from its positive tangency counterpart. Fix a pair $(X|D)$, and let $\Sigma=\Sigma(X|D)$ denote the tropicalisation and $\Acal(\Sigma)$ its Artin fan. This is a locally toric stack equipped with a boundary divisor, and is the universal pair whose boundary complex is identified with that of $(X|D)$.

Given numerical data $\Lambda$, there is a morphism of moduli stacks
\begin{equation} \label{eqn: map to maps to universal target}
 \Punct_\Lambda(X|D) \xrightarrow{\upmu} \Punct_\Lambda(\Acal(\Sigma)|\partial \Acal(\Sigma))
 \end{equation}
 equipped with a relative perfect obstruction theory. The key difference with positive tangency is that the base of this obstruction theory is typically reducible, and its connected components can fail to be equidimensional. In \cite{PuncturedMaps} the punctured theory is defined by taking virtual pullback of \emph{any} Chow class from the base of \eqref{eqn: map to maps to universal target}. Usually this is the class of an irreducible component, or more generally a stratum, and such classes admit an explicit tropical enumeration.

While $\Punct_\Lambda(\Acal(\Sigma)|\partial \Acal(\Sigma))$ may be singular, it is always \emph{idealised} logarithmically smooth, meaning that it is locally isomorphic to a monomial subscheme of a toric variety (in keeping with universality results such as~\cite{Universality,Vak06}, one expects arbitrary monomial subschemes to arise). This gives
\[
\begin{tikzcd}
\Punct_\Lambda(\Acal(\Sigma)|\partial \Acal(\Sigma)) \ar[d,"\upnu"] & \\
\Vcal(\Tsf) \ar[r,hook] & \Acal(\Tsf)
\end{tikzcd}
\]
where $\Acal(\Tsf)$ is the Artin fan of the space of punctured maps $\Punct_\Lambda(\Acal(\Sigma)|\partial \Acal(\Sigma))$ and $\Vcal(\Tsf) \hookrightarrow \Acal(\Tsf)$ is the inclusion of a monomial substack (the \textbf{puncturing substack}). The morphism $\upnu$ is smooth and encodes the stratification of the punctured mapping space. The irreducible components of $\Vcal(\Tsf)$ biject with those of $\Punct_\Lambda(\Acal(\Sigma)|\partial \Acal(\Sigma))$.

We now explain the failure of equidimensionality in conceptual terms. This will lead naturally to the definition of the refined punctured theory. Given a marking $p_i$ and a divisor component $D_j$ such that $\upalpha_{ij} < 0$, the irreducible component of $C$ containing $p_i$ must map entirely inside $D_j$. This imposes $k_P$ conditions on the moduli space, where $k_P$ is the \textbf{total puncturing rank} recording the total number of negative directions of tangency at the punctures (see \Cref{def: puncturing rank}). However, these $k_P$ conditions are not independent: for example, when two markings lie on the same irreducible component of $C$, we obtain the same condition twice.

These observations (and their tropical counterparts) explain why $\Vcal(\Tsf) \hookrightarrow \Acal(\Tsf)$ fails to be a regular embedding, and hence why $\Vcal(\Tsf)$ and $\Punct_\Lambda(\Acal(\Sigma)|\partial \Acal(\Sigma))$ fail to be equidimensional. On the other hand, we can express these $k_P$ conditions via the following cartesian square:
\[
\begin{tikzcd}
\Vcal(\Tsf) \ar[r,hook] \ar[d] \ar[rd,phantom,"\square"] & \Acal(\Tsf) \ar[d] \\
\Bcal \Gm^{k_P} \ar[r,hook,"\upiota"] & \Acal^{k_P}
\end{tikzcd}
\] 
where $\Acal=[\Aone/\Gm]$. This describes $\Vcal(\Tsf) \hookrightarrow \Acal(\Tsf)$ as the pullback of a regular embedding, or equivalently as the intersection of $k_P$ Cartier divisors inside $\Acal(\Tsf)$. From this description we obtain a refined intersection class, which we denote:
\[ 
[\Vcal(\Tsf)]^{\refined} \colonequals \upiota^![\Acal(\Tsf)].
\]
Applying smooth and virtual pullback, we obtain \textbf{refined virtual classes} (\Cref{def: refined virtual class}):
\begin{align*}
[\Punct_\Lambda(\Acal(\Sigma)|\partial \Acal(\Sigma))]^{\refined} & \colonequals \upnu^\star [\Vcal(\Tsf)]^{\refined}, \\
[\Punct_\Lambda(X|D)]^{\refined} & \colonequals \upmu^! [\Punct_\Lambda(\Acal(\Sigma)|\partial \Acal(\Sigma))]^{\refined}.
\end{align*}
The latter class has pure dimension
\[
(\dim X-3)(1-g)-(K_X+D)\cdot \upbeta+n-k_P
\]
where $k_P$ is the total puncturing rank introduced above. We stress that the refined virtual class already differs from the fundamental class when $D$ is smooth, in which case $k_P$ is the number of punctures.

Building on this conceptual definition, we exploit intersection theory on the Artin fan to express the refined virtual class in terms of strata and tautological classes. Details and examples are given in \Cref{sec: Aluffi}. In the smooth pair case, the resulting formulae resemble the classes appearing in the graph sum definition of negative contact invariants given by Fan--Wu--You \cite[Section~3]{FWY21}. However our constructions are logically independent, and a direct comparison is complicated by subtle differences between the moduli spaces: logarithmic without expansions versus non-logarithmic with expansions. Nevertheless, \Cref{thm: main smooth pairs introduction} below combined with \cite[Theorem~3.13]{FWY21} shows indirectly that the associated invariants coincide, at least in genus zero.

\begin{remark}[Global nature of the refined punctured theory] The space $\Punct_\Lambda(\Acal(\Sigma)|\partial \Acal(\Sigma))$ has a stratification indexed by tropical types, with irreducible components corresponding to minimal tropical types. It is common to fix a single tropical type and focus solely on the corresponding stratum, see e.g. \cite{JohnstonBirational}. This avoids difficulties caused by the non-equidimensionality of the moduli space.

With the refined punctured theory, it is crucial to tackle such difficulties head-on. The refined virtual class is defined intrinsically on the entire punctured moduli space. It is not pieced together from classes defined on strata (although, as discussed above, such a description can be obtained a posteriori). Accordingly, our numerical data (\Cref{sec: numerical data}) isolate entire connected components of the punctured moduli space, in contrast with tropical types, which only isolate closed substacks.
\end{remark}

\subsection{Main results: punctured-orbifold comparisons} We now restrict to genus zero. Our main theorems equate the refined punctured invariants with the orbifold invariants, under appropriate hypotheses. The refined virtual class is strictly necessary here: otherwise, the theories may not even have the same virtual dimension. We first state the comparison for smooth pairs.

\begin{customthm}{A}[\Cref{thm: correspondence smooth pairs}] \label{thm: main smooth pairs introduction}
Let $(X|D)$ be a smooth pair. In genus zero, with fixed numerical data $\Lambda$, the refined punctured Gromov--Witten theory coincides with the orbifold Gromov--Witten theory up to a rooting factor. Precisely, fix $r$ sufficiently large and let $X_r$ denote the $r$th root stack of $X$ along $D$. Then there is a correspondence
\[
\begin{tikzcd}
& \PunctOrb_\Lambda(X_r|D_r) \ar[ld,"\upalpha"] \ar[rd,"\upomega" {yshift=-0.35cm,xshift=-0.3cm}] & \\
\Punct_\Lambda(X|D) & & \Orb_\Lambda(X_r).
\end{tikzcd}
\]
satisfying the following virtual pushforward identities
\begin{align*}
\upalpha_\star[\PunctOrb_\Lambda(X_r|D_r)]^{\refined} & = r^{-k_P} [\Punct_\Lambda(X|D)]^{\refined} \\[0.2cm]
\upomega_\star[\PunctOrb_\Lambda(X_r|D_r)]^{\refined} & = [\Orb_\Lambda(X_r)]^{\virt}.
\end{align*}
Moreover $\upalpha$ induces an isomorphism on coarse spaces, while $\upomega$ is an isomorphism.
\end{customthm}
\noindent We highlight two notable features. First, the comparison of spaces is significantly stronger than parallel results in the literature~\cite{AbramovichCadmanWise}. We do not simply prove that the spaces are birational or even virtually birational, but identify the spaces up to explicit root constructions\footnote{There is also a superficial difference between the comparison here and in~\cite{AbramovichCadmanWise}, arising from the fact that they use J. Li's expanded model for relative Gromov--Witten theory. In particular, it is noted there that the analogue of the map $\upomega$ is ``not an isomorphism even for large $r$''. This occurs because its domain carries the data of an expansion of the target. See~\cite{AbramovichMarcusWise} for further details on comparisons between models of logarithmic Gromov--Witten theory.}. Second, this stronger comparison of spaces, even at the level of universal targets, is actually insufficient to conclude the equality of invariants. As we have discussed, unlike the spaces of maps to universal targets appearing in~\cite{AbramovichWiseBirational,AbramovichCadmanWise,AbramovichMarcusWise}, the space $\Punct_\Lambda(\Acal(\Sigma)|\partial \Acal(\Sigma))$ carries a nontrivial virtual structure. A delicate comparison of obstruction theories is required even at the level of universal targets. See \Cref{sec: subtleties of negative tangency} for a detailed discussion.

The morphism $\upalpha$ is constructed explicitly in \Cref{def: chimera}. At the level of universal targets, it is the pullback of a root construction along a stratum inclusion. It therefore behaves like a root construction followed by a finite gerbe. This should be contrasted with the situation of positive tangency, where the analogous map is simply a root construction. This explains the rational factor in the pushforward identity. 

We now state the generalisation to simple normal crossings pairs. Given a pair $(X|D)$, an \textbf{iterated blowup} is a sequence of blowups of closed strata. Given numerical data $\Lambda$, our previous work \cite{BNR2} identified a condition on iterated blowups called \textbf{$\Lambda$-sensitivity} (see \Cref{sec: slope-sensitivity}). It is satisfied by all sufficiently refined iterated blowups. 

\begin{customthm}{B}[\Cref{thm: main}] \label{thm: main snc pairs introduction} Let $(X|D)$ be a simple normal crossings pair and fix numerical data $\Lambda$ for stable punctured maps in genus zero. For each $\Lambda$-sensitive blowup $(X^\prime|D^\prime) \to (X|D)$ the refined punctured Gromov--Witten theory of $(X^\prime|D^\prime)$ coincides with the orbifold Gromov--Witten theory of $X^\prime_{\bm{r}^\prime}$ up to a rooting factor. Precisely, there is a diagram of moduli spaces
\[
\begin{tikzcd}[column sep=small]
\Punct_{\Lambda^\prime}(X^\prime|D^\prime) \ar[rd,"\uptheta" below] && \NPunctOrb_{\Lambda^\prime}(X^\prime_{\bm{r}^\prime}|D^\prime_{\bm{r}^\prime}) \ar[ld,"\upalpha"] \ar[rd,"\upomega" below] && \Orb_{\Lambda^\prime}(X^\prime_{\bm{r}^\prime}) \ar[ld,"\upeta"] \\
& \NPunct_{\Lambda^\prime}(X^\prime|D^\prime) && \NOrb_{\Lambda^\prime}(X^\prime_{\bm{r}^\prime})	
\end{tikzcd}
\]
satisfying the following virtual pushforward identities
\begin{align*}
\uptheta_\star[\Punct_{\Lambda^\prime}(X^\prime|D^\prime)]^{\refined} & = [\NPunct_{\Lambda^\prime}(X^\prime|D^\prime)]^{\refined} \\
\upalpha_\star[\NPunctOrb_{\Lambda^\prime}(X^\prime_{\bm{r}^\prime}|D^\prime_{\bm{r}^\prime})]^{\refined} & = \Pi_{j=1}^{k^\prime} (r_{j}^\prime)^{-n_j^\prime} [\NPunct_{\Lambda^\prime}(X^\prime|D^\prime)]^{\refined} \\
\upomega_\star[\NPunctOrb_{\Lambda^\prime}(X^\prime_{\bm{r}^\prime}|D^\prime_{\bm{r}^\prime})]^{\refined} & = [\NOrb_{\Lambda^\prime}(X^\prime_{\bm{r}^\prime})]^{\virt} \\
\upeta_\star[\Orb_{\Lambda^\prime}(X^\prime_{\bm{r}^\prime})]^{\virt} & = [\NOrb_{\Lambda^\prime}(X^\prime_{\bm{r}^\prime})]^{\virt}.
\end{align*}
The rooting factor $\Pi_{j=1}^{k^\prime}(r_{j}^\prime)^{-n_j^\prime}$ is defined in \Cref{sec: main theorem}.
\end{customthm}

We restrict to \emph{simple} normal crossings pairs. Arbitrary normal crossings pairs are included by performing an initial sequence of blowups to remove self-intersections of the boundary divisors.

For smooth pairs there are no non-trivial blowups, and one recovers \Cref{thm: main smooth pairs introduction}. The maps $\uptheta$ and $\upeta$ are isomorphisms in this case. In line with our previous work for positive tangencies, the ``$\mathsf N$'' spaces in the diagram are ``naive'' versions of the simple normal crossings theory, constructed from the smooth pair theories $(X|D_j)$ using intersection products.

\subsection{Subtleties of negative tangency} \label{sec: subtleties of negative tangency} Except for the appearance of the refined virtual class, the above comparison results closely resemble previous comparison results for positive tangencies \cite{AbramovichCadmanWise,BNR2}. However the proofs are significantly more challenging, and we take a moment to explain the source of the complications. 

The standard approach to comparison theorems has been to use the fact that, while spaces of stable maps to a specific target can be badly-behaved, spaces of maps to a sufficiently universal target are irreducible. Typical universal targets include: a point, the stack $\Bcal \Gm$, and the stack $\Acal = [\Aone/\Gm]$. Thus, if two spaces of maps to the universal target are shown to be birational --- which is the case if they contain a common dense locus parametrising non-degenerate maps --- the general comparison will follow from a formal compatibility of obstruction theories. This proof technique has been used in numerous places~\cite{AbramovichWiseBirational,AbramovichCadmanWise,AbramovichMarcusWise,RangExpansions,BNR2}.

Focusing on the logarithmic-orbifold comparison, in positive tangency the spaces
\[ \Log_\Lambda(\Acal|\Dcal), \qquad \Orb_\Lambda(\Acal_r), \]
of logarithmic maps to the universal pair and twisted maps to the universal root stack are both irreducible, and are birational since they share a common dense locus parametrising non-degenerate maps. However, when negative tangencies are permitted, the corresponding spaces
\begin{equation} \label{eqn: spaces puncture orbifold} \Punct_\Lambda(\Acal|\Dcal), \qquad \Orb_\Lambda(\Acal_r), \end{equation}
are reducible and non-equidimensional (see \Cref{sec: refined punctured theory introduction}). Virtual structures enter already at the level of universal targets, and the previous strategy fundamentally breaks down. Even if one could match up the irreducible components via birational maps, if these components are not of the correct dimension this would have little consequence for the virtual class. 

Because of this, we require much finer control over the geometry of the moduli space than simply its birational type. A key step in the proof is to show that the spaces \eqref{eqn: spaces puncture orbifold} are \emph{isomorphic}, up to explicit root constructions. 

\subsection{Proof highlights} \Cref{sec: new theory of negative tangency} introduces the refined punctured theory and \Cref{sec: correspondence} states the punctured-orbifold comparison. The smooth pair comparison (\Cref{thm: main smooth pairs introduction}) is established in Sections~\ref{sec: twisted punctured moduli} and \ref{sec: proof smooth}. The general case (\Cref{thm: main snc pairs introduction}) is deduced from this special case in \Cref{sec: proof snc}. We draw attention to several novel aspects of the proof:

\begin{itemize}
\item \textbf{Constructing the chimera.} We construct a ``chimera'' moduli space $\PunctOrb_\Lambda(\Acal_r|\Dcal_r)$ interpolating between punctures and roots. We take an unusual approach: we first construct the tropical chimera (\Cref{sec: tropical twisted punctured maps}), which we then use to define the algebraic chimera as a fibre product over Artin fans (\Cref{def: chimera}). Finally, we manually verify the key properties of the algebraic chimera, such as the existence of a universal family (\Cref{sec: chimera}).
\item  \textbf{Strong comparison of moduli spaces.} We prove that the natural morphism
\[ \PunctOrb_\Lambda(\Acal_r|\Dcal_r) \to \Orb_\Lambda(\Acal_r) \]
is in fact an \emph{isomorphism} (\Cref{thm: spaces same}). This is required later when comparing obstruction theories. During the proof we construct a new logarithmic structure on $\Orb_\Lambda(\Acal_r)$, different from the logarithmic structure pulled back from the moduli space of twisted curves.
\item \textbf{Obstruction theories for universal targets.} The comparison of obstruction theories at the level of universal targets is far from formal (\Cref{thm: classes same}). It requires the use of ``positivised'' numerical data (\Cref{sec: positivised numerical data}) to obtain an embedding of the negative tangency moduli space into an auxiliary positive tangency moduli space (\Cref{rmk: negative inside positive}), followed by a careful identification of line bundles arising from logarithmic and orbifold structures (\Cref{sec: description of POTs}).
\end{itemize}

\subsection{Birational invariance} \label{sec: birational invariance introduction} 
\Cref{thm: main snc pairs introduction} identifies the refined punctured theory of a blowup with the orbifold theory of a blowup. This is a level by level comparison: it does not relate either theory to the respective theory of the original pair, nor to the respective theory of a further blowup.

The punctured theory does not satisfy the naive generalisation of birational invariance \cite{AbramovichWiseBirational}. Indeed, in the negative contact setting, even the virtual dimension can change under birational modification. This is expected: since punctured moduli appears in the boundary of logarithmic moduli, when the latter is replaced by a birational modification the former is replaced by a bundle (or a birational modification thereof). This behaviour is essentially unconstrained, and the strongest statements one can hope for in general are obtained in \cite{JohnstonBirational}.

The \emph{refined} punctured theory is more structured, and may satisfy stronger birational invariance properties. We investigate these in \Cref{sec: birational invariance}. We first show that the virtual dimension eventually stabilises under blowups, as long as the numerical data is lifted carefully (\Cref{sec: stabilisation of virtual dimension}). This is already subtle: careless lifts of the numerical data can change the virtual dimension.

In this stable range, it is at least reasonable to hope that the refined virtual class also stabilises. However we show via an explicit counterexample that this does not occur: the refined virtual class can change, even when both targets are $\Lambda$-sensitive (\Cref{sec: birational counterexample}). This reveals a surprising asymmetry in the orbifold theory: in the positive tangency setting the orbifold theory stabilises after a $\Lambda$-sensitive blowup by the results of~\cite{BNR2}, but the analogous result fails in the negative tangency setting. It also raises the following:

\begin{problem*} Let $(X^\prime|D^\prime)\to (X|D)$ be an iterated blowup. Find an explicit formula for the difference between the refined virtual classes associated to $(X^\prime|D^\prime)$ and $(X|D)$, after pushforward to the space of punctured maps to $(X|D)$. \end{problem*}

An inexplicit comparison between these classes comes out of the calculation scheme we present in Section~\ref{sec: Aluffi}, but we believe a comprehensive study to be worthwhile. One might hope for a characterisation of the difference between the two virtual classes in tropical moduli theoretic terms. The calculation scheme does give a sense for the ingredients that should go into an explicit formula: strata classes on the puncturing substack of the Artin fan of the space of punctured maps, and Chern classes of natural line bundles on this Artin fan. Each of these ingredients seems to have some tropical meaning when they arise.

\subsection{Prospects} With the results of the present paper, the comparison between the two approaches to negative tangency in genus zero is essentially complete. In addition to the questions surrounding birational invariance raised above, the remaining foundational questions centre on positive genus. Here, even for smooth pairs, the comparison between the refined punctured theory and the theory of Fan--Wu--You via double ramification cycles is incomplete~\cite{FWY21}. For simple normal crossings pairs, the comparison expected in light of our previous work~\cite[Conjecture~W]{BNR2} remains open. 

It would be interesting to apply our logarithmic-orbifold comparison to obtain complete calculations of logarithmic Gromov--Witten invariants, and to uncover hidden structures. For instance, the precise form of the translation of the orbifold WDVV relations is not known. The work of Johnston on mirror manifolds includes new calculations that may shed light on this \cite{JohnstonFrobenius}.

\subsection*{History} The refined punctured theory and the statement of the correspondence were discovered in April 2022. The ideas were subsequently communicated to the community via seminar talks and informal discussions, but the complete write-up has been significantly delayed. In the intervening period, the main result has found applications: see \cite{YouLG,JohnstonFrobenius}.

\subsection*{Acknowledgements} Parts of this work were conducted at the International Centre for Mathematical Sciences in Edinburgh supported by a Research-in-Groups grant. We are grateful to the centre and its staff for providing us with ideal working conditions. We have benefited from numerous conversations over the years with friends and colleagues concerning negative contact orders and orbifold Gromov--Witten theory, including Dan Abramovich, Dori Bejleri, Mark~Gross, Sam~Johnston, Jonathan~Wise, and Fenglong You. We extend special thanks to Honglu Fan, whose comments concerning the material in~\cite{FWY20} were particularly enlightening. We thank the anonymous referees for their valuable suggestions.

\subsection*{Conventions and notation.} We will work over an algebraic closed field of characteristic zero. All stacks will be locally of finite type. Logarithmic structures will be fine and saturated. Given a cone stack $\mathsf{T}$ we write $\Acal(\mathsf{T})$ for the corresponding Artin fan \cite{CavalieriChanUlirschWise}. We write $\Acal$ for the Artin fan $[\Aone/\Gm]$. For a positive natural number $n$ we write $[n]$ for the set $\{1,\ldots,n\}$.

\section{Refined punctured theory: construction and examples} \label{sec: new theory of negative tangency} 
 
\noindent Fix a simple normal crossings pair $(X|D)$. We choose a labelling of the irreducible components by the set $[k]$, and hence express $D=D_1+\cdots+D_k$. Write
\[ \Sigma \colonequals \Sigma(X|D) \]
for the tropicalisation. There is a natural map
\begin{equation} \label{eqn: global framing} c \colon \Sigma \to \Rplus^k \end{equation}
which is a finite cover of a subcomplex.\footnote{If every intersection of divisor components is connected, then $c$ is the inclusion of a subcomplex. The reader can assume this without significant loss of generality, but the more general context is useful in applications, see for instance~\cite{vanGarrelNabijouSchuler}.} The map sends the ray in $\Sigma$ corresponding to $D_j$ isomorphically onto the $j$th coordinate ray in $\Rplus^k$. The induced morphism of Artin fans $\Acal(\Sigma) \to \Acal^k$ is \'etale and strict. We make frequent use of this \textbf{global framing}, and the coordinate system induced by the map $c$.

\subsection{Numerical data for punctured maps} \label{sec: numerical data} We recall the basic setup for punctured Gromov--Witten theory, following the foundational paper of Abramovich--Chen--Gross--Siebert~\cite{PuncturedMaps}.

\begin{definition} \label{def: numerical data}
The \textbf{numerical data} $\Lambda$ for a moduli space of punctured maps to $(X|D)$ consists of:
\begin{itemize}
	\item \textbf{Genus.} A natural number $g \in \N$ fixing the genus of the source curve.
	\item \textbf{Curve class.} An effective curve class $\upbeta \in H_2(X;\Z)$.
	\item \textbf{Contact orders.} A number $n \in \N$ of markings, and, for every marking $i \in [n]$ and divisor component $j \in [k]$, a contact order
	\[ \upalpha_{ij} \in \Z.\]
\end{itemize}

Fix $i \in [n]$. If $\upalpha_{ij} \geqslant 0$ for all $j \in [k]$ we refer to $i$ as an \textbf{ordinary marking}. Otherwise we refer to it as a \textbf{puncture}. We decompose the markings as
\[ [n] = O \sqcup P \]
with $O$ the set of ordinary markings and $P$ the set of punctures.

For $j \in [k]$ we define the \textbf{degree} with respect to $j$ as $d_j \colonequals D_j \cdot \upbeta \in \Z$, and impose the following \textbf{global balancing constraint} on the numerical data:
\[ d_j = \Sigma_{i=1}^n \upalpha_{ij}. \]
Note that both $d_j$ and $\upalpha_{ij}$ may be negative.
\end{definition}

Fix once and for all numerical data $\Lambda$. The paper \cite{PuncturedMaps} constructs moduli stacks of punctured logarithmic maps:
\[ \Punct_\Lambda(X|D) \xrightarrow{\upeta} \Punct_\Lambda(\Acal(\Sigma)|\partial \Acal(\Sigma)) \xrightarrow{\uptheta} \Punct_\Lambda(\Acal^k|\partial \Acal^k).\]
In the first space, the underlying map is required to be stable in the usual sense. In the second two spaces, no stability condition is imposed (consequently, these spaces are only locally of finite type). The numerical data $\Lambda$ only affects these latter two spaces through the genus and contact orders, although the pairing of $\upbeta$ with each $D_j$ can be recovered from the balancing constraint.

A key structural result established in the above paper is that $\upeta$ is virtually smooth and strict, while $\uptheta$ is logarithmically \'etale and strict. It is important to note that, in sharp contrast to the positive tangency setting, the latter two stacks are typically reducible, non-reduced, and non-equidimensional. In the smooth topology, they are locally isomorphic to monomial subschemes in toric varieties.

\begin{remark}[Relationship to global contact orders] \label{rmk: global contact orders} The numerical data above should be compared to the \textbf{global contact orders} appearing in~\cite[Definition~2.40]{PuncturedMaps}, which record slightly more refined information. Tropically, a global contact order specifies a cone $\upsigma \in \Sigma$ to whose star the vertex supporting the marking leg must map, and a direction vector emanating from this cone and giving the slope along the marking leg. The latter belongs to the groupification of the lattice of a cone in the star of $\upsigma$. This setup is well-adapted to the splitting problem, where it is desirable to impose that an irreducible component of the source curve is mapped to a specific boundary stratum.

Our setup in \Cref{def: numerical data} takes all $\upsigma=0$. This does not dramatically change the theory. Moduli spaces corresponding to the choices $\upsigma \neq 0$ appear as boundary strata in the moduli space corresponding to $\upsigma=0$. Our setup is necessary in order to formulate the correspondence between refined punctured and orbifold Gromov--Witten theory, since the discrete data for the orbifold theory do not distinguish the $\upsigma \neq 0$ loci as spaces in their own right. However, the refined punctured theory  (\Cref{sec: refined punctured class}) is itself well-defined in the general setting of global contact orders.

Another difference is that we do not distinguish between cones spanned by the same set of rays, working always with respect to the global framing \eqref{eqn: global framing}. Again, this does not result in any dramatic changes.
\end{remark}

\subsection{Tropical punctured maps} \label{sec: tropical maps} The numerical data $\Lambda$ determines a moduli space of tropical punctured maps, obtained by gluing cones indexed by tropical types. We make use of the global framing \eqref{eqn: global framing} and focus on tropical maps to $\Rplus^k$. These can be coarser than tropical maps to $\Sigma$, but are sufficient for our purposes (see \Cref{rmk: proxy for Artin fan}).

\begin{definition} \label{def: type of tropical punctured map} A \textbf{tropical type} of punctured map to $\Rplus^k$ with numerical data $\Lambda$ consists of the following data:
\begin{itemize}
\item \textbf{Source graph.} A finite graph $\Gamma$ consisting of vertices $V(\Gamma)$, finite edges $E(\Gamma)$, ordinary markings $O(\Gamma)$, and punctures $P(\Gamma)$. This is equipped with bijections
\[ O(\Gamma) \cong O, \qquad P(\Gamma) \cong P \]
where $O$ and $P$ are the sets of ordinary markings and punctures in \Cref{def: numerical data}. Each vertex $v \in V(\Gamma)$ is equipped with genus and multidegree labels
\[ g_v \in \N, \quad d_v \in \Z^k\]
such that $g = b_1(\Gamma) + \Sigma_{v \in V(\Gamma)} g_v$ and $(d_j)_{j=1}^k = \Sigma_{v \in V(\Gamma)} d_v$.
\item \textbf{Image cones.} Faces of $\Rplus^k$ associated to every vertex and edge of $\Gamma$
\[ v \rightsquigarrow \upsigma_v, \quad e \rightsquigarrow \upsigma_e \]
such that if $v \in e$ then $\upsigma_v \subseteq \upsigma_e$.
\item \textbf{Slopes.} For each oriented edge $\vec{e}$ a slope vector $m_{\vec{e}} \in N_{\upsigma_e}$ satisfying $m_{\vec{e}}=-m_{\cev{e}}$. At each vertex $v \in V(\Gamma)$ these slopes must satisfy the balancing condition:
\[ d_v = \sum_{v \in e} c(m_{\vec{e}}) + \sum_{v \in i} (\upalpha_{ij})_{j=1}^k. \]
Here each edge $e$ is oriented to point away from $v$, the map $c$ is the global framing \eqref{eqn: global framing}, and the second sum is over markings supported at $v$. 
\end{itemize}
\end{definition}

To each tropical type $\Theta$ we associate a rational polyhedral cone $\uptau_\Theta$ embedded in the orthant
\[
\prod_{e\in E(\Gamma)}\mathbb R_{\geqslant 0} \times \prod_{v\in V(\Gamma)} \upsigma_v
\]
where the coordinates in the first factor assign a length $\ell_e \in \Rplus$ to each edge, and the coordinates in the second factor assign a position $\ftrop(v) \in \upsigma_v$ to each vertex. The cone $\uptau_\Theta$ is the locus where these quantities arise from a piecewise linear map from a metrised graph, with slopes as specified by $\Theta$.

\begin{definition}[{\!\cite[Proposition~2.32]{PuncturedMaps}}] The \textbf{tropical moduli cone} associated to $\Theta$ is the rational polyhedral cone
\[
\uptau_\Theta \subseteq \prod_{e\in E(\Gamma)}\mathbb R_{\geqslant 0} \times \prod_{v\in V(\Gamma)} \upsigma_v
\]
defined by the equations
\[ \ftrop(v_2) = \ftrop(v_1) + m_{\vec{e}}\, \ell_e \]
for every oriented edge $\vec{e} \in \vec{E}(\Gamma)$ starting at $v_1$ and ending at $v_2$.
\end{definition}

Specialisations of tropical types, including automorphisms, induce face maps of tropical moduli cones. These are used to glue the cones together into a cone stack \cite{CavalieriChanUlirschWise}:

\begin{definition} The moduli space of \textbf{tropical punctured maps} is the colimit
\[ \Tsf \colonequals \TropMap = \varinjlim_\Theta \uptau_\Theta, \]
over all cones $\uptau_\Theta$ associated to tropical types $\Theta$ with numerical data $\Lambda$, connected via face maps corresponding to specialisations of tropical types.
\end{definition}

As in \cite{CavalieriChanUlirschWise} this represents a functor on the category of rational polyhedral cones, parametrising tropical maps to $\Rplus^k$ with numerical data $\Lambda$.

Note that parallel to $\Punct_\Lambda(\Acal^k|\partial \Acal^k)$, the cone stack $\Tsf$ is not of finite type. In applications, one can restrict to a union of finitely many cones, by selecting only those tropical types that arise from geometric points in $\Punct_\Lambda(X|D)$. Examples are given in \Cref{sec: Aluffi}.

\begin{notation} To ease readability, we will use the same symbol $\uptau$ to denote both a tropical type of punctured map to $\Rplus^k$, and the associated cone in $\Tsf$. Since the former index the latter, there is no ambiguity. \end{notation}

\begin{remark} \label{rmk: proxy for Artin fan} The cone stack $\Tsf$ is the tropicalisation of $\Punct_\Lambda(\Acal^k|\partial \Acal^k)$ (see for instance \cite[Theorem~1.1]{KHNSZ} for a proximate argument). It is related to, but coarser than, the tropicalisation of $\Punct_\Lambda(\Acal(\Sigma)|\partial \Acal(\Sigma))$. We essentially never contemplate the tropicalisation of $\Punct_\Lambda(X|D)$ -- even the question of when a putative stratum of this space is nonempty is impractical; it is also unnecessary. We have strict maps
\[ \Punct_\Lambda(X|D) \xrightarrow{\upeta} \Punct_\Lambda(\Acal(\Sigma)|\partial \Acal(\Sigma)) \xrightarrow{\uptheta} \Punct_\Lambda(\Acal^k|\partial \Acal^k).\]
 The natural stratifications of the latter two spaces, indexed by cones in a cone stack, give rise to stratifications on the space $\Punct_\Lambda(X|D)$ by pullback. These two stratifications differ slightly. The morphism $\uptheta$ is usually the inclusion of the complement of a union of closed strata; however it may map multiple strata onto one, if an intersection of divisor components $D_j$ is disconnected. The cone stack $\Tsf = \Trop_\Lambda(\Rplus^k)$ contains sufficient information for our purposes, and will serve as a proxy for the tropicalisation of the geometric moduli space.
\end{remark}

\begin{remark} Given a puncture $p \in P(\Gamma)$ with supporting vertex $v \in V(\Gamma)$, we do not impose that $\upsigma_v$ contains all the rays corresponding to the divisors $j \in [k]$ with $\upalpha_{pj} < 0$. However, tropical types which do not satisfy this condition will correspond to strata disjoint from the puncturing substack $\Vcal(\Tsf) \hookrightarrow \Acal(\Tsf)$, and hence disjoint from the image of $\Punct_\Lambda(\Acal^k | \partial \Acal^k) \to \Acal(\Tsf)$.\end{remark}

\subsection{Puncturing substack} \label{sec: puncturing substack}
Punctures are depicted tropically by appending a leg of slope $(\upalpha_{ij})_{j=1}^k$ to the supporting vertex:
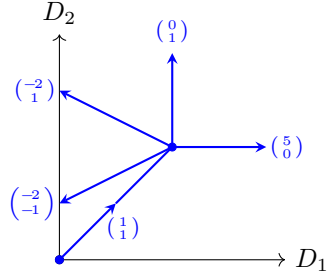
\[
\begin{tikzpicture}

\draw[fill=black,black] (0,0) circle[radius=1.5pt];
\draw[black,->] (0,0) -- (3,0);
\draw[black] (3,0) node[right]{\small$D_1$};
\draw[black,->] (0,0) -- (0,3);
\draw[black] (0,3) node[above]{\small$D_2$};

\draw[fill=blue,blue] (0,0) circle[radius=1.5pt];
\draw[-stealth,thick,blue] (0,0) -- (0.75,0.75);
\draw[blue] (0.85,0.7) node[below]{\tiny$\begin{psmallmatrix} 1 \\ 1 \end{psmallmatrix}$};
\draw[thick,blue] (0.75,0.75) -- (1.5,1.5);
\draw[fill=blue,blue] (1.5,1.5) circle[radius=1.5pt];

\draw[-stealth,thick,blue] (1.5,1.5) -- (0,0.75);
\draw[blue] (0.1,0.75) node[left]{\tiny$\begin{psmallmatrix}\!-\!2 \\ \!-\!1 \end{psmallmatrix}$};
	
\draw[-stealth,thick,blue] (1.5,1.5) -- (0,2.25);
\draw[blue] (0.1,2.25) node[left]{	\tiny$\begin{psmallmatrix}\!-\!2 \\ 1 \end{psmallmatrix}$};

\draw[-stealth,thick,blue] (1.5,1.5) -- (2.75,1.5);
\draw[blue] (2.65,1.5) node[right]{\tiny$\begin{psmallmatrix} 5 \\ 0 \end{psmallmatrix}$};

\draw[-stealth,thick,blue] (1.5,1.5) -- (1.5,2.75);
\draw[blue] (1.5,2.75) node[above]{\tiny$\begin{psmallmatrix} 0 \\ 1 \end{psmallmatrix}$};
\end{tikzpicture}
\]
The puncturing offset is the position of the vertex supporting the puncture, measured with respect to those directions $j$ for which $\upalpha_{ij} < 0$, see \Cref{def: puncturing offset} below. In \cite[Section~2.2]{PuncturedMaps} the puncturing offset is required to be positive in all directions. We do not impose this condition on $\Tsf$. The special locus inside $\Tsf$ where this condition holds will play a central role in the theory.

\begin{definition} \label{def: puncturing rank} Given a marking $i \in [n]$ the \textbf{puncturing rank} $k_i$ is the number of divisor components $D_j$ for which $\upalpha_{ij} < 0$. A marking is a puncture if and only if it has positive puncturing rank. The \textbf{total puncturing rank} is the sum:
\[ k_P \colonequals \sum_{i=1}^n k_i.\]
For $i \in [n]$ we have $0 \leqslant k_i \leqslant k$, and hence we have $0 \leqslant k_P \leqslant nk$.
\end{definition}

\begin{definition} \label{def: puncturing offset}
 For each puncture $p \in P$ there is a morphism of cone stacks recording the \textbf{puncturing offset}:
\[ \ftrop(p) \colon \Tsf \to \Rplus^{k_p}. \]
This is given by $\ftrop(p) \colonequals \ptrop(\ftrop(v))$, where $v \in V(\Gamma)$ is the vertex supporting the puncture $p$, the position $\ftrop(v) \in \Rplus^k$ is the image of this vertex under the tropical map, and $\ptrop$ is the projection
\[ \Rplus^k \to \Rplus^{k_p} \]
onto the face spanned by the rays corresponding to those $D_j$ with $\upalpha_{ij}<0$.
\end{definition}

Note that $\ftrop(p)$ records a length in the \emph{target} of the tropical map, not the source, and does not depend on the multiplicities of the puncture. Heuristically, requiring the puncturing offset to be positive cuts out a union of face interiors $\mathsf{V} \subseteq \Tsf$, obtained as a fibre product:
\begin{equation} \label{eqn: diagram tropical puncturing offsets positive}
\begin{tikzcd}
\mathsf{V} \ar[r,hook] \ar[d] \ar[rd,phantom,"\square" right] & \Tsf \ar[d] \\
\RR_{>0}^{k_P} \ar[r,hook] & \RR_{\geqslant 0}^{k_P}.
\end{tikzcd}
\end{equation}
Inspired by this heuristic tropical diagram, we make the following:
\begin{definition}
The \textbf{puncturing substack} is the fibre product of algebraic stacks
\begin{equation} \label{eqn: diagram puncturing substack as fibre product}
\begin{tikzcd}
\Vcal(\Tsf) \ar[r,hook] \ar[d] \ar[rd,phantom,"\square"] & \Acal(\Tsf) \ar[d] \\
\Bcal\Gm^{k_P} \ar[r,hook,"\upiota"] & \Acal^{k_P}
\end{tikzcd}
\end{equation}
where the morphism $\Acal(\Tsf) \to \Acal^{k_P}$ is induced by the puncturing offset morphisms $\ftrop(p)$ for all $p \in P$.
\end{definition}

We now reformulate the main results on punctured logarithmic maps in terms of the puncturing substack.

\begin{lemma} The puncturing substack $\Vcal(\Tsf) \hookrightarrow \Acal(\Tsf)$ is the vanishing substack of the puncturing logarithmic ideal, as defined in \cite[Definition~2.49]{PuncturedMaps}.\end{lemma}
\begin{proof} This follows directly from \cite[Proposition~2.57]{PuncturedMaps} and its proof. The radical appearing in the statement of that result accounts for the difference between source and target lengths. \end{proof}

The stack of punctured curves is a closed substack of the stack of logarithmic curves, which is itself obtained from the stack of prestable curves by base changing over its Artin fan \cite[Proposition~3.3]{PuncturedMaps}. Maps from a fixed punctured curve to $(\Acal^k|\partial \Acal^k)$ are given by a collection of piecewise linear functions on the tropicalisation \cite[Proposition~2.10]{ACGSDecomposition}. We conclude (see also \cite[Theorem~3.25]{PuncturedMaps}) that the following square is cartesian
\begin{equation} \label{eqn: cartesian square punctured maps over prestable curves}
\begin{tikzcd}
\Punct_\Lambda(\Acal^k | \partial \Acal^k) \ar[r] \ar[d,"\upnu"] \ar[rd,phantom,"\square"] & \Mfrak \ar[d] \\
\Vcal(\Tsf) \ar[r] & \Acal(\MTrop),
\end{tikzcd}	
\end{equation}
where $\Mfrak$ is the stack of prestable curves and $\MTrop$ is the stack of tropical curves \cite{ACP,CavalieriChanUlirschWise}. The right vertical arrow is smooth and the bottom horizontal arrow is the composite of the inclusion $\Vcal(\Tsf) \hookrightarrow \Acal(\Tsf)$ with the map $\Acal(\Tsf) \to \Acal(\MTrop)$ forgetting the vertex positions.

We now summarise \cite[Section~4]{PuncturedMaps}. There are morphisms:
\[ \Punct_\Lambda(X|D) \xrightarrow{\upmu} \Punct_\Lambda(\Acal^k|\partial \Acal^k) \xrightarrow{\upnu} \Vcal(\Tsf).\]
The authors prove that $\upmu$ is virtually smooth, and it follows from \eqref{eqn: cartesian square punctured maps over prestable curves} that $\upnu$ is smooth. However, the following composite is typically not smooth:
\[ \Punct_\Lambda(\Acal^k|\partial \Acal^k) \to \Vcal(\Tsf) \hookrightarrow \Acal(\Tsf).\]
While $\Acal(\Tsf)$ is logarithmically smooth and hence irreducible, the substack $\Vcal(\Tsf)$ typically consists of multiple irreducible components of various dimensions. This is visible tropically: the irreducible components correspond to minimal tropical types of punctured maps to $\Rplus^k$, i.e. types that cannot be further specialised without setting some coordinates of the puncturing offset to zero. The dimension of each irreducible component is the negative of the dimension of the corresponding cone.

\subsection{Refined punctured theory} \label{sec: refined punctured class} The base of the obstruction theory $\Vcal(\Tsf)$ is typically not equidimensional, and so does not carry a preferred Chow class. In \cite{PuncturedMaps} the authors define punctured invariants by pulling back \emph{any} Chow class from $\Vcal(\Tsf)$. Typically this is taken to be the class of an irreducible component, or more generally a stratum.

Our reformulation via \eqref{eqn: diagram puncturing substack as fibre product} expresses $\Vcal(\Tsf) \hookrightarrow \Acal(\Tsf)$ as the pullback of a regular embedding (we emphasise that $\Vcal(\Tsf) \hookrightarrow \Acal(\Tsf)$ is rarely itself a regular embedding). This allows us to canonically define a virtual class on $\Vcal(\Tsf)$ of pure dimension. 

\begin{definition} \label{def: refined virtual class}
The \textbf{refined virtual class} on $\Vcal(\Tsf)$ is the Gysin pullback:
\[
[\Vcal(\Tsf)]^{\refined} \colonequals \upiota^![\mathcal A(\Tsf)]. 
\]
The \textbf{refined virtual class} on $\Punct_\Lambda(\Acal^k|\partial \Acal^k)$ is the smooth pullback
\[ [\Punct_\Lambda(\Acal^k|\partial \Acal^k)]^{\refined} \colonequals \upnu^\star [\Vcal(\Tsf)]^{\refined} \]
and the \textbf{refined virtual class} on $\Punct_\Lambda(X|D)$ is the virtual pullback
\[ [\Punct_\Lambda(X|D)]^{\refined} \colonequals \upmu^![\Punct_\Lambda(\Acal^k|\partial \Acal^k)]^{\refined} \]
with respect to the perfect obstruction theory for $\upmu$ constructed in \cite[Section~4]{PuncturedMaps}.
\end{definition}

Unlike the spaces $\Vcal(\Tsf)$ or $\Punct_\Lambda(\Acal^k|\partial \Acal^k)$, the refined virtual class has pure dimension, given by:
\[ (\dim X -3)(1-g) - (K_X+D)\cdot \upbeta + n - k_P. \]
This new virtual class isolates a special sector of the punctured Gromov--Witten theory. For each marked point $p_i$ one can discern a priori, from just the contact orders, a stratum of $(X|D)$ to which $p_i$ must map: the intersection of all boundary divisors with which $p_i$ has nonzero contact order. We introduce this formally in the next section, and denote it $D_{J(i)}$. These strata, taken over all marked points, give the natural space of insertions for the refined theory:

\begin{definition} The \textbf{refined sector} of the punctured Gromov--Witten theory of $(X|D)$ is the system of integrals against the refined virtual class
\[ \langle \upgamma_1 \uppsi_1^{k_1},\ldots,\upgamma_n \uppsi_n^{k_n} \rangle^{\refined}_{\Lambda} \colonequals \int_{[\Punct_\Lambda(X|D)]^{\refined}} \prod_{i=1}^n \ev_i^\star \upgamma_i \cdot \uppsi_i^{k_i} \]
where $\upgamma_i \in A^\star(D_{J(i)})$ is a class on the closed stratum of $X$ to which the marking $p_i$ is constrained by the tangency conditions (see \Cref{sec: main theorem}).
\end{definition}

\subsection{Strata formulae for the refined virtual class and examples} \label{sec: Aluffi} The refined virtual class is defined via intersection theory on the Artin fan, and hence (after possibly passing to a resolution) admits a formula in terms of strata classes decorated by Chow operators. Further details on these calculations are given in~\cite{MolchoRanganathan}.

The purpose of this section is to demystify the refined virtual class. We record two special cases (Sections~\ref{sec: rank 1 class} and \ref{sec: regularly embedded class}), and then describe the general case (\Cref{sec: general case class}). These formulae take advantage of the computability of Segre classes of monomial schemes, following work of Aluffi~\cite{AluffiPolytopes}. Several explicit examples are provided.

\subsubsection{Puncturing rank one} \label{sec: rank 1 class} Recall the total puncturing rank from \Cref{def: puncturing rank}. We begin with the important special case
\[ k_P=1.\]
These are precisely the invariants which appear in the intrinsic mirror construction \cite{GS19}. Note that the logarithmic rank $k$ of the target may be arbitrary.

In this setting, the closed embedding $\Bcal \Gm \hookrightarrow \Acal$ appearing in \eqref{eqn: diagram puncturing substack as fibre product} is a Cartier divisor. It follows that $\Vcal(\Tsf) \hookrightarrow \Acal(\Tsf)$ is also a Cartier divisor, and
\begin{equation} [\Vcal(\Tsf)]^{\refined} = [\Vcal(\Tsf)].\end{equation}
Thus for $k_P=1$ the refined punctured theory coincides with the standard punctured theory. The fundamental class decomposes as a sum over minimal tropical types, indexing the irreducible components of $\Vcal(\Tsf)$. Since $\Vcal(\Tsf)$ has pure dimension $-\!1$, each minimal tropical type has dimension~$1$. 

\begin{example} \label{example: puncturing rank one} Consider the geometric target $(X|D)=(\PP^r|H)$ with $r \geqslant 2$ and where $H\subseteq\PP^r$ is a hyperplane. Take numerical data $\upbeta=d=1$ and $n=2$ with tangency orders:
\[ \upalpha_1=2, \quad \upalpha_2=-1.\]
As noted previously, fixing the geometric target $(X|D)$ and numerical data $\Lambda$ allows us to replace $\Tsf$ by a substack comprising finitely many cones. This replaces $\Acal(\Tsf)$ by an open substack containing the image of $\Punct_\Lambda(\PP^r|H)$. This does not affect the calculation of the refined virtual class. On this open substack we identify two minimal tropical types
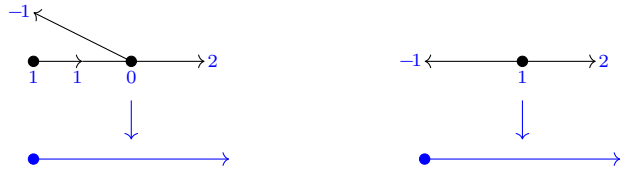
\[
\begin{tikzpicture}[scale=1.3]


\draw[fill=black] (-2,-1) circle[radius=1.5pt];
\draw[blue] (-2,-1) node[below]{\tiny$1$};

\draw[->] (-2,-1) -- (-1.5,-1);
\draw (-1.5,-1) -- (-1,-1);
\draw[blue] (-1.55,-1) node[below]{\tiny$1$};

\draw[fill=black] (-1,-1) circle[radius=1.5pt];
\draw[blue] (-1,-1) node[below]{\tiny$0$};

\draw[->] (-1,-1) -- (-0.25,-1);
\draw[blue] (-0.325,-1) node[right]{\tiny$2$};

\draw[->] (-1,-1) -- (-2,-0.5);
\draw[blue] (-1.925,-0.5) node[left]{\tiny$-\!1$};

\draw[fill=blue,blue] (-2,-2) circle[radius=1.5pt];
\draw[blue,->] (-2,-2) -- (0,-2);

\draw[blue,->] (-1,-1.4) -- (-1,-1.8);



\draw[fill=black] (3,-1) circle[radius=1.5pt];
\draw[blue] (3,-1) node[below]{\tiny$1$};

\draw[->] (3,-1) -- (3.75,-1);
\draw[blue] (3.675,-1) node[right]{\tiny$2$};

\draw[->] (3,-1) -- (2,-1);
\draw[blue] (2.075,-1) node[left]{\tiny$-\!1$};

\draw[fill=blue,blue] (2,-2) circle[radius=1.5pt];
\draw[blue,->] (2,-2) -- (4,-2);

\draw[blue,->] (3,-1.4) -- (3,-1.8);

\draw (3,-2.25);
\end{tikzpicture}
\]
where each vertex is labelled with its degree. Each tropical moduli cone is isomorphic to $\Rplus$, and the puncturing substack corresponds to $\RR_{>0}$. These index the irreducible components of $\Vcal(\Tsf)$. 

Since $(\PP^r|H)$ is convex, the logarithmic structure morphism $\Punct_\Lambda(\PP^r|H) \to \Vcal(\Tsf)$ is smooth, and so these irreducible components pull back to the irreducible components of the punctured moduli space. In the first case, the general point parametrises a line in $\PP^r$ transverse to $H$, with a $2$-marked rational tail attached at its intersection point with $H$. In the second case, the general point parametrises a line in $H$ with a choice of two distinct points on it. Both components have dimension $2r-2$. The refined virtual class is simply the sum of the fundamental classes of the two components.
\end{example}

\subsubsection{Regularly embedded puncturing substack} \label{sec: regularly embedded class} Extending the previous case, let $k_P \in \N$ be arbitrary and consider the fibre square:
\[
\begin{tikzcd}
\Vcal(\Tsf) \ar[r,hook,"\upkappa"] \ar[d,"\uprho"] \ar[rd,phantom,"\square"] & \Acal(\Tsf) \ar[d] \\
\Bcal\Gm^{k_P} \ar[r,hook,"\upiota"] & \Acal^{k_P}.
\end{tikzcd}
\]
Now assume that $\upkappa$ is a regular embedding (when $\Tsf$ is smooth this occurs if and only if all moduli cones associated to minimal tropical types have the same dimension). In this case $\Vcal(\Tsf)$ is pure-dimensional and hence carries a fundamental class. Let $e \colonequals \operatorname{codim} \upiota - \operatorname{codim} \upkappa$. The excess intersection formula \cite[Theorem~6.3]{FultonBig} then gives:
\begin{equation} \label{eqn: refined class excess intersection} [\Vcal(\Tsf)]^{\refined} = c_e(\uprho^\star N_{\upiota}/N_\upkappa) \cap [\Vcal(\Tsf)]. \end{equation}

\begin{example} \label{example: higher rank regular embedding} Consider the geometric target $(X|D)=(\PP^2|L_1+L_2)$ where $L_1,L_2$ are distinct lines. Take the numerical data $\upbeta=d=1$ and $n=2$ with tangency orders:
\[ \upalpha_1 = (2,2), \quad \upalpha_2=(-1,-1).\]
There is a single puncture of puncturing rank $k_2=2$ and hence the total puncturing rank is $k_P=2$. We replace $\Tsf$ by a subcomplex, corresponding to an open substack of $\Acal(\Tsf)$ through which $\Punct_\Lambda(X|D)$ factors. This is obtained by enumerating the tropical types arising from punctured maps with the given numerical data. The subcomplex $\Tsf$ is illustrated below, with each cone labelled by its corresponding tropical type and a symbol $Z_i$ which denotes the corresponding boundary divisor in the Artin fan (and its Chow class):
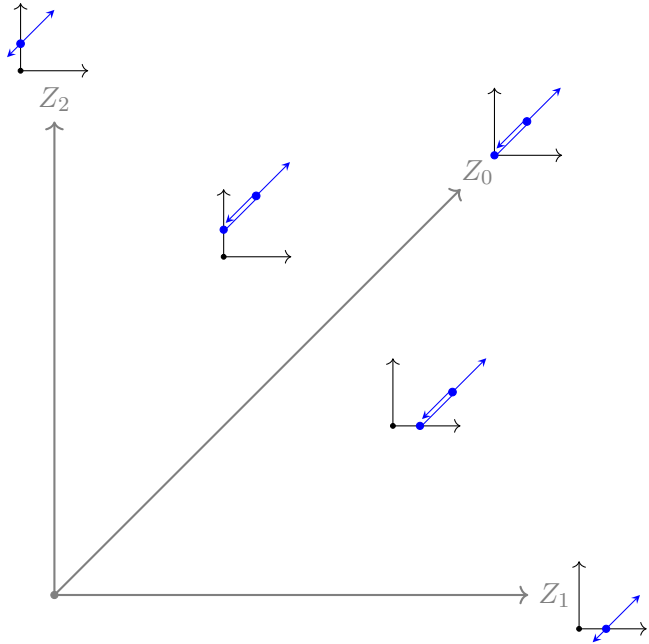
\[
\begin{tikzpicture}[scale=0.9]

\draw[fill=gray,gray] (0,0) circle[radius=1.5pt];
\draw[gray,thick,->] (0,0) -- (7,0);
\draw[gray] (7,0) node[right]{$Z_1$};
\draw[gray,thick,->] (0,0) -- (6,6);
\draw[gray] (6.25,6.25) node{$Z_0$};
\draw[gray,thick,->] (0,0) -- (0,7);
\draw[gray] (0,7) node[above]{$Z_2$};

\draw[black,fill=black] (6.5,6.5) circle[radius=1pt];
\draw[black,->] (6.5,6.5) -- (7.5,6.5);
\draw[black,->] (6.5,6.5) -- (6.5,7.5);
\draw[blue,fill=blue] (6.5,6.5) circle[radius=1.5pt];
\draw[blue] (6.525,6.5) -- (7.025,7);
\draw[blue,fill=blue] (6.98,7) circle[radius=1.6pt];
\draw[blue,-stealth] (6.93,7) -- (6.53,6.6);
\draw[blue,-stealth] (6.98,7) -- (7.48,7.5);

\draw[black,fill=black] (7.75,-0.5) circle[radius=1pt];
\draw[black,->] (7.75,-0.5) -- (8.75,-0.5);
\draw[black,->] (7.75,-0.5) -- (7.75,0.5);
\draw[blue,fill=blue] (8.15,-0.5) circle[radius=1.5pt];
\draw[blue,-stealth] (8.15,-0.5) -- (7.95,-0.7);
\draw[blue,-stealth] (8.15,-0.5) -- (8.65,0);

\draw[black,fill=black] (-0.5,7.75) circle[radius=1pt];
\draw[black,->] (-0.5,7.75) -- (0.5,7.75);
\draw[black,->] (-0.5,7.75) -- (-0.5,8.75);
\draw[blue,fill=blue] (-0.5,8.15) circle[radius=1.6pt];
\draw[blue,-stealth] (-0.5,8.15) -- (-0.7,7.95);
\draw[blue,-stealth] (-0.5,8.15) -- (0,8.65);	

\draw[black,fill=black] (5,2.5) circle[radius=1pt];
\draw[black,->] (5,2.5) -- (6,2.5);
\draw[black,->] (5,2.5) -- (5,3.5);
\draw[blue,fill=blue] (5.4,2.5) circle[radius=1.5pt];
\draw[blue] (5.425,2.5) -- (5.925,3);
\draw[blue,fill=blue] (5.88,3) circle[radius=1.6pt];
\draw[blue,-stealth] (5.83,3) -- (5.43,2.6);
\draw[blue,-stealth] (5.88,3) -- (6.38,3.5);

\draw[black,fill=black] (2.5,5) circle[radius=1pt];
\draw[black,->] (2.5,5) -- (3.5,5);
\draw[black,->] (2.5,5) -- (2.5,6);
\draw[blue,fill=blue] (2.5,5.4) circle[radius=1.5pt];
\draw[blue] (2.525,5.4) -- (3.025,5.9);
\draw[blue,fill=blue] (2.98,5.9) circle[radius=1.6pt];
\draw[blue,-stealth] (2.93,5.9) -- (2.53,5.5);
\draw[blue,-stealth] (2.98,5.9) -- (3.48,6.4);
	
\end{tikzpicture}
\]
The tropical moduli consists of the position of the vertex $v_0$ which supports the markings. When $v_0$ belongs to the interior of $\Rplus^2$ we assign it degree $0$, and bubble an edge of slope $(-1,-1)$ which ends at a vertex of degree $1$ belonging to the boundary of $\Rplus^2$. When $v_0$ belongs to the boundary, we assign it degree $1$ and do not bubble any edges. We see immediately that
\[ \Vcal(\Tsf) = Z_0.\]
The puncturing offset map $\Tsf \to \Rplus^2$ records the position of the vertex $v_0$. Examining the slope along the rays, we obtain
\[ \uprho^\star N_\upiota = \OO(Z_0+Z_1) \oplus \OO(Z_0+Z_2).\]
From this we compute the total Chern class of the excess bundle:
\begin{align*} c(\uprho^\star N_\upiota/N_\upkappa) & = \dfrac{(1+Z_0+Z_1)(1+Z_0+Z_2)}{1+Z_0}	 = 1 + (Z_0+Z_1+Z_2) + \cdots
\end{align*}
We have $e= \operatorname{codim} \upiota - \operatorname{codim} \upkappa = 2-1 = 1$ and from \eqref{eqn: refined class excess intersection} we obtain:
\begin{align*} [\Vcal(\Tsf)]^{\refined} & = c_1(\uprho^\star N_\upiota/N_\upkappa) \cap [\Vcal(\Tsf)] \\
& = (Z_0+Z_1+Z_2)  Z_0 \\
& = Z_0^2 + Z_0Z_1 + Z_0Z_2 \\
& = (-\uppsi_{v_0 \in e}-\uppsi_{v_1 \in e}) Z_0 + Z_0Z_1 + Z_0Z_2
\end{align*}
where the cotangent line classes are associated to the flags of the source graph. We have thus expressed the refined virtual class as a linear combination of tautological strata classes.
\end{example}

\subsubsection{General case} \label{sec: general case class} In general the puncturing substack $\Vcal(\Tsf)$ is not pure-dimensional and the previous arguments do not apply. In this setting, the refined virtual class is given by the following formula \cite[Proposition~6.1(a)]{FultonBig}:
\begin{equation} \label{eqn: refined class general} [\Vcal(\Tsf)]^{\refined} = \{ c(\uprho^\star N_\upiota) \cap s(\Vcal(\Tsf),\Acal(\Tsf)) \}_{-k_P}.\end{equation}
The Chern class is easily described in terms of piecewise polynomials (see Examples~\ref{example: higher rank regular embedding} and \ref{ex: resolution calculation}). The Segre class is computed via either of the following two techniques.
\begin{enumerate}
\item \textbf{Aluffi approach.} The puncturing substack $\Vcal(\Tsf) \hookrightarrow \Acal(\Tsf)$ is a union of boundary strata. Precisely, it is cut out by monomials in the equations defining the boundary divisors. Then \cite[Theorem~1.1]{AluffiPolytopes} expresses its Segre class as a polynomial in strata classes. 
\item \textbf{Resolution rubric.} Order the irreducible components of $\Vcal(\Tsf)$ and perform the corresponding iterated blowup of $\Acal(\Tsf)$. The total transform of $\Vcal(\Tsf)$ is then a Cartier divisor, and pushing forward the inverse of its total Chern class produces $s(\Vcal(\Tsf),\Acal(\Tsf))$ by birational invariance of Segre classes \cite[Corollary~4.2.1]{FultonBig}. 
\end{enumerate}

\begin{remark} Both techniques require $\Tsf$ smooth, which is typically not the case for tropical maps \cite{Universality}. To begin the calculation, we replace $\Tsf$ with a resolution $\Tsf^\prime$ and base change the puncturing substack:
\[ 
\begin{tikzcd}
\uprho^{-1}\Vcal(\Tsf) \ar[r,hook] \ar[d,"\uprho"] \ar[rd,phantom,"\square"] & \Acal(\Tsf^\prime) \ar[d,"\uprho"] \\
\Vcal(\Tsf) \ar[r,hook] & \Acal(\Tsf).	
\end{tikzcd}
\]
We can perform the computation on the resolution, since \cite[Proposition~4.2]{FultonBig} gives:
\[ \uprho_\star s(\uprho^{-1}\Vcal(\Tsf),\Acal(\Tsf^\prime)) = s(\Vcal({\Tsf}),\Acal({\Tsf})).\]
Alternatively, the singular locus can be excised if it has smaller dimension than the refined virtual class, as in the following example.
\end{remark}

\begin{example}	\label{ex: resolution calculation} Let $X$ be the blowup of the Hirzebruch surface $\FF_1$ at a torus-fixed point on the $(-1)$-curve. Let $E_1 \subseteq X$ be the strict transform of the $(-1)$-curve, and $E_2 \subseteq X$ the exceptional divisor. These generate a face of the effective cone. We have
\[ E_1^2 = -2, \quad E_2^2 = -1, \quad E_1 E_2 = 1.\]
Consider the pair $(X|E_1)$. Take $\upbeta=2E_1+2E_2$ so that $E_1 \cdot \upbeta=-2$, and take $n=2$ with $\upalpha_1=\upalpha_2=-1$. This defines numerical data $\Lambda$ for a moduli space of stable punctured maps to $(X|E_1)$.

Fix a tropical type arising from a stable punctured map to $(X|E_1)$, and choose a vertex $v \in V(\Gamma)$. By balancing and the rigidity of $E_1$, we see that if $\upsigma_v=0$ then $\upbeta_v \in \N E_2$, and if $\upsigma_v=\Rplus$ then $\upbeta_v \in \N E_1$. From this we deduce that $\Vcal(\Tsf)$ has three irreducible components, indexed by the following types:
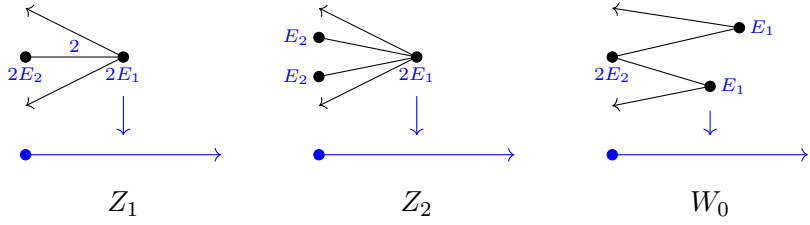
\[
\begin{tikzpicture}[scale=1.3]


\draw[fill=black] (-2,-1) circle[radius=1.5pt];
\draw[blue] (-2,-1) node[below]{\tiny$2E_2$};

\draw (-2,-1) -- (-1,-1);
\draw[blue] (-1.5,-1.05) node[above]{\tiny$2$};

\draw[fill=black] (-1,-1) circle[radius=1.5pt];
\draw[blue] (-1,-1) node[below]{\tiny$2E_1$};

\draw[->] (-1,-1) -- (-2,-0.5);
\draw[->] (-1,-1) -- (-2,-1.5);

\draw[fill=blue,blue] (-2,-2) circle[radius=1.5pt];
\draw[blue,->] (-2,-2) -- (0,-2);

\draw[blue,->] (-1,-1.4) -- (-1,-1.8);

\draw (-1,-2.5) node{$Z_1$};



\draw[fill=black] (1,-0.8) circle[radius=1.5pt];
\draw[blue] (1,-0.8) node[left]{\tiny$E_2$};

\draw[fill=black] (1,-1.2) circle[radius=1.5pt];
\draw[blue] (1,-1.2) node[left]{\tiny$E_2$};

\draw (1,-0.8) -- (2,-1);
\draw (1,-1.2) -- (2,-1);

\draw[fill=black] (2,-1) circle[radius=1.5pt];
\draw[blue] (2,-1) node[below]{\tiny$2E_1$};

\draw[->] (2,-1) -- (1,-0.5);
\draw[->] (2,-1) -- (1,-1.5);

\draw[fill=blue,blue] (1,-2) circle[radius=1.5pt];
\draw[blue,->] (1,-2) -- (3,-2);

\draw[blue,->] (2,-1.4) -- (2,-1.8);

\draw (2,-2.5) node{$Z_2$};



\draw[fill=black] (4,-1) circle[radius=1.5pt];
\draw[blue] (4,-1) node[below]{\tiny$2E_2$};

\draw[fill=black] (5.3,-0.7) circle[radius=1.5pt];
\draw[blue] (5.3,-0.7) node[right]{\tiny$E_1$};

\draw[fill=black] (5,-1.3) circle[radius=1.5pt];
\draw[blue] (5,-1.3) node[right]{\tiny$E_1$};

\draw (4,-1) -- (5.3,-0.7);
\draw (4,-1) -- (5,-1.3);

\draw[->] (5.3,-0.7) -- (4,-0.5);
\draw[->] (5,-1.3) -- (4,-1.5);

\draw[fill=blue,blue] (4,-2) circle[radius=1.5pt];
\draw[blue,->] (4,-2) -- (6,-2);

\draw[blue,->] (5,-1.55) -- (5,-1.8);

\draw (5,-2.5) node{$W_0$};

\draw (3,-2.25);
\end{tikzpicture}
\]
An undecorated edge has weight $1$, while each vertex is decorated with the corresponding curve class. The symbols $Z_i$ and $W_0$ denote the corresponding closed substacks of the Artin fan (and their Chow classes). We have $\dim Z_1=\dim Z_2 = -1$ and $\dim W_0 = -2$. The pairwise intersections correspond to the following types:
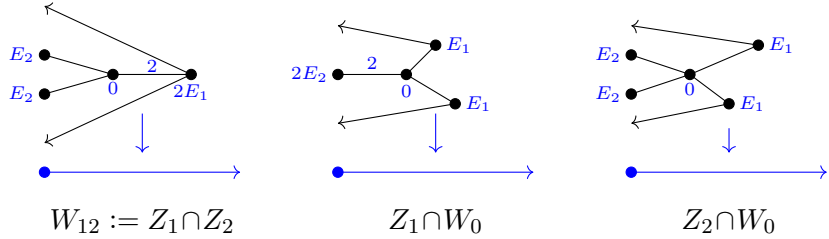
\[
\begin{tikzpicture}[scale=1.3]


\draw[fill=black] (-2,-1.2) circle[radius=1.5pt];
\draw[blue] (-2,-1.2) node[left]{\tiny$E_2$};

\draw[fill=black] (-2,-0.8) circle[radius=1.5pt];
\draw[blue] (-2,-0.8) node[left]{\tiny$E_2$};

\draw (-2,-1.2) -- (-1.3,-1);
\draw (-2,-0.8) -- (-1.3,-1);

\draw[fill=black] (-1.3,-1) circle[radius=1.5pt];
\draw[blue] (-1.3,-0.975) node[below]{\tiny$0$};

\draw[fill=black] (-0.5,-1) circle[radius=1.5pt];
\draw[blue] (-0.5,-1) node[below]{\tiny$2E_1$};

\draw[fill=black] (-1.3,-1) -- (-0.5,-1);
\draw[blue] (-0.9,-1.08) node[above]{\tiny$2$};

\draw[->] (-0.5,-1) -- (-2,-0.3);
\draw[->] (-0.5,-1) -- (-2,-1.7);

\draw[fill=blue,blue] (-2,-2) circle[radius=1.5pt];
\draw[blue,->] (-2,-2) -- (0,-2);

\draw[blue,->] (-1,-1.4) -- (-1,-1.8);

\draw (-1,-2.5) node{$W_{12} \colonequals Z_1\!\cap\!Z_2$};



\draw[fill=black] (1,-1) circle[radius=1.5pt];
\draw[blue] (1,-1) node[left]{\tiny$2E_2$};

\draw (1,-1) -- (1.7,-1);
\draw[blue] (1.35,-1.05) node[above]{\tiny$2$};

\draw[fill=black] (1.7,-1) circle[radius=1.5pt];
\draw[blue] (1.7,-1) node[below]{\tiny$0$};

\draw (1.7,-1) -- (2,-0.7);
\draw[blue] (2,-0.7) node[right]{\tiny$E_1$};

\draw (1.7,-1) -- (2.2,-1.3);
\draw[blue] (2.2,-1.3) node[right]{\tiny$E_1$};

\draw[fill=black] (2,-0.7) circle[radius=1.5pt];
\draw[fill=black] (2.2,-1.3) circle[radius=1.5pt];

\draw[->] (2,-0.7) -- (1,-0.5);
\draw[->] (2.2,-1.3) -- (1,-1.5);

\draw[fill=blue,blue] (1,-2) circle[radius=1.5pt];
\draw[blue,->] (1,-2) -- (3,-2);

\draw[blue,->] (2,-1.4) -- (2,-1.8);

\draw (2,-2.5) node{$Z_1\! \cap \! W_0$};



\draw[fill=black] (4,-1.2) circle[radius=1.5pt];
\draw[blue] (4,-1.2) node[left]{\tiny$E_2$};

\draw[fill=black] (4,-0.8) circle[radius=1.5pt];
\draw[blue] (4,-0.8) node[left]{\tiny$E_2$};

\draw[fill=black] (4.6,-1) circle[radius=1.5pt];
\draw[blue] (4.6,-0.98) node[below]{\tiny$0$};

\draw[fill=black] (5.3,-0.7) circle[radius=1.5pt];
\draw[blue] (5.3,-0.7) node[right]{\tiny$E_1$};

\draw[fill=black] (5,-1.3) circle[radius=1.5pt];
\draw[blue] (5,-1.3) node[right]{\tiny$E_1$};

\draw (4,-1.2) -- (4.6,-1);
\draw (4,-0.8) -- (4.6,-1);

\draw (4.6,-1) -- (5.3,-0.7);
\draw (4.6,-1) -- (5,-1.3);

\draw[->] (5.3,-0.7) -- (4,-0.5);
\draw[->] (5,-1.3) -- (4,-1.5);

\draw[fill=blue,blue] (4,-2) circle[radius=1.5pt];
\draw[blue,->] (4,-2) -- (6,-2);

\draw[blue,->] (5,-1.55) -- (5,-1.8);

\draw (5,-2.5) node{$Z_2\! \cap\! W_0$};

\draw (3,-2.25);
\end{tikzpicture}
\]
The total puncturing rank is $k_P=2$. Applying \eqref{eqn: refined class general} we obtain
\begin{align} \nonumber [\Vcal(\Tsf)]^{\refined} & = \{ c(\uprho^\star N_\upiota) \cap s(\Vcal(\Tsf),\Acal(\Tsf)) \}_{-2} \\
\label{eqn: refined class long example} & = c_1(\uprho^\star N_\upiota) \cap s(\Vcal(\Tsf),\Acal(\Tsf))_{-1} + s(\Vcal(\Tsf),\Acal(\Tsf))_{-2}.
\end{align}
We first compute the Segre class terms. Consider the blowup of $\Acal(\Tsf)$ in $W_0$
\[ \uprho \colon \Acal(\Tsf^\prime) \to \Acal(\Tsf)\]
and let $E_0 = \uprho^{-1}(W_0) \subseteq \Acal(\Tsf^\prime)$ denote the exceptional divisor. Since the $Z_i$ intersect the blowup centre transversely, their strict and total transforms coincide:
\[ \uprho^{-1}(Z_1) = Z_1^\prime, \quad \uprho^{-1}(Z_2) = Z^\prime_2.\]
From $\Vcal(\Tsf) = Z_1 \cup Z_2 \cup W_0$ we obtain $\uprho^{-1} \Vcal(\Tsf) = Z_1^\prime \cup Z_2^\prime \cup E_0$. The latter is a Cartier divisor, so:
\[ s(\uprho^{-1} \Vcal(\Tsf), \Acal(\Tsf^\prime)) = \dfrac{Z^\prime_1+Z^\prime_2+E_0}{1+(Z^\prime_1+Z^\prime_2+E_0)} = \sum_{k\geq 1} (-1)^{k-1} (Z^\prime_1+Z^\prime_2+E_0)^k.\]
Following \eqref{eqn: refined class long example} we must extract the $k=1$ and $k=2$ terms. For $k=1$ we obtain:
\begin{equation} \label{eqn: Segre example dim -1} s(\Vcal(\Tsf),\Acal(\Tsf))_{-1} = \uprho_\star (Z^\prime_1+Z^\prime_2+E_0) = Z_1+Z_2.\end{equation}
Examining the $k=2$ term we have
\[ -(Z^\prime_1+Z^\prime_2+E_0)^2 = -(\uprho^\star(Z_1+Z_2)+E_0)^2 = -\uprho^\star(Z_1+Z_2)^2 - 2\uprho^\star(Z_1+Z_2) E_0 - E_0^2 \]
and pushing forward along $\uprho$ we obtain
\begin{equation} \label{eqn: Segre example dim -2} s(\Vcal(\Tsf),\Acal(\Tsf))_{-2} = -(Z_1+Z_2)^2 + W_0 = -Z_1^2 - Z_2^2 - 2W_{12} + W_0 \end{equation}
where $W_{12} = Z_1 Z_2 = Z_1 \! \cap Z_2$ as above. Combining \eqref{eqn: Segre example dim -1} and \eqref{eqn: Segre example dim -2} with \eqref{eqn: refined class long example} we obtain:
\begin{equation} \label{eqn: halfway formula refined class long example} [\Vcal(\Tsf)]^{\refined} = c_1(\uprho^\star N_\upiota) \cap Z_1 + c_1(\uprho^\star N_\upiota) \cap Z_2 - Z_1^2 - Z_2^2 - 2W_{12} + W_0.\end{equation}
It remains to describe the Chern class terms. The rank $2$ vector bundle $\uprho^\star N_\upiota$ splits into summands associated to the piecewise linear functions $\ftrop(p_1), \ftrop(p_2)$ recording the puncturing offsets. We begin with $c_1(\uprho^\star N_\upiota) \cap Z_1$. We compare $\ftrop(p_1)$ and $\ftrop(p_2)$ to the piecewise linear function $\mathsf{z}_1$ corresponding to $Z_1$. Along the ray corresponding to $Z_1$ we have 
\[ \ftrop(p_1) = \ftrop(p_2) = 2 \mathsf{z}_1 \]
where the multiplicity accounts for the difference between source and target lengths. This identity fails on cones where the tropical map degenerates in such a way that the puncturing offset $\ftrop(p_i)$ differs from the above edge length. It suffices to consider two-dimensional cones containing the ray corresponding to $Z_1$. Together with $W_{12}$ above, we identify the following relevant cones:
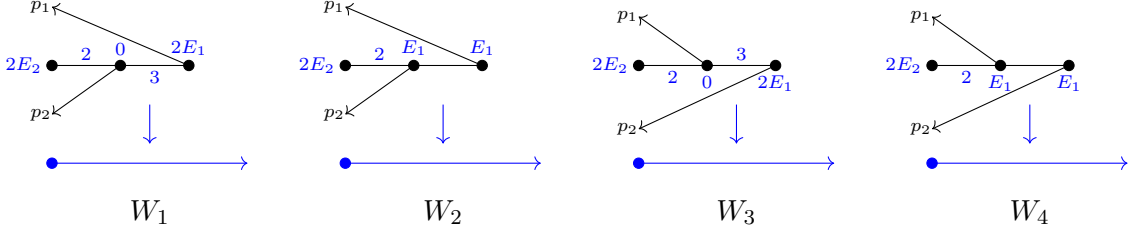
\[
\begin{tikzpicture}[scale=1.3]


\draw[fill=black] (-8,-1) circle[radius=1.5pt];
\draw[blue] (-8,-1) node[left]{\tiny$2E_2$};

\draw (-8,-1) -- (-7.3,-1);
\draw[blue] (-7.65,-1.05) node[above]{\tiny$2$};

\draw[fill=black] (-7.3,-1) circle[radius=1.5pt];
\draw[blue] (-7.3,-1) node[above]{\tiny$0$};

\draw (-7.3,-1) -- (-6.6,-1);
\draw[blue] (-6.95,-0.95) node[below]{\tiny$3$};

\draw[fill=black] (-6.6,-1) circle[radius=1.5pt];
\draw[blue] (-6.6,-1) node[above]{\tiny$2E_1$};

\draw[->] (-7.3,-1) -- (-8,-1.5);
\draw (-7.9,-1.5) node[left]{\tiny$p_2$};

\draw[->] (-6.6,-1) -- (-8,-0.4);
\draw (-7.9,-0.4) node[left]{\tiny$p_1$};

\draw[fill=blue,blue] (-8,-2) circle[radius=1.5pt];
\draw[blue,->] (-8,-2) -- (-6,-2);

\draw[blue,->] (-7,-1.4) -- (-7,-1.8);

\draw (-7,-2.5) node{$W_1$};



\draw[fill=black] (-5,-1) circle[radius=1.5pt];
\draw[blue] (-5,-1) node[left]{\tiny$2E_2$};

\draw (-5,-1) -- (-4.3,-1);
\draw[blue] (-4.65,-1.05) node[above]{\tiny$2$};

\draw[fill=black] (-4.3,-1) circle[radius=1.5pt];
\draw[blue] (-4.3,-1) node[above]{\tiny$E_1$};

\draw (-4.3,-1) -- (-3.6,-1);

\draw[fill=black] (-3.6,-1) circle[radius=1.5pt];
\draw[blue] (-3.6,-1) node[above]{\tiny$E_1$};

\draw[->] (-4.3,-1) -- (-5,-1.5);
\draw (-4.9,-1.5) node[left]{\tiny$p_2$};

\draw[->] (-3.6,-1) -- (-5,-0.4);
\draw (-4.9,-0.4) node[left]{\tiny$p_1$};

\draw[fill=blue,blue] (-5,-2) circle[radius=1.5pt];
\draw[blue,->] (-5,-2) -- (-3,-2);

\draw[blue,->] (-4,-1.4) -- (-4,-1.8);

\draw (-4,-2.5) node{$W_2$};



\draw[fill=black] (-2,-1) circle[radius=1.5pt];
\draw[blue] (-2,-1) node[left]{\tiny$2E_2$};

\draw (-2,-1) -- (-1.3,-1);
\draw[blue] (-1.65,-0.95) node[below]{\tiny$2$};

\draw[fill=black] (-1.3,-1) circle[radius=1.5pt];
\draw[blue] (-1.3,-1) node[below]{\tiny$0$};

\draw (-1.3,-1) -- (-0.6,-1);
\draw[blue] (-0.95,-1.05) node[above]{\tiny$3$};

\draw[fill=black] (-0.6,-1) circle[radius=1.5pt];
\draw[blue] (-0.6,-1) node[below]{\tiny$2E_1$};

\draw[->] (-1.3,-1) -- (-2,-0.5);
\draw (-1.9,-0.5) node[left]{\tiny$p_1$};

\draw[->] (-0.6,-1) -- (-2,-1.65);
\draw (-1.9,-1.65) node[left]{\tiny$p_2$};

\draw[fill=blue,blue] (-2,-2) circle[radius=1.5pt];
\draw[blue,->] (-2,-2) -- (0,-2);

\draw[blue,->] (-1,-1.4) -- (-1,-1.8);

\draw (-1,-2.5) node{$W_3$};



\draw[fill=black] (1,-1) circle[radius=1.5pt];
\draw[blue] (1,-1) node[left]{\tiny$2E_2$};

\draw (1,-1) -- (1.7,-1);
\draw[blue] (1.35,-0.95) node[below]{\tiny$2$};

\draw[fill=black] (1.7,-1) circle[radius=1.5pt];
\draw[blue] (1.7,-1) node[below]{\tiny$E_1$};

\draw (1.7,-1) -- (2.4,-1);

\draw[fill=black] (2.4,-1) circle[radius=1.5pt];
\draw[blue] (2.4,-1) node[below]{\tiny$E_1$};

\draw[->] (1.7,-1) -- (1,-0.5);
\draw (1.1,-0.5) node[left]{\tiny$p_1$};

\draw[->] (2.4,-1) -- (1,-1.65);
\draw (1.1,-1.65) node[left]{\tiny$p_2$};

\draw[fill=blue,blue] (1,-2) circle[radius=1.5pt];
\draw[blue,->] (1,-2) -- (3,-2);

\draw[blue,->] (2,-1.4) -- (2,-1.8);

\draw (2,-2.5) node{$W_4$};

\end{tikzpicture}
\]
Comparing the piecewise linear functions, we have
\begin{align*} \OO_{Z_1}(\ftrop(p_1)) & = \OO_{Z_1}(2Z_1+W_{12}+3W_1+W_2) \\
\OO_{Z_1}(\ftrop(p_2)) & = \OO_{Z_1}(2Z_1+W_{12}+3W_3+W_4)
\end{align*}
which gives
\[ c_1(\uprho^\star N_\upiota) \cap Z_1 = 4Z_1^2 + 2W_{12} + 3W_1 + W_2 + 3W_3 + W_4. \]
Similarly, for $c_1(\uprho^\star N_\upiota) \cap Z_2$ we identify the following cones:
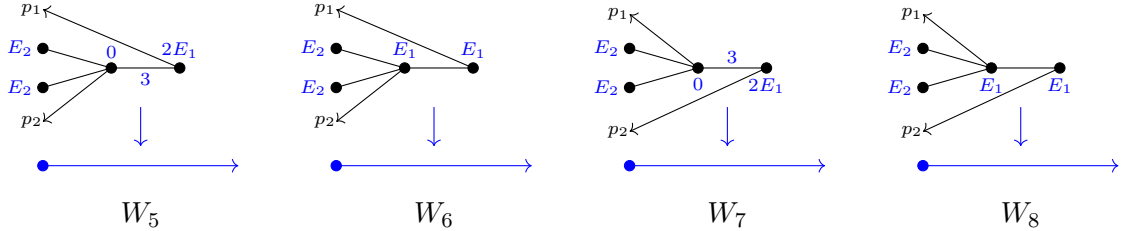
\[
\begin{tikzpicture}[scale=1.3]


\draw[fill=black] (-8,-0.8) circle[radius=1.5pt];
\draw[blue] (-8,-0.8) node[left]{\tiny$E_2$};

\draw[fill=black] (-8,-1.2) circle[radius=1.5pt];
\draw[blue] (-8,-1.2) node[left]{\tiny$E_2$};

\draw (-8,-0.8) -- (-7.3,-1);
\draw (-8,-1.2) -- (-7.3,-1);

\draw[fill=black] (-7.3,-1) circle[radius=1.5pt];
\draw[blue] (-7.3,-1) node[above]{\tiny$0$};

\draw (-7.3,-1) -- (-6.6,-1);
\draw[blue] (-6.95,-0.95) node[below]{\tiny$3$};

\draw[fill=black] (-6.6,-1) circle[radius=1.5pt];
\draw[blue] (-6.6,-1) node[above]{\tiny$2E_1$};

\draw[->] (-7.3,-1) -- (-8,-1.55);
\draw (-7.9,-1.55) node[left]{\tiny$p_2$};

\draw[->] (-6.6,-1) -- (-8,-0.4);
\draw (-7.9,-0.4) node[left]{\tiny$p_1$};

\draw[fill=blue,blue] (-8,-2) circle[radius=1.5pt];
\draw[blue,->] (-8,-2) -- (-6,-2);

\draw[blue,->] (-7,-1.4) -- (-7,-1.8);

\draw (-7,-2.5) node{$W_5$};



\draw[fill=black] (-5,-0.8) circle[radius=1.5pt];
\draw[blue] (-5,-0.8) node[left]{\tiny$E_2$};

\draw[fill=black] (-5,-1.2) circle[radius=1.5pt];
\draw[blue] (-5,-1.2) node[left]{\tiny$E_2$};

\draw (-5,-0.8) -- (-4.3,-1);
\draw (-5,-1.2) -- (-4.3,-1);

\draw[fill=black] (-4.3,-1) circle[radius=1.5pt];
\draw[blue] (-4.3,-1) node[above]{\tiny$E_1$};

\draw (-4.3,-1) -- (-3.6,-1);

\draw[fill=black] (-3.6,-1) circle[radius=1.5pt];
\draw[blue] (-3.6,-1) node[above]{\tiny$E_1$};

\draw[->] (-4.3,-1) -- (-5,-1.55);
\draw (-4.9,-1.55) node[left]{\tiny$p_2$};

\draw[->] (-3.6,-1) -- (-5,-0.4);
\draw (-4.9,-0.4) node[left]{\tiny$p_1$};

\draw[fill=blue,blue] (-5,-2) circle[radius=1.5pt];
\draw[blue,->] (-5,-2) -- (-3,-2);

\draw[blue,->] (-4,-1.4) -- (-4,-1.8);

\draw (-4,-2.5) node{$W_6$};



\draw[fill=black] (-2,-0.8) circle[radius=1.5pt];
\draw[blue] (-2,-0.8) node[left]{\tiny$E_2$};

\draw[fill=black] (-2,-1.2) circle[radius=1.5pt];
\draw[blue] (-2,-1.2) node[left]{\tiny$E_2$};

\draw (-2,-0.8) -- (-1.3,-1);
\draw (-2,-1.2) -- (-1.3,-1);

\draw[fill=black] (-1.3,-1) circle[radius=1.5pt];
\draw[blue] (-1.3,-1) node[below]{\tiny$0$};

\draw (-1.3,-1) -- (-0.6,-1);
\draw[blue] (-0.95,-1.05) node[above]{\tiny$3$};

\draw[fill=black] (-0.6,-1) circle[radius=1.5pt];
\draw[blue] (-0.6,-1) node[below]{\tiny$2E_1$};

\draw[->] (-1.3,-1) -- (-2,-0.45);
\draw (-1.9,-0.45) node[left]{\tiny$p_1$};

\draw[->] (-0.6,-1) -- (-2,-1.65);
\draw (-1.9,-1.65) node[left]{\tiny$p_2$};

\draw[fill=blue,blue] (-2,-2) circle[radius=1.5pt];
\draw[blue,->] (-2,-2) -- (0,-2);

\draw[blue,->] (-1,-1.4) -- (-1,-1.8);

\draw (-1,-2.5) node{$W_7$};



\draw[fill=black] (1,-0.8) circle[radius=1.5pt];
\draw[blue] (1,-0.8) node[left]{\tiny$E_2$};

\draw[fill=black] (1,-1.2) circle[radius=1.5pt];
\draw[blue] (1,-1.2) node[left]{\tiny$E_2$};

\draw (1,-0.8) -- (1.7,-1);
\draw (1,-1.2) -- (1.7,-1);

\draw[fill=black] (1.7,-1) circle[radius=1.5pt];
\draw[blue] (1.7,-1) node[below]{\tiny$E_1$};

\draw (1.7,-1) -- (2.4,-1);

\draw[fill=black] (2.4,-1) circle[radius=1.5pt];
\draw[blue] (2.4,-1) node[below]{\tiny$E_1$};

\draw[->] (1.7,-1) -- (1,-0.45);
\draw (1.1,-0.45) node[left]{\tiny$p_1$};

\draw[->] (2.4,-1) -- (1,-1.65);
\draw (1.1,-1.65) node[left]{\tiny$p_2$};

\draw[fill=blue,blue] (1,-2) circle[radius=1.5pt];
\draw[blue,->] (1,-2) -- (3,-2);

\draw[blue,->] (2,-1.4) -- (2,-1.8);

\draw (2,-2.5) node{$W_8$};

\end{tikzpicture}
\]
Comparing the piecewise linear functions, we have
\begin{align*} \OO_{Z_2}(\ftrop(p_1)) & = \OO_{Z_2}(Z_2+2W_{12}+3W_5+W_6) \\
\OO_{Z_2}(\ftrop(p_2)) & = \OO_{Z_2}(Z_2+2W_{12}+3W_7+W_8)
\end{align*}
which gives
\[ c_1(\uprho^\star N_\upiota) \cap Z_2 = 2Z_2^2 + 4W_{12} + 3W_5 + W_6 + 3W_7 + W_8. \]
Plugging into \eqref{eqn: halfway formula refined class long example} we obtain the final formula:
\[ [\Vcal(\Tsf)]^{\refined} = 3Z_1^2 + Z_2^2 + 4W_{12} + W_0 + 3W_1 + W_2 + 3W_3 + W_4 + 3W_5 + W_6 + 3W_7 + W_8.\]
This is a polynomial in strata classes. Strata denoted $Z$ have codimension $1$ while strata denoted $W$ have codimension $2$. The square terms can be described in terms of psi classes as in \Cref{example: higher rank regular embedding}.

In the preceding argument we have implicitly assumed that $\Acal(\Tsf)$ is smooth. This is false, but an examination of cones shows that the singular locus has codimension $3$. The excision sequence \cite[Chapter~1.8]{FultonBig} thus produces an identification of Chow groups
\[ A_{-2}(\Acal(\Tsf)) = A_{-2}(\Acal(\Tsf)^{\mathrm{sm}}).\]
We may therefore restrict to the smooth locus, on which the above arguments are valid.
\end{example}

\section{Statement of the comparison} \label{sec: correspondence}

\noindent We work towards the statement of our main theorem: the comparison between refined punctured Gromov--Witten theory and orbifold Gromov--Witten theory with extremal ages.

\subsection{Orbifold theory} \label{sec: orbifold theory} The refined punctured theory will be identified with the orbifold theory. To define the latter, we enhance the numerical data $\Lambda$ in \Cref{def: numerical data} as follows.
\begin{definition} \label{def: rooting data} \textbf{Rooting data} enhancing $\Lambda$ consist of a choice of:
\begin{itemize}
\item \textbf{Target rooting parameters.} Positive integers $(r_1,\ldots,r_k) \in \Z_{\geqslant 1}^k$.
\item \textbf{Source rooting parameters.} Positive integers $(s_1,\ldots,s_n) \in \Z_{\geqslant 1}^n$.
\end{itemize} 
Given a marking $i \in [n]$ let $J(i) \subseteq [k]$ denote the set of divisors which the marking is tangent to
\[ J(i) \colonequals \{ j \in [k] : \upalpha_{ij} \neq 0\}\]
and for any subset $J \subseteq [k]$ let $r_J$ denote the product of target rooting parameters: 
\[ r_J \colonequals \Pi_{j \in J} r_j.\]
We impose the following conditions on the rooting data:
\begin{enumerate}
\item \textbf{Divisibility.} For each marking $i \in [n]$ and divisor $j \in J(i)$ we have
\begin{align*} r_j \mid \upalpha_{ij} s_i.
\end{align*}
\item \textbf{Coprimality.} For each marking $i \in [n]$ we have
\[ \operatorname{lcm}\!\big( r_j/ \operatorname{gcd}(r_j,\upalpha_{ij}) \, \colon \, j \in J(i)\big)= s_i.\]
\item \textbf{Size.} For each marking $i \in [n]$ and divisor $j \in J(i)$ we have
\[ r_j > |\upalpha_{ij}|.\]
\end{enumerate}
\end{definition}
\begin{remark}
At each marking the homomorphism on isotropy groups is 
\begin{align*}
\upmu_{s_i} & \to \prod_{j \in J(i)} \upmu_{r_j} \\	
\upzeta_{s_i} & \mapsto \big(\upzeta_{r_j}^{\upalpha_{ij}}\big)_{j \in J(i)} = \big(\upzeta_{s_i}^{\widetilde{\upalpha}_{ij}}\big)_{j\in J(i)} \end{align*}
where $\widetilde{\upalpha}_{ij} \colonequals \upalpha_{ij} s_i/r_j$ is the gerby tangency (see~\Cref{sec: tropical twisted punctured maps}). This homomorphism is well-defined by the divisibility condition, and injective by the coprimality condition. A posteriori we obtain
\[ s_i \mid r_{J(i)} \]
which also follows numerically from the coprimality condition. The size condition ensures that the homomorphism on isotropy encodes the tangency, rather than only the tangency modulo $r$. Note that coprimality determines each $s_i$ in terms of the $r_j$ and $\upalpha_{ij}$.
\end{remark}
In contrast to the numerical data (\Cref{def: numerical data}), the rooting data (\Cref{def: rooting data}) is auxiliary, and we fix it once and for all. When dealing with the orbifold theory we always assume that rooting data has been chosen, and all the results we establish will be independent of that choice. A valid choice is obtained by taking $r_1,\ldots,r_k$ to be distinct large primes, and $s_i=r_{J(i)}$.

We use the same symbol $\Lambda$ to denote the numerical data enhanced with the auxiliary rooting data. This determines moduli spaces of orbifold maps:
\[
\begin{tikzcd}
 \Orb_\Lambda(X_{\bm{r}}) \ar[r,"\upmu"] & \Orb_\Lambda(\Acal_{\bm{r}}^k) \ar[r,"\upnu"] & \Mfrak^{\tw}(\Bcal \mathbb{G}^k_{\mathrm{m},\bm{r}}).
\end{tikzcd}
\]
The tangencies are encoded in the twisted sectors at the special points. The age of the $j$th universal line bundle at a marking $i$ is $\upalpha_{ij}/r_j$ (mod $1$). More precisely: 
\[ \operatorname{age}_i \Lcal_j = \begin{cases} \upalpha_{ij}/r_j \qquad & \text{if $\upalpha_{ij} \geqslant 0$,} \\ 1+ \upalpha_{ij}/r_j \qquad & \text{if $\upalpha_{ij} < 0$}. \end{cases} \]
The base $\Mfrak^{\tw}(\Bcal \mathbb{G}^k_{\mathrm{m},\bm{r}})$ is smooth \cite[Section~5.2]{AbramovichCadmanWise}, while the morphisms $\upmu$ and $\upnu$ are equipped with perfect obstruction theories
\begin{align*} \EE_\upmu & = (\Rder \uppi_\star f^\star T_{X_{\bm{r}}|D_{\bm{r}}})^\vee \\
\EE_\upnu & = \oplus_{j=1}^k (\Rder \uppi_\star	\Lcal_j)^\vee
\end{align*}
where $T_{X_{\bm{r}}|D_{\bm{r}}}$ denotes the logarithmic tangent bundle of the pair $(X_{\bm{r}}|D_{\bm{r}})$, i.e. the relative tangent bundle of the smooth morphism $X_{\bm{r}}\to\Acal_{\bm{r}}^k$ encoding the normal crossings divisor $D_{\bm{r}}$.

\begin{remark} In the literature it is sometimes assumed that the target rooting parameters are pairwise coprime. This simplifies the numerics, but is not necessary for stating or proving the correspondence theorems. In certain contexts it is desirable to relax this assumption, for instance when investigating birational invariance (see \Cref{sec: birational invariance}). Consider the blowup of a pair $(X|D_1\!+\!D_2)$ in the stratum $D_1 \cap D_2$. This gives a morphism of pairs
\[ (X^\prime|D_0^\prime\!+\!D_1^\prime\!+\!D_2^\prime) \to (X|D_1\!+\!D_2)\]
with $D_0^\prime$ the exceptional divisor and $D_1^\prime,D_2^\prime$ the strict transforms. Fix target rooting parameters $(r_0,r_1,r_2)$. Then the induced morphism on root stacks
\[ X^\prime_{(r_0,r_1,r_2)} \to X_{(r_1,r_2)} \]
only exists if $r_1$ and $r_2$ both divide $r_0$.
\end{remark}

\subsection{Blowups and numerical data} \label{sec: lifted numerical data} 

Fix a simple normal crossings pair $(X|D)$. An \textbf{iterated blowup} is a logarithmic modification
\[ (X^\prime|D^\prime) \to (X|D) \]
which factors into a sequence of blowups along closed strata. These correspond to iterated stellar subdivisions $\Sigma^\prime \to \Sigma$ of the tropicalisation. The system of iterated stellar subdivisions is cofinal in the system of all subdivisions \cite[Chapter~1.7]{Oda}. Since we focus on limiting behaviour, it is sufficient to restrict to this subsystem.

Given an iterated blowup, there is a cone map $\mathsf{r} \colon \Rplus^{k^\prime} \to \Rplus^k$ recording the pullbacks of boundary divisors. This commutes with the global framings:
\[
\begin{tikzcd}
\Sigma^\prime \ar[r] \ar[d] & \Rplus^{k^\prime} \ar[d,"\mathsf{r}"] \\
\Sigma \ar[r] & \Rplus^k.
\end{tikzcd}
\]

\begin{definition}
Fix an iterated blowup $\uprho \colon (X^\prime|D^\prime) \to (X|D)$ and numerical data $\Lambda^\prime$ for a moduli space of punctured maps to $(X^\prime|D^\prime)$. The \textbf{pushed forward numerical data} $\Lambda \colonequals \uprho_\star \Lambda^\prime$ is defined as follows.

The genus is taken to be the same (i.e. $g \colonequals g^\prime$) and the curve class is taken to be the pushforward (i.e. $\upbeta \colonequals \uprho_\star \upbeta^\prime$). Given a marking $i \in [n]$ the vector of contact orders is taken to be the pushforward:
\[ (\upalpha_{ij})_{j=1}^k \colonequals \mathsf{r} (\upalpha^\prime_{ij})_{j=1}^{k^\prime}. \]
\end{definition}

\begin{remark} \label{rmk: contact data lift not unique} Fix numerical data $\Lambda$ and consider the set of lifts $\Lambda^\prime$. In the positive contact setting, the vector of contact orders $(\upalpha_{ij})_{j=1}^k$ can be identified with a point in $\Sigma(\N)$. The identification
\[ \Sigma^\prime(\N)=\Sigma(\N) \]
then ensures that the lifting of contact orders is unique. From this we also deduce the uniqueness of the lifted curve class, see \cite[Section~3.5]{BNR2}. It follows that $\Lambda$ admits a \emph{unique} lift $\Lambda^\prime$.

In the negative contact setting, this is no longer the case. Conceptually this occurs because $\Sigma^\prime(\Z)$ is larger than $\Sigma(\Z)$. Suppose for simplicity that $\Sigma=\Rplus^k$. Given a puncturing $p \in P$ let $J(p) \subseteq [k]$ denote the set of divisor components to which $p$ has contact
\[ J(p) = \{ j \in [k] : \upalpha_{pj} \neq 0\}.\]
Suppose that the associated cone $\upsigma_{J(p)} \in \Sigma$ is subdivided into multiple cones via $\Sigma^\prime \to \Sigma$. There are then different choices of lifted cone, giving rise to different lifted contact orders. An example is illustrated below.
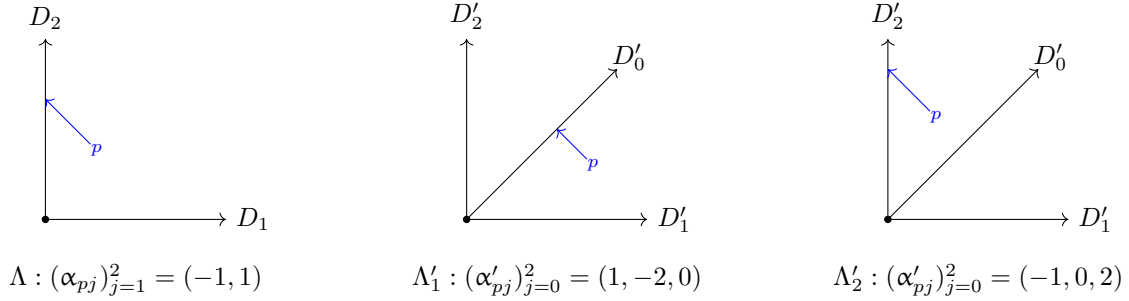
\[
\begin{tikzpicture}[scale=0.8]
\draw[fill=black] (-7,0) circle[radius=1.5pt];
\draw[->] (-7,0) -- (-4,0);
\draw (-4,0) node[right]{\small$D_1$};
\draw[->] (-7,0) -- (-7,3);
\draw (-7,3) node[above]{\small$D_2$};

\draw[blue,->] (-6.25,1.25) -- (-7,2);
\draw[blue] (-6.15,1.15) node{\tiny$p$};

\draw (-5.5,-1) node{\small$\Lambda : (\upalpha_{pj})_{j=1}^2 = (-1,1)$};
	
\draw[fill=black] (0,0) circle[radius=1.5pt];
\draw[->] (0,0) -- (3,0);
\draw (3,0) node[right]{\small${D}^\prime_1$};
\draw[->] (0,0) -- (0,3);
\draw (0,3) node[above]{\small${D}^\prime_2$};

\draw[->] (0,0) -- (2.5,2.5);
\draw (2.7,2.7) node{\small$D^\prime_0$};

\draw[blue,->] (2,1) -- (1.5,1.5);
\draw[blue] (2.1,0.9) node{\tiny$p$};

\draw (1.5,-1) node{\small$\Lambda^\prime_1 : (\upalpha^\prime_{pj})_{j=0}^2 = (1,-2,0)$};

\draw[fill=black] (7,0) circle[radius=1.5pt];
\draw[->] (7,0) -- (10,0);
\draw (10,0) node[right]{\small${D}^\prime_1$};
\draw[->] (7,0) -- (7,3);
\draw (7,3) node[above]{\small${D}^\prime_2$};

\draw[->] (7,0) -- (9.5,2.5);
\draw (9.7,2.7) node{\small$D^\prime_0$};

\draw[blue,->] (7.7,1.8) -- (7,2.5);
\draw[blue] (7.8,1.7) node{\tiny$p$};

\draw (8.5,-1) node{\small$\Lambda^\prime_2 : (\upalpha^\prime_{pj})_{j=0}^2 = (-1,0,2)$};
\end{tikzpicture}
\]
This failure of uniqueness is expected. Each negative contact moduli space appears as a boundary stratum in a positive contact moduli space. A logarithmic modification of the target induces a logarithmic modification of the positive contact moduli space, under which a single stratum may be dominated by multiple strata. 

Alternatively: in the positive contact setting, the lift $\Lambda^\prime$ formally computes the strict transform of a generic curve with numerical data $\Lambda$; in the negative contact setting, it can occur that all such curves map into the blowup centre, rendering the strict transform meaningless.
\end{remark}

Given numerical data $\Lambda$ and an iterated blowup $(X^\prime|D^\prime) \to (X|D)$, we fix once and for all numerical data $\Lambda^\prime$ pushing forward to $\Lambda$. The main theorem applies to any choice of such data.

\subsection{Slope-sensitivity} \label{sec: slope-sensitivity} Fix numerical data $\Lambda$ for a moduli space of punctured maps to $(X|D)$. We define $\Lambda$-sensitive iterated blowups, following \cite[Section~4]{BNR2}. For each subset $J \subseteq [k]$ of size two, consider the pair $(X|D_J)$ and write
\[ \Sigma_J \colonequals \Sigma(X|D_J).\]
The numerical data $\Lambda$ induces numerical data $\Lambda_J$ giving a moduli space of stable punctured maps $\Punct_{\Lambda_J}(X|D_J)$. This decomposes into strata indexed by tropical types of maps to $\Sigma_J$. We assemble all the slopes $m_{\vec{e}}$ which appear in these tropical types, and which also belong to the nonnegative orthant of their associated cone, meaning that they have nonnegative coordinates with respect to the lattice basis consisting of the ray generators. Fix an iterated stellar subdivision
\[ \Sigma_J^\prime \to \Sigma_J \]
which contains all the rays generated by these slopes (such subdivisions always exist). Since $\Sigma$ is smooth there is a canonical projection
\[ \ptrop_J \colon \Sigma \to \Sigma_J\]
and the subdivision $\Sigma_J^\prime \to \Sigma_J$ pulls back to a subdivision $\ptrop_J^\star \Sigma_J^\prime \to \Sigma$.

\begin{definition} Consider an iterated blowup $(X^\prime|D^\prime) \to (X|D)$ corresponding to an iterated stellar subdivision $\Sigma^\prime \to \Sigma$. This iterated blowup is \textbf{$\Lambda$-sensitive} (or \textbf{slope-sensitive}) if and only if the subdivision $\Sigma^\prime \to \Sigma$ simultaneously refines $\ptrop_J^\star \Sigma_J^\prime \to \Sigma$ for every subset $J \subseteq [k]$ of size two.
\end{definition}

Slope-sensitive subdivisions always exist, are stable under further blowups, and are cofinal in the system of all iterated blowups \cite[Lemma~4.1]{BNR2}.
 
\subsection{Main theorem} \label{sec: main theorem}
Fix a simple normal crossings pair $(X|D)$ and numerical data $\Lambda$ for a moduli space of stable punctured maps in genus zero. For each divisor component $j \in [k]$, let $n_j$ denote the number of markings $i \in [n]$ with $\upalpha_{ij} < 0$.

\begin{theorem}[\Cref{thm: main snc pairs introduction}] \label{thm: main} For each $\Lambda$-sensitive blowup $(X^\prime|D^\prime) \to (X|D)$ the refined punctured Gromov--Witten theory of $(X^\prime|D^\prime)$ coincides with the orbifold Gromov--Witten theory of $X^\prime_{\bm{r}^\prime}$ up to a rooting factor. Precisely, there is a diagram of moduli spaces:
\[
\begin{tikzcd}[column sep=small]
\Punct_{\Lambda^\prime}(X^\prime|D^\prime) \ar[rd,"\uptheta" below] && \NPunctOrb_{\Lambda^\prime}(X^\prime_{\bm{r}^\prime}|D^\prime_{\bm{r}^\prime}) \ar[ld,"\upalpha"] \ar[rd,"\upomega" below] && \Orb_{\Lambda^\prime}(X^\prime_{\bm{r}^\prime}) \ar[ld,"\upeta"] \\
& \NPunct_{\Lambda^\prime}(X^\prime|D^\prime) && \NOrb_{\Lambda^\prime}(X^\prime_{\bm{r}^\prime})	
\end{tikzcd}
\]
satisfying the following virtual pushforward identities
\begin{align*}
\uptheta_\star[\Punct_{\Lambda^\prime}(X^\prime|D^\prime)]^{\refined} & = [\NPunct_{\Lambda^\prime}(X^\prime|D^\prime)]^{\refined} \\
\upalpha_\star[\NPunctOrb_{\Lambda^\prime}(X^\prime_{\bm{r}^\prime}|D^\prime_{\bm{r}^\prime})]^{\refined} & = \Pi_{j=1}^{k^\prime} (r_{j}^\prime)^{-n_j^\prime} [\NPunct_{\Lambda^\prime}(X^\prime|D^\prime)]^{\refined} \\
\upomega_\star[\NPunctOrb_{\Lambda^\prime}(X^\prime_{\bm{r}^\prime}|D^\prime_{\bm{r}^\prime})]^{\refined} & = [\NOrb_{\Lambda^\prime}(X^\prime_{\bm{r}^\prime})]^{\virt} \\
\upeta_\star[\Orb_{\Lambda^\prime}(X^\prime_{\bm{r}^\prime})]^{\virt} & = [\NOrb_{\Lambda^\prime}(X^\prime_{\bm{r}^\prime})]^{\virt}
\end{align*}
\end{theorem}

Every moduli space appearing in the above diagram admits a forgetful morphism to the space of \textbf{stable maps with refined insertions}
\[ \Msf_{0,n,\upbeta}(X^\prime) \times_{X^{\prime n}} \Pi_{i=1}^n D^\prime_{J(i)} \]
where $\Msf_{0,n,\upbeta}(X^\prime)$ is the space of ordinary stable maps, and
\[ D^\prime_{J(i)} = \bigcap_{j \in J(i)} D^\prime_j \]
is the closed stratum of $X^\prime$ to which the marking $i \in [n]$ is constrained by the tangency conditions. This allows us to compare invariants with arbitrary evaluation and cotangent line classes.

\section{Twisted punctured maps} \label{sec: twisted punctured moduli}

\noindent In the next two sections we establish the correspondence for smooth pairs.

\begin{theorem}[\Cref{thm: main smooth pairs introduction}] \label{thm: correspondence smooth pairs} Let $(X|D)$ be a smooth pair. In genus zero, the refined sector of the punctured Gromov--Witten theory coincides with the orbifold Gromov--Witten theory. Precisely, there is a correspondence
\[
\begin{tikzcd}
& \PunctOrb_\Lambda(X_r|D_r) \ar[ld,"\upalpha"] \ar[rd,"\upomega" {yshift=-0.35cm,xshift=-0.3cm}] & \\
\Punct_\Lambda(X|D) & & \Orb_\Lambda(X_r).
\end{tikzcd}
\]
satisfying the following virtual pushforward identities
\begin{align*}
\upalpha_\star[\PunctOrb_\Lambda(X_r|D_r)]^{\refined} & = r^{-|P|} [\Punct_\Lambda(X|D)]^{\refined} \\[0.2cm]
\upomega_\star[\PunctOrb_\Lambda(X_r|D_r)]^{\refined} & = [\Orb_\Lambda(X_r)]^{\virt}
\end{align*}
and such that $\upalpha$ induces an isomorphism on coarse spaces, and $\upomega$ is an isomorphism.
\end{theorem}
We first construct the space $\PunctOrb_\Lambda(X_r|D_r)$, which we refer to as the \textbf{chimera}, and establish its key properties (\Cref{sec: twisted punctured moduli}). We then prove the virtual pushforward identities (\Cref{sec: proof smooth}). Throughout we make extensive use of the universal smooth pair, given by $[\Aone/\Gm]$ with boundary divisor $\Bcal \Gm$. We denote this pair by $(\Acal|\Dcal)$. 

The chimera will be built ``tropical first'': we first construct the analogous tropical moduli space, and then use it to produce the desired algebraic moduli space. We emphasise that the construction of the chimera works in all genera. The genus zero assumption only enters in~\Cref{sec: proof smooth} in the proof of the pushforward identities. 

\subsection{Tropical twisted curves} \label{sec: twisted tropical curves} 

\begin{definition} \label{def: type of twisted tropical curve} A \textbf{type of twisted tropical curve} consists of the following data:
\begin{itemize}
\item \textbf{Coarse graph.} A finite graph $\Gamma$ consisting of vertices $V(\Gamma)$, edges $E(\Gamma)$, and labelled marking legs $L(\Gamma)$. Each vertex $v \in V(\Gamma)$ is equipped with a genus label $g_v \in \N$.	
\item \textbf{Rooting data.} A positive rooting parameter associated to every edge and marking leg:
\[ s \colon E(\Gamma) \sqcup L(\Gamma) \to 	\Z_{\geqslant 1}.\]
\end{itemize}
Given a type $\widetilde{\Gamma}$ of twisted tropical curve, we consider the associated cone
\[ \uptau_{\widetilde\Gamma} \colonequals \Rplus^{E(\Gamma)}.\]
A specialisation of types $\widetilde\Gamma \rightsquigarrow \widetilde\Gamma^\prime$ is obtained by contracting a (possibly empty) subset of the edges and composing with an automorphism of the resulting type, and induces a face morphism $\uptau_{\widetilde\Gamma^\prime} \hookrightarrow \uptau_{\widetilde\Gamma}$. Ranging over all specialisations, we obtain a diagram of cones connected by face morphisms. The \textbf{moduli space of twisted tropical curves} 
\[ \MTropOrb\]
is the colimit of this diagram in the category of cone stacks.
\end{definition}

This moduli space decomposes into connected components, obtained by fixing the genus, number of legs, and rooting parameters along the legs. Note that even these connected components are not of finite type, as there is no restriction on the rooting parameters along the edges. In applications we typically restrict to a finite type substack, justified since spaces of maps to geometric targets are of finite type.

\begin{remark}
The cone stack $\MTropOrb$ can also be described as a moduli stack, following \cite{CavalieriChanUlirschWise}. Given a strictly convex rational polyhedral cone $\upsigma$ we let $Q_\upsigma$ denote the associated dual monoid. A \textbf{twisted tropical curve} over $\upsigma$ consists of a type of twisted tropical curve (as in \Cref{def: type of twisted tropical curve}) together with gerby edge length assignments  (the adjective is justified by comparison with the underlying coarse curve, see \Cref{construction: map twisted tropical curves to tropical curves}):
\[ \tilde\ell \colon E(\Gamma) \to Q_{\upsigma} \! \setminus \! \{0\}.\]
The objects of $\MTropOrb$ over $\upsigma$ consist of the twisted tropical curves over $\upsigma$. Morphisms in $\MTropOrb$ correspond to specialisations of edge length parameters, as in \cite[Definition~3.3]{CavalieriChanUlirschWise}. This defines a category fibred in groupoids over the category of strictly convex rational polyhedral cones, and this category forms a stack with respect to the face topology.
\end{remark}

\begin{construction} \label{construction: map twisted tropical curves to tropical curves} We construct a morphism of cone stacks
\begin{equation} \label{eqn: morphism twisted tropical curves to tropical curves} \MTropOrb \to \MTrop \end{equation}
as follows. Each type $\widetilde\Gamma$ of twisted tropical curve induces a type $\Gamma$ of tropical curve (the \textbf{coarsening}) by forgetting the rooting parameters. The associated cones are both orthants of dimension $E(\Gamma)$. We coordinatise them as follows:
\[ \uptau_{\widetilde\Gamma} = (\tilde\ell_e)_{e \in E(\Gamma)}, \qquad \uptau_\Gamma = (\ell_e)_{e \in E(\Gamma)}.\]
We refer to the $\tilde\ell_e$ as \textbf{gerby edge lengths} and the $\ell_e$ as \textbf{coarse edge lengths}. These are related by scaling by the rooting parameter:
\[ s_e \tilde\ell_e = \ell_e. \]
This identity determines a morphism $\uptau_{\widetilde\Gamma} \to \uptau_{\Gamma}$. These morphisms are compatible under specialisation, and hence glue to produce the morphism \eqref{eqn: morphism twisted tropical curves to tropical curves}.\qed
\end{construction}

There is an algebraic stack $\Mfrak^{\tw}$ parameterising twisted curves, see for instance~\cite{AbramovichVistoli,OlssonLogTwisted,AbramovichLectures}. The stack $\Mfrak^{\tw}$ admits a morphism 
\[ \Mfrak^{\tw} \to \Acal(\MTropOrb) \]
induced by the boundary stratification. The following proof is useful for understanding the geometry of these spaces.

\begin{lemma} \label{lem: twisted curves as log modification of curves} The following diagram is cartesian:
\[
\begin{tikzcd}
\Mfrak^{\tw} \ar[r] \ar[d] \ar[rd,phantom,"\square"] & 	\Mfrak \ar[d] \\
\Acal(\MTropOrb) \ar[r] & \Acal(\MTrop).
\end{tikzcd}
\]
\end{lemma}

\begin{proof} This follows immediately from \Cref{construction: map twisted tropical curves to tropical curves} and the fact that $\Mfrak^{\tw} \to \Mfrak$ is locally on the source a root stack \cite[Theorem~1.10]{OlssonLogTwisted}. However since we will require a generalisation of this result, we provide a self-contained argument.

Temporarily let $S$ denote the fibre product. There is a morphism $\Mfrak^\tw \to S$ induced by the following diagram:
\[
\begin{tikzcd}
\Mfrak^{\tw} \ar[rd] \ar[ddr, bend right] \ar[rrd, bend left] & & \\
& S \ar[r] \ar[d] \ar[rd,phantom,"\square"] & \Mfrak \ar[d] \\
& \Acal(\MTropOrb) \ar[r] & \Acal(\MTrop).	
\end{tikzcd}
\]
We construct an inverse. The universal curve over $\Mfrak$ pulls back to a family of curves $\uppi \colon C \to S$. This admits a unique enhancement to a minimal logarithmically smooth curve
\[ (C,M_C) \to (S,M_S)\]
where the logarithmic structure $M_S$ is pulled back along the morphism $S \to \Acal(\MTrop)$. Pulling back the logarithmic structure along the morphism $S \to \Acal(\MTropOrb)$ produces a finite-index extension:
\begin{equation} \label{eqn: twisted curve extension of base structure} M_S \hookrightarrow M_S^\prime. \end{equation}
Finally, since the stack $\Acal(\MTropOrb)$ decomposes into connected components indexed by rooting parameters along the legs, the morphism $S \to \Acal(\MTropOrb)$ defines locally constant rooting parameter functions
\begin{equation} \label{eqn: twisted curve rooting parameters} s_i \colon S \to \Z_{\geqslant 1}\end{equation}
for $i \in [n]$. Taken together, \eqref{eqn: twisted curve extension of base structure} and \eqref{eqn: twisted curve rooting parameters} give the data of a logarithmically twisted curve in the sense of \cite[Definition~1.8]{OlssonLogTwisted}. Given this data, \cite[Theorem~1.9]{OlssonLogTwisted} produces a twisted curve. We summarise the construction. The twisted curve arises as a logarithmic root stack 
\[ \Ccal \to C\]
induced by a finite-index extension of ghost sheaves
\begin{equation} \label{eqn: twisted curve extension ghost sheaf} \Mbar_C \hookrightarrow \Mbar_C^\prime. \end{equation}
We construct the extension as follows. At a general point $x \in C$ we have $\Mbar_{C,x} = \Mbar_{S,\uppi(x)}$. Choose the extension:
\[ \Mbar_{S,\uppi(x)} \hookrightarrow \Mbar^\prime_{S,\uppi(x)}.\]
At a marking $p_i \in C$ we have $\Mbar_{C,p_i} = \Mbar_{S,\uppi(p_i)} \oplus \N$. Choose the extension:
\[ \Mbar_{S,\uppi(p_i)} \oplus \N \hookrightarrow \Mbar^\prime_{S,\uppi(p_i)} \oplus \tfrac{1}{s_i} \N.\]
Finally at a node $q \in C$ consider the associated smoothing parameter $u_q \in \Mbar_{S,\uppi(q)}$. Let 
\[ u_q^\prime \in \Mbar^\prime_{S,\uppi(q)} \]
denote the unique primitive element such that $s_q u_q^\prime = u_q$ for some $s_q \in \Z_{\geqslant 1}$. Then take the extension
\[ \Mbar_{S,\uppi(q)} \oplus_\N \N^2 \hookrightarrow \Mbar^\prime_{S,\uppi(q)} \oplus_\N \N^2 \]
where the morphism $\N \to \Mbar^\prime_{S,\uppi(q)}$ is given by $1 \mapsto u_q^\prime$. The map is induced by the given inclusion on the first factor, and scaling by $s_q$ on the second factor.

These extensions of stalks are compatible with generisation. We thus obtain a finite-index extension \eqref{eqn: twisted curve extension ghost sheaf}. This induces a family of twisted curves $\Ccal \to S$ and hence a morphism
\[ S \to \Mfrak^{\tw} \]
providing an inverse to the morphism $\Mfrak^{\tw} \to S$ constructed above.
\end{proof}

\subsection{Tropical twisted punctured maps} \label{sec: tropical twisted punctured maps} Fix numerical data $\Lambda$ enhanced by rooting data, as in Definitions~\ref{def: numerical data} and \ref{def: rooting data}. With this setup, we define the \textbf{gerby degree} to be
\[ \tilde{d} \colonequals d/r \in \tfrac{1}{r} \Z \]
and we define the \textbf{gerby contact order} at each marking $i \in [n]$ to be
\begin{equation} \label{eqn: gerby vs coarse tangency} \widetilde{\upalpha}_i \colonequals \dfrac{\upalpha_i}{(r/s_i)} \in \Z. \end{equation}
We emphasise that the gerby contact order is an integer; this follows from the divisibility condition (\Cref{def: rooting data}). Its geometric interpretation is as follows. Restricting to a neighbourhood of $p_i \in C$ we obtain a diagram
\begin{equation} \label{eqn: commuting square gerby contact order}
\begin{tikzcd}
(\widetilde{C}|\widetilde{p}_i) \ar[d,"{\sqrt[\uproot{4} s_i]{p_i}}" left] \ar[r,"\tilde{f}"] & 	(\widetilde{X}|\widetilde{D}) \ar[d,"{\sqrt[r]{D}}" right]\\
(C|p_i) \ar[r,"f"] & (X|D)
\end{tikzcd}
\end{equation}
where the vertical morphisms are root stacks. The gerby contact order is the exponent relating the equation of the gerby divisor with the equation of the gerby marking:
\[ \tilde{f}^\star z_{\widetilde{D}} = z_{\widetilde{p}_i}^{\widetilde{\upalpha}_i}. \]
The identity \eqref{eqn: gerby vs coarse tangency} then follows from the commutativity of the square \eqref{eqn: commuting square gerby contact order} and the corresponding relation on the coarse spaces:
\[ f^\star z_D = z_{p_i}^{\upalpha_i}.\]

The following definition arises naturally by examining line bundles and sections on twisted curves and root stacks. When $r=1$, it reduces to \Cref{def: type of tropical punctured map}.
\begin{definition} \label{def: type of twisted tropical punctured map} A \textbf{tropical type} of twisted punctured map to $(\Rplus,r)$ with numerical data $\Lambda$ consists of the following data:
\begin{itemize}
\item \textbf{Source graph.} A finite graph $\Gamma$ consisting of vertices $V(\Gamma)$, finite edges $E(\Gamma)$, ordinary markings $O(\Gamma)$, and punctures $P(\Gamma)$. This is equipped with bijections
\[ O(\Gamma) \cong O, \qquad P(\Gamma) \cong P \]
where $O$ and $P$ are the sets of ordinary markings and punctures in \Cref{def: numerical data}. Each vertex $v \in V(\Gamma)$ is equipped with genus and gerby degree labels
\[ g_v \in \N, \quad \tilde{d}_v \in \tfrac{1}{r} \Z\]
such that $g = b_1(\Gamma) + \Sigma_{v \in V(\Gamma)} g_v$ and $\tilde{d} = \Sigma_{v \in V(\Gamma)} \tilde{d}_v$.
\item \textbf{Rooting parameters.} An assignment of a rooting parameter to every edge:
\[ s \colon E(\Gamma) \to \Z_{\geqslant 1}.\]
\item \textbf{Image cones.} Cones of $\Rplus$ associated to every vertex and edge of $\Gamma$
\[ v \rightsquigarrow \upsigma_v, \quad e \rightsquigarrow \upsigma_e \]
such that if $v \in e$ then $\upsigma_v \subseteq \upsigma_e$.
\item \textbf{Gerby slopes.} For each oriented edge $\vec{e}$ of $\Gamma$ with $\upsigma_e=\Rplus$ a  gerby slope
\[ \widetilde{m}_{\vec{e}} \in \Z \]
satisfying $\widetilde{m}_{\vec{e}}=-\widetilde{m}_{\cev{e}}$. At each vertex $v \in V(\Gamma)$ these slopes must satisfy the following gerby balancing condition
\begin{equation} \label{eqn: gerby balancing} \tilde{d}_v = \sum_{v \in e} \dfrac{\widetilde{m}_{\vec{e}}}{s_e} + \sum_{v \in i} \dfrac{\widetilde\upalpha_i}{s_i} \end{equation}
where each edge $e$ is oriented to point away from $v$, and the second sum is over markings supported at $v$.

Finally, at each oriented edge $\vec{e} \in \vec{E}(\Gamma)$ the pair $(\widetilde{m}_{\vec{e}},s_e)$ of gerby slope and rooting parameter must satisfy the divisibility, coprimality, and size conditions of \Cref{def: rooting data}.
\end{itemize}
\end{definition}
To each tropical type $\Theta$ of twisted punctured map to $(\Rplus,r)$ we associate a rational polyhedral cone. This proceeds similarly to the usual ($r=1$) case. The cone is embedded in an orthant
\[
\prod_{e\in E(\Gamma)}\mathbb R_{\geqslant 0}\times\prod_{v\in V(\Gamma)} \upsigma_v
\]
coordinatised by gerby edge lengths $\tilde{\ell}_e$ and gerby vertex positions $\tilde{\ftrop}(v)$.

\begin{definition} The \textbf{tropical moduli cone} associated to $\Theta$ is the rational polyhedral cone
\[
\widetilde{\uptau}_{\Theta} \subseteq \prod_{e\in E(\Gamma)}\mathbb R_{\geqslant 0}\times\prod_{v\in V(\Gamma)} \upsigma_v
\]
defined by the equations
\begin{equation} \label{eqn: gerby continuity}
\tilde{\ftrop}(v_2) = \tilde{\ftrop}(v_1) + \widetilde{m}_{\vec{e}} \tilde{\ell}_e
\end{equation}
for every oriented edge $\vec{e} \in \vec{E}(\Gamma)$ starting at $v_1$ and ending at $v_2$.
\end{definition}

Specialisations of tropical types, including automorphisms, induce face maps of tropical moduli cones. These are used to glue the tropical moduli cones into a cone stack.

\begin{definition} The moduli space of \textbf{tropical twisted punctured maps} is the colimit
\[ \mathsf{T}_r \colonequals \Trop_\Lambda(\Rplus,r) = \varinjlim_{\Theta} \widetilde\uptau_{\Theta} \]
over all cones $\widetilde\uptau_{\Theta}$ associated to tropical types $\Theta$ with numerical data $\Lambda$, connected via face maps corresponding to specialisations of tropical types.
\end{definition}

\begin{notation} To ease readability, we will use the same symbol $\widetilde{\uptau}$ to denote both a tropical type of twisted punctured map to $(\Rplus,r)$, and the associated cone in $\Tsf_r$. Since the former index the latter, there is no ambiguity. \end{notation}


\begin{construction} \label{construction: twisted tropical maps to tropical maps} We construct a natural morphism of cone stacks 
\[ \Tsf_r \to \Tsf \]
as follows. Each type $\widetilde{\uptau}$ of twisted tropical punctured map induces a type $\uptau$ of tropical punctured map (the \textbf{coarsening}) as follows. We retain the source graph and image cones, and forget the rooting parameters. Gerby degrees and gerby slopes are replaced by coarse degrees and coarse slopes as follows:
\[ d_v \colonequals r \tilde{d}_v, \qquad m_{\vec{e}} \colonequals (r/s_e) \widetilde{m}_{\vec{e}}.	\]
It remains to check the balancing condition. At each vertex $v \in V(\Gamma)$ we multiply the gerby balancing condition \eqref{eqn: gerby balancing} by $r$ to obtain:
\[ d_v = r\tilde{d}_v = \sum_{v \in e} (r/s_e) \widetilde{m}_{\vec{e}} + \sum_i (r/s_i) \widetilde{\upalpha}_i = \sum_{v \in e} m_{\vec{e}} + \sum_i \upalpha_i. \]
This verifies that the coarsening $\uptau$ is a valid type of tropical punctured map. We now define a morphism $\widetilde{\uptau} \to \uptau$ by imposing the following relationships between coordinates:
\begin{align} 
\label{eqn: relation gerby coarse edge length} \ell_e & = s_e \tilde{\ell}_e \qquad \text{\ \ for $e \in E(\Gamma)$,}\\
\label{eqn: relation gerby coarse vertex positions} \ftrop(v) & = r \, \tilde{\ftrop}(v) \qquad \text{for $v \in V(\Gamma)$.}	
\end{align}
To check that this respects the relations, multiply \eqref{eqn: gerby continuity} by $r$ to obtain
\[ r \tilde{\ftrop}(v_1) + r \widetilde{m}_{\vec{e}}\, \tilde{\ell}_e = r \tilde{\ftrop}(v_2).\]
The right-hand side is equal to $\ftrop(v_2)$. We simplify the left-hand side as follows:
\begin{align*}
r \tilde{\ftrop}(v_1) + r \widetilde{m}_{\vec{e}} \tilde{\ell}_e & = \ftrop(v_1) + r \left(\frac{s_e m_{\vec{e}}}{r} \right) \left(\frac{\ell_e}{s_e}\right)  \\
& = \ftrop(v_1) + m_{\vec{e}}\, \ell_e.
\end{align*}
This verifies that the morphism $\widetilde{\uptau} \to \uptau$ given by \eqref{eqn: relation gerby coarse edge length}, \eqref{eqn: relation gerby coarse vertex positions} is well-defined. These morphisms respect generisation of types, and hence glue to produce a morphism of cone stacks $\Tsf_r \to \Tsf$ as required.\qed
\end{construction}

\begin{lemma} \label{lem: twisted maps to maps separated} The morphism of Artin fans $\Acal(\Tsf_r) \to \Acal(\Tsf)$ is a generalised root stack.
\end{lemma}

\begin{proof} 
We will show that $\Tsf_r \to \Tsf$ is obtained by taking a system of compatible lattice extensions of each cone in the target. By \Cref{construction: twisted tropical maps to tropical maps}, on each cone of the source $\widetilde{\uptau} \to \uptau$, the morphism is induced by a lattice extension: each source length (corresponding to the edge $e$) is multiplied by $s_e$, every target length is multiplied by $r$. On the other hand the map $\Tsf_r \to \Tsf$ induces a bijection between the cones of the source and target. This is because the coprimality condition (\Cref{def: rooting data}) uniquely determines the rooting parameter $s_e$ in terms of $r$ and $m_{\vec{e}}$.
\end{proof}

\subsection{Twisted puncturing substack}
As in \Cref{sec: puncturing substack} there is for each puncture a natural puncturing offset morphism
\[ \tilde\ftrop(p) \colon \mathsf{T}_r \to \Rplus\]
given by $\tilde\ftrop(p) = \tilde\ftrop(v)$ where $v$ is the vertex supporting the puncture $p$. Since the target is a smooth pair we have $k_P=|P|$. We perform the corresponding algebraic fibre product to obtain
\[
\begin{tikzcd}
\Vcal(\mathsf{T}_r) \ar[r,hook] \ar[d] \ar[rd,phantom,"\square"] & \Acal(\mathsf{T}_r) \ar[d] \\
\Bcal \Gm^{P} \ar[r,hook,"\upiota_r"] & \Acal^{P}
\end{tikzcd}
\]
and refer to $\Vcal(\mathsf{T}_r)$ as the \textbf{puncturing substack}. As in \Cref{sec: refined punctured class}, the puncturing substack is equipped with a \textbf{refined virtual class}
\[ [\Vcal(\Tsf_r)]^{\refined} \colonequals \upiota_r^![\Acal(\Tsf_r)].\]

\begin{remark}
The identity $\ftrop(v)=r \tilde\ftrop(v)$ relates the coarse and gerby puncturing offsets. These fit into the commutative diagram
\[
\begin{tikzcd}
\Tsf_r \ar[r] \ar[d] & \Tsf \ar[d] \\
\Rplus^{P} \ar[r,"r"] & \Rplus^{P}
\end{tikzcd}
\]
from which we obtain the algebraic diagram
\[
\begin{tikzcd}	
\Acal(\Tsf_r) \ar[r] \ar[d] & \Acal(\Tsf) \ar[d] \\
\Acal^{P} \ar[r,"r"] & \Acal^{P}
\end{tikzcd}
\]
but we caution that this diagram is typically not cartesian, see \ref{item:TrVsT'} in the proof of \Cref{prop: map between puncturing substacks} below.
\end{remark}

There is a morphism $\upnu \colon \Vcal(\Tsf_r) \to \Vcal(\Tsf)$ fitting into the following cubical diagram, with front and back faces cartesian:
\begin{equation} \label{eqn: cubical diagram comparing punctured substacks}
\begin{tikzcd}
\Vcal(\Tsf_r) \ar[rr,hook] \ar[dd] \ar[rd,"\upnu"] & & \Acal(\Tsf_r) \ar[dd] \ar[rd] \\
& \Vcal(\Tsf) \ar[rr,crossing over,hook] & & \Acal(\Tsf) \ar[dd] \\
\Bcal \Gm^{P} \ar[rr,"\upiota_r" {yshift=0.05cm, xshift=0.5cm},hook] \ar[rd,"r"] & & \Acal^{P} \ar[rd,"r"] \\
& \Bcal \Gm^{P} \ar[from=uu, crossing over] \ar[rr,"\upiota",hook] & & \Acal^{P}.
\end{tikzcd}
\end{equation}

\begin{proposition} \label{prop: map between puncturing substacks} The morphism $\upnu \colon \Vcal(\Tsf_r) \to \Vcal(\Tsf)$ is proper, and factors as a finite gerbe followed by an iterated root stack. It satisfies:
\[ \upnu_\star [\Vcal(\Tsf_r)]^{\refined} = r^{-|P|} [\Vcal(\Tsf)]^{\refined}.\]
\end{proposition}

\begin{proof} 
To see that $\upnu$ is proper, first note that the closed embedding $\Vcal(\Tsf) \hookrightarrow \Acal(\Tsf)$ is proper. Moreover in the composite
\[ \Vcal(\Tsf_r) \hookrightarrow \Acal(\Tsf_r) \to \Acal(\Tsf) \]
both morphisms are proper: the first is a closed embedding and the second is an iterated root stack (\Cref{lem: twisted maps to maps separated}). We obtain a diagram
\[
\begin{tikzcd}
\Vcal(\Tsf_r) \ar[rr,"\upnu"] \ar[rd] && \Vcal(\Tsf) \ar[ld] \\
& \Acal(\Tsf) &	
\end{tikzcd}
\]
in which both morphisms to $\Acal(\Tsf)$ are proper. It follows that $\upnu$ is proper as claimed. 

To compare the refined virtual classes, we wish to push and pull around the cubical diagram \eqref{eqn: cubical diagram comparing punctured substacks}. There is a minor irritation: only the front and back faces of the cube are cartesian. On the other hand, \emph{all} the faces are cartesian up to appropriate birational modifications and finite covers. The comparison is achieved by passing to a more intricate diagram which keeps track of these alterations (similar trickery occurs, for instance, in the proof of the product formula \cite{BehrendProduct}). Consider therefore the following diagram, in which all the squares are cartesian:
\[
\begin{tikzcd}
\Vcal(\Tsf_r) \ar[r,"\uplambda"] \ar[d] & C \ar[d,"\upzeta"] \ar[rr] && \Acal(\Tsf_r) \ar[d,"\upsigma"] \\
B \ar[r] \ar[dd] & A \ar[rr,hook] \ar[dd] \ar[rd,"\upeta"] & & \Acal(\Tsf_r^\prime) \ar[dd,"\uptau" yshift=15pt] \ar[rd,"\uprho"] \\
& & \Vcal(\Tsf) \ar[rr,crossing over,hook] & & \Acal(\Tsf) \ar[dd] \\
\Bcal \Gm^{P} \ar[r,"\upmu",hook] \ar[rrd, bend right = 10pt, "r"] & (\Bcal \mathbb{G}_{\mathrm{m}}^{(r)})^{P} \ar[rr,"\upkappa" xshift=10pt,hook] \ar[rd] & & \Acal^{P} \ar[rd,"r"] \\
& & \Bcal \Gm^{P} \ar[from=uu, crossing over] \ar[rr,hook,"\upiota"] & & \Acal^{P}.
\end{tikzcd}
\]
Note that $\Acal(\Tsf_r^\prime), A, B,$ and $C$ are auxiliary fiber product stacks. We use the following geometric facts about this diagram:
\begin{enumerate}
\item The morphism $r \colon \Acal^P \to \Acal^P$ is flat.
\item The stack $\Bcal \mathbb{G}_{\mathrm{m}}^{(r)}$ is the root stack of $\Bcal \Gm$ along the pair $(\Lcal,0)$ where $\Lcal$ is the universal line bundle. It is given explicitly by:
\[ \Bcal \mathbb{G}_{\mathrm{m}}^{(r)} = [(t^r\!=\!0)/\Gm] \subseteq [\Aone_t/\Gm].\]
The morphism $\upmu$ is the inclusion of the reduced substack $[(t\!=\!0)/\Gm]$ in each factor, and so:
\[ \upmu_\star [ \Bcal \Gm^{P}] = r^{-|P|} [(\Bcal \mathbb{G}_{\mathrm{m}}^{(r)})^{P}].\]
The morphism $r \colon \Bcal \Gm^P \to \Bcal \Gm^P$ is a finite gerbe. See \cite[Example 2.4.3]{Cadman}.
\item\label{item:TrVsT'} The cone stack $\Tsf_r^\prime$ is obtained from $\Tsf$ by introducing $r$th roots of all puncturing offsets. In $\Tsf_r$ we further introduce roots of \emph{all} vertex positions (including of those vertices that do not support any punctures) and all source edge lengths (according to their associated rooting parameters). Thus, the morphisms
\[ \Acal(\Tsf_r) \xrightarrow{\upsigma} \Acal(\Tsf_r^\prime) \xrightarrow{\uprho} \Acal(\Tsf)\]
are both iterated root stacks, in particular proper and birational.
\item The morphism $\uptau \circ \upsigma$ has smooth codomain and irreducible domain, and hence there is a well-defined pullback obtained via the graph construction \cite[Chapter~8]{FultonBig}.
\end{enumerate}
It follows immediately that $B \to \Vcal(\Tsf)$ is a finite gerbe and $\Vcal(\Tsf_r) \to B$ is an iterated root stack. It remains to establish the pushforward result. We will compute $\upnu_\star [\Vcal(\Tsf_r)]^{\refined}$ by pushing it forward along $\uplambda$, $\upzeta$, $\upeta$ in turn. The necessary commutativity and compatibility of Gysin pullbacks with proper pushforwards are discussed in \cite[Section 5.1]{Kresch}. We have:
\[ [\Vcal(\Tsf_r)]^{\refined} = (\upkappa \upmu)^![\Acal(\Tsf_r)] = (\upkappa \upmu)^! (\uptau \upsigma)^! [ \Acal^{P}] = (\uptau \upsigma)^!(\upkappa \upmu)^!  [ \Acal^{P}] = (\uptau \upsigma)^! [\Bcal \Gm^{P}].\]
We first push forward along $\uplambda$:
\begin{align*} \uplambda_\star [\Vcal(\Tsf_r)]^{\refined} & = \uplambda_\star (\uptau \upsigma)^! [\Bcal \Gm^{P}] = (\uptau \upsigma)^! \upmu_\star [\Bcal \Gm^{P}] = r^{-|P|} (\uptau \upsigma)^! [(\Bcal \mathbb{G}_{\mathrm{m}}^{(r)})^{P}].
\end{align*}
We next push forward along $\upzeta$:
\begin{align*} \upzeta_\star \uplambda_\star [\Vcal(\Tsf_r)]^{\refined} & = r^{-|P|} \upzeta_\star (\uptau \upsigma)^! [(\Bcal \mathbb{G}_{\mathrm{m}}^{(r)})^{P}] = r^{-|P|}\upzeta_\star (\uptau \upsigma)^! \upkappa^! [\Acal^{P}] \\
& = r^{-|P|} \upzeta_\star \upkappa^! (\uptau \upsigma)^! [\Acal^{P}] = r^{-|P|} \upzeta_\star \upkappa^! [\Acal(\Tsf_r)] \\
& = r^{-|P|} \upkappa^! \upsigma_\star [\Acal(\Tsf_r)] = r^{-|P|} \upkappa^! [\Acal(\Tsf_r^\prime)] \\
& = r^{-|P|} \upiota^! [\Acal(\Tsf_r^\prime)].
\end{align*}
The final equality uses $\upkappa^!=\upiota^!$ which holds because $r$ is flat. Finally we push forward along $\upeta$:
\begin{align*} \upnu_\star [\Vcal(\Tsf_r)]^{\refined} & = \upeta_\star \upzeta_\star \uplambda_\star [\Vcal(\Tsf_r)]^{\refined} = r^{-|P|} \upeta_\star \upiota^! [\Acal(\Tsf_r^\prime)] \\
& = r^{-|P|} \upiota^! \uprho_\star [\Acal(\Tsf_r^\prime)] = r^{-|P|} \upiota^! [\Acal(\Tsf)] \\
& = r^{-|P|} [\Vcal(\Tsf)]^{\refined}	.\qedhere
\end{align*}
\end{proof}

\subsection{Chimera} \label{sec: chimera} The chimera will be used to interpolate between punctured and orbifold moduli. 

\begin{definition} \label{def: chimera} The moduli space $\PunctOrb_\Lambda$ for both universal and geometric targets is defined by the fibre diagram:
\begin{equation} \label{diagram: defining chimera}
\begin{tikzcd}
\PunctOrb_\Lambda(X_r|D_r) \ar[r] \ar[d,"\uppsi_r"] \ar[rd,phantom,"\square"] & \Punct_\Lambda(X|D) \ar[d,"\uppsi"] \\
\PunctOrb_\Lambda(\Acal_r|\Dcal_r) \ar[r,"\upalpha"] \ar[d,"\upphi_r"] \ar[rd,phantom,"\square"] & \Punct_\Lambda(\Acal|\Dcal) \ar[d,"\upphi"] \\
\Vcal(\mathsf{T}_r) \ar[r,"\upnu"] & \Vcal(\mathsf{T}).
\end{tikzcd}
\end{equation}
We establish the key properties of the chimera:
\begin{itemize}
\item Obstruction theory (\Cref{sec: chimera obstruction theory}).
\item Universal family (\Cref{sec: chimera universal family}).
\item Formal properties (\Cref{sec: chimera formal properties}).	
\end{itemize}
A related construction appears independently in \cite{JohnstonFrobenius}.
\end{definition}

\subsubsection{Obstruction theories} \label{sec: chimera obstruction theory}
The morphism $\upphi_r$ is smooth because $\upphi$ is smooth. The morphism $\uppsi_r$ carries a perfect obstruction theory because $\uppsi$ carries a perfect obstruction theory \cite[Proposition~7.2]{BehrendFantechi}. Since $(X_r|D_r) \to (X|D)$ is logarithmically \'etale, the pullback of $\Tlog_{X|D}$ to $X_r$ equals $\Tlog_{X_r|D_r}$ and so:
\begin{equation} \label{eqn: obstruction theory chimera} \EE_{\uppsi_r} = (\Rder \uppi_\star f^\star \Tlog_{X_r|D_r})^\vee.\end{equation}
As in \Cref{sec: refined punctured class}, we use smooth and virtual pullback to define \textbf{refined virtual classes}:
\begin{align*}
[\PunctOrb_\Lambda(\Acal_r|\Dcal_r)]^{\refined} & \colonequals \upphi_r^\star [\Vcal(\Tsf_r)]^{\refined}, \\
[\PunctOrb_\Lambda(X_r|D_r)]^{\refined} & \colonequals \uppsi_r^![\PunctOrb_\Lambda(\Acal_r|\Dcal_r)]^{\refined}.
\end{align*}

\subsubsection{Universal family} \label{sec: chimera universal family} Since we have defined the chimera via a fibre product, its universal family needs to be constructed manually. 

\begin{construction}  \label{construction: universal family PunctOrb} We construct a \textbf{universal twisted punctured map} $\Ccal \to (\Acal_r|\Dcal_r)$ over $\PunctOrb_\Lambda(\Acal_r|\Dcal_r)$. This fits into a diagram
\[
\begin{tikzcd}
(\Acal_r|\Dcal_r) \ar[r] & (\Acal|\Dcal) & \\
\Ccal \ar[u] \ar[rd] \ar[r] & C \ar[r] \ar[d] \ar[u] \ar[rd,phantom,"\square"] & C \ar[d]  \\
& \PunctOrb_\Lambda(\Acal_r|\Dcal_r) \ar[r] & \Punct_\Lambda(\Acal|\Dcal).
\end{tikzcd}
\]
where $C \to (\Acal|\Dcal)$ is the universal punctured map pulled back from $\Punct_\Lambda(\Acal|\Dcal)$, and $\Ccal \to C$ is a root construction.

The construction emulates the proof of \Cref{lem: twisted curves as log modification of curves}. To ease notation, let $S \colonequals \PunctOrb_\Lambda(\Acal_r|\Dcal_r)$. The pullback of the universal punctured curve carries a natural logarithmic structure
\[ (C,M_C) \to (S,M_S)\]
where the logarithmic structure $M_S$ is pulled back via the morphism $S \to \Acal(\Tsf)$. The morphism $S \to \Acal(\Tsf_r)$ produces a finite-index extension $M_S \hookrightarrow M_S^\prime$. We will now produce a finite-index extension
\begin{equation} \label{eqn: twisted punctured maps ghost sheaf extension} \Mbar_C \hookrightarrow \Mbar^\prime_C \end{equation}
whose associated root stack $\Ccal \to C$ will be our twisted punctured curve. The extension at general points, ordinary markings, and nodes is given as in the proof of \Cref{lem: twisted curves as log modification of curves}. It remains to describe the extension at punctures. At a puncture $p_i$ we have
\[ \ol{M}_{C,p_i} \hookrightarrow \ol{M}_{S,\uppi(p_i)} \oplus \Z \]
with $\ol{M}_{C,p_i}$ generated by $\ol{M}_{S,\uppi(p_i)} \oplus \N$ and $f^\flat(1)$. We define $\ol{M}_{C,p_i}^\prime$ to be the finite-index extension generated by:
\begin{enumerate}
\item $\ol{M}_{S,\uppi(p_i)}^\prime$.
\item $\tfrac{1}{s_i} \N$.
\item $\tfrac{1}{r} f^\flat(1)$.
\end{enumerate}
These extensions of stalks are compatible with generisation, and thus globalise to an extension \eqref{eqn: twisted punctured maps ghost sheaf extension}. This induces a root stack
\[ \Ccal \to C \]
which we refer to as the \textbf{universal twisted punctured curve}. Finally, we must verify that the punctured map $C \to (\Acal|\Dcal)$ lifts to a twisted punctured map to the root stack:
\[
\begin{tikzcd}
\Ccal \ar[r] \ar[d,dashed] & C  \ar[d,"f"] \\
(\Acal_r|\Dcal_r) \ar[r]  & (\Acal|\Dcal).
\end{tikzcd}
\]
Since both horizontal morphisms are root stacks, this is a purely tropical problem. It suffices to show that the section
\[ f^\flat(1) \in \Mbar_C \]
has an $r$th root in $\Mbar_C^\prime$. This can be seen by directly inspecting the stalks defined above. \qed
\end{construction}

\begin{remark}
 This construction suggests that the chimera can be defined directly as a moduli problem over fine and saturated logarithmic schemes, whose objects over $S$ are diagrams
 \bcd
 \Ccal\ar[r,"\widetilde{f}"]\ar[d] & (\Acal_r|\Dcal_r) \ar[d,"r"]\\
 C\ar[r,"f"]\ar[d] &(\Acal|\Dcal)\\ 
 S
 \ecd
 where $f$ is the relative coarsening of $r\circ \widetilde{f}$. The characteristic sheaves of $C$ and $\Ccal$ at a puncture $p_i\in \Ccal$ are:
 \[\ol{M}_{C,p_i}=\langle \ol{M}_{S,\uppi(p_i)}, \N, f^\flat(1)\rangle\quad\subseteq \quad
  \ol{M}_{\Ccal,p_i}=\langle \ol{M}_{S,\uppi(p_i)},\tfrac{1}{s_i} \N,\tfrac{1}{r} f^\flat(1)\rangle \quad\subseteq\quad
  \ol{M}_{S,\uppi(p_i)}\oplus\Z.
 \]
In the spirit of this paper, we prefer to stress the tropical origin of the chimera.
\end{remark}

\subsubsection{Formal properties} \label{sec: chimera formal properties}
Finally, we establish formal properties required of the chimera.

\begin{lemma} \label{lem: map LogOrb to Orb} There is a natural morphism $\PunctOrb_\Lambda(\Acal_r|\Dcal_r) \to \Orb_\Lambda(\Acal_r)$.
	\end{lemma}
\begin{proof} This follows immediately from \Cref{construction: universal family PunctOrb}. The morphism $\PunctOrb_\Lambda(\Acal_r|\Dcal_r) \to \Mfrak^{\tw}$ can also be obtained formally, using the following diagram
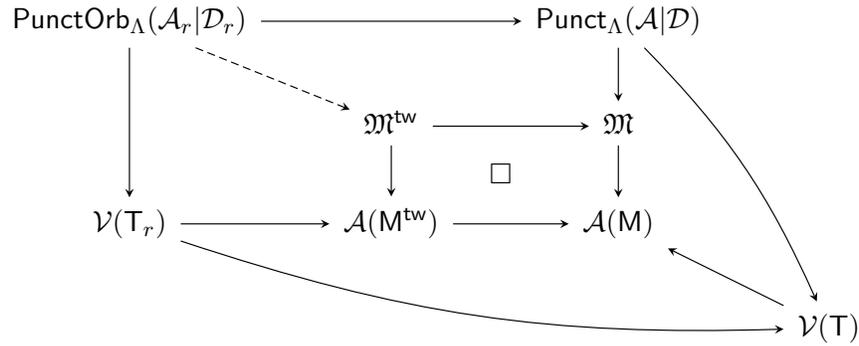
\[
\begin{tikzcd}
\PunctOrb_\Lambda(\Acal_r|\Dcal_r) \ar[rd,dashed] \ar[rr] \ar[dd] & & \Punct_\Lambda(\Acal|\Dcal) \ar[d] \ar[dddr,bend left=10pt] \\
& \Mfrak^{\tw} \ar[r] \ar[d] \ar[rd,phantom,"\square"] & \Mfrak \ar[d] \\
\Vcal(\mathsf{T}_r) \ar[r] \ar[rrrd, bend right=10pt] & \Acal(\MTropOrb) \ar[r] & \Acal(\MTrop) \\
& & & \Vcal(\mathsf{T}) \ar[ul]
\end{tikzcd}
\]
where the morphisms through $\Vcal(\Tsf)$ verify that the outer square commutes. 
\end{proof}

\begin{lemma} \label{lem: LogOrb as fibre product over twisted curves} There is a cartesian square:
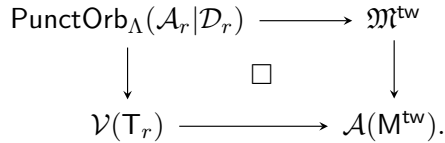
\[
\begin{tikzcd}
\PunctOrb_\Lambda(\Acal_r|\Dcal_r) \ar[r] \ar[d] \ar[rd,phantom,"\square"] & \Mfrak^{\tw} \ar[d] \\
\Vcal(\mathsf{T}_r) \ar[r] & \Acal(\MTropOrb).	
\end{tikzcd}
\]	
\end{lemma}
\begin{proof}
Consider the following cubical diagram:
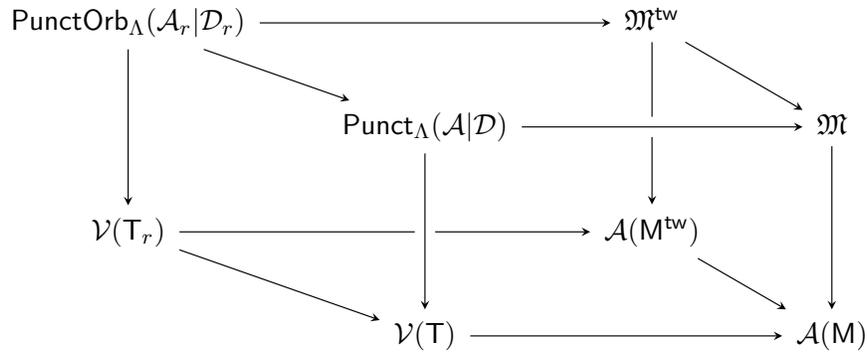
\[
\begin{tikzcd}
\PunctOrb_\Lambda(\Acal_r|\Dcal_r) \ar[rr] \ar[dd] \ar[rd] & & \Mfrak^{\tw} \ar[dd] \ar[rd] \\
& \Punct_\Lambda(\Acal|\Dcal) \ar[rr,crossing over]  & & \Mfrak \ar[dd] \\
\Vcal(\mathsf{T}_r) \ar[rr] \ar[rd] & & \Acal(\MTropOrb) \ar[rd] \\
& \Vcal(\mathsf{T}) \ar[from=uu, crossing over]\ar[rr] & & \Acal(\MTrop)
\end{tikzcd}
\]
The front face is cartesian by \eqref{eqn: cartesian square punctured maps over prestable curves}. The left face is cartesian by the definition of $\PunctOrb_\Lambda(\Acal_r|\Dcal_r)$. It follows that the composition of the front and left faces is cartesian. This is equal to the composition of the back and right faces. The right face is cartesian by \Cref{lem: twisted curves as log modification of curves}. It follows that the back face is cartesian, as claimed.
\end{proof}

\section{Proof: smooth pairs}  \label{sec: proof smooth}

\noindent In the previous section we constructed the chimera $\PunctOrb_\Lambda(X_r|D_r)$ and the diagram
\[
\begin{tikzcd}
& \PunctOrb_\Lambda(X_r|D_r) \ar[ld,"\upalpha"] \ar[rd,"\upomega" {yshift=-0.35cm,xshift=-0.3cm}] & \\
\Punct_\Lambda(X|D) & & \Orb_\Lambda(X_r).
\end{tikzcd}
\]
We now establish the comparisons of moduli spaces and virtual classes, first for universal and then for general targets.

We begin by comparing the chimera with the punctured theory (\Cref{sec: chimera vs punctured}). This is direct and relatively straightforward: the chimera is constructed from the punctured moduli space by fibring over puncturing substacks, and the latter are entirely controlled by tropical data.

Comparing the chimera with the orbifold theory is significantly more challenging and subtle. We begin by showing that $\upomega$ is an isomorphism (\Cref{sec: spaces same}). This is far from formal: the two spaces have entirely different origins and are equipped with incompatible stratifications.

Having shown that the spaces are the same, our problem reduces to comparing two different obstruction theories on the same space: one corresponding to the refined punctured theory, the other corresponding to the orbifold theory. We carry out a direct comparison of these obstruction theories (\Cref{sec: comparing obstruction theories}) and thus establish the comparison for universal targets. Finally, we bootstrap from universal to general targets (\Cref{sec: correspondence smooth pairs general targets}).\medskip

\noindent \textbf{A word of caution.} Before we begin in earnest we stress an important point. The moduli space of orbifold maps
\begin{equation} \label{eqn: orbifold moduli space} \Orb_\Lambda(\Acal_r) \end{equation}
carries a virtual fundamental class, induced by an obstruction theory defined relative to the universal Picard stack over the moduli space of twisted curves \cite[Section~5.2]{AbramovichCadmanWise}. However, while the space \eqref{eqn: orbifold moduli space} is always virtually smooth, it is almost never smooth. It typically consists of several irreducible components of varying dimensions. This phenomenon occurs already in genus zero.

Consequently, virtual structures enter already at the level of universal targets. The strategy of proving birationality for universal targets and concluding via virtual pullback no longer applies. Even at the level of universal targets, a comparison of obstruction theories is required. As we have noted already, this marks a sharp departure from related prior work \cite{AbramovichCadmanWise,AbramovichMarcusWise,BNR2}.

An exception is when the numerical data $\Lambda$ does not include any negative contact/high age markings. Here a miracle occurs, and $\Orb_\Lambda(\Acal_r)$ is irreducible of the expected dimension. However even in this case, the space is not smooth: it becomes obstructed over the boundary, whenever the section vanishes on a component of the curve. Combined with the splitting formalism for twisted stable maps, this accounts for the excess dimensionality when negative contact/high age markings appear.

Since we show that the orbifold moduli space coincides with the chimera, this identifies the irreducible components of the orbifold moduli space with the irreducible components of the puncturing substack $\Vcal(\Tsf_r)$. The latter are entirely controlled by tropical geometry.

\subsection{Chimera versus punctured} \label{sec: chimera vs punctured}

\begin{theorem} \label{thm: punctured orbifold vs punctured} The morphism $\upalpha$ induces an isomorphism on coarse spaces, and satisfies:
\[ \upalpha_\star[\PunctOrb_\Lambda(\Acal_r|\Dcal_r)]^{\refined} = r^{-|P|} [\Punct_\Lambda(\Acal|\Dcal)]^{\refined}.\]
\end{theorem}

\begin{proof} From \eqref{diagram: defining chimera} and the description of $\upnu \colon \Vcal(\Tsf_r) \to \Vcal(\Tsf)$ as a finite gerbe followed by an iterated root stack (\Cref{prop: map between puncturing substacks}) we see that $\upalpha$ induces an isomorphism on coarse spaces. Moreover we have:
\[ \upalpha_\star[\PunctOrb_\Lambda(\Acal_r|\Dcal_r)]^{\refined} = \upalpha_\star \upphi_r^\star [\Vcal(\Tsf_r)]^{\refined} = \upphi^\star \upnu_\star [\Vcal(\Tsf_r)]^{\refined}. \]
Applying \Cref{prop: map between puncturing substacks} we obtain
\[ \upphi^\star \upnu_\star [\Vcal(\Tsf_r)]^{\refined} = r^{-|P|} \upphi^\star [\Vcal(\Tsf)]^{\refined} = r^{-|P|}[\Punct_\Lambda(\Acal|\Dcal)]^{\refined}\] 
as required.
\end{proof}

This establishes one half of \Cref{thm: correspondence smooth pairs} for universal targets. The next two sections are dedicated to the other half, which is significantly more delicate.

\subsection{Identifying spaces} \label{sec: spaces same} 

\begin{theorem}\label{thm: spaces same}
The morphism
\[ \upomega \colon \PunctOrb_\Lambda(\Acal_r|\Dcal_r) \to \Orb_\Lambda(\Acal_r) \]
is an isomorphism.
\end{theorem}

While this implies birationality of moduli spaces, this does \emph{not} by itself imply an equality of invariants since, as discussed above, virtual structures play a key role already for universal targets. This also explains why \Cref{thm: spaces same} appears in such a strong form. In order to compare obstruction theories we require complete control over the relative geometry of $\upomega$. Birationality or finiteness alone is insufficient, in contrast to similar arguments in related contexts.

\begin{proof} We construct an inverse to $\upomega$. By \Cref{lem: LogOrb as fibre product over twisted curves} there is a cartesian square
\[
\begin{tikzcd}
\PunctOrb_\Lambda(\Acal_r|\Dcal_r) \ar[r] \ar[d] \ar[rd,phantom,"\square"] & \Mfrak^{\tw} \ar[d] \\
\Vcal(\mathsf{T}_r) \ar[r] & \Acal(\MTropOrb),
\end{tikzcd}
\]
so it suffices to construct morphisms from $\Orb_\Lambda(\Acal_r)$ to the two factors of the fibre product. We already have the forgetful morphism
\[ \Orb_\Lambda(\Acal_r) \to \Mfrak^{\tw}\]
so it remains to construct a morphism
\begin{equation} \label{eqn: morphism giving new log structure on Orb} \Orb_\Lambda(\Acal_r) \to \Vcal(\mathsf{T}_r). \end{equation}
We will first construct a morphism to $\Acal(\mathsf{T}_r)$ and then show that it factors through the closed substack $\Vcal(\mathsf{T}_r)$. Constructing a morphism $\Orb_\Lambda(\Acal_r) \to \Acal(\mathsf{T}_r)$ amounts to constructing a new logarithmic structure on $\Orb_\Lambda(\Acal_r$), distinct from the logarithmic structure pulled back from the moduli space of twisted curves.\smallskip

\noindent
\emph{Step 1: Tropical type of a twisted map.} Given a twisted map $\Ccal \to \Acal_r$ we construct a tropical type of twisted punctured map to $(\Rplus,r)$ in the sense of \Cref{def: type of twisted tropical punctured map}. The source graph $\Gamma$ is the decorated dual graph of $\Ccal$. The gerby degree labels are given by
\[ \tilde{d}_v \colonequals \deg \Lcal|_{\Ccal_v} \]
where $\Lcal$ is the line bundle and $\Ccal_v \subseteq \Ccal$ is the irreducible component corresponding to $v \in V(\Gamma)$. The rooting parameter at an edge is given by the order of isotropy at the corresponding node. Given a vertex $v \in V(\Gamma)$ the image cone is assigned via the following rule: if the section $s \in H^0(\Ccal,\Lcal)$ vanishes identically on $\Ccal_v$ then we set $\upsigma_v=\Rplus$, and otherwise we set $\upsigma_v=0$. The image cone assigned to an edge is obtained similarly, by examining the vanishing of the section at the corresponding node.

It remains to describe the gerby slopes. These are determined uniquely by imposing the balancing condition at each vertex. Since $\Gamma$ is a tree, this system of equations is consistent and has a unique solution. This solution can be obtained explicitly by choosing a root vertex of $\Gamma$ and constructing the gerby slopes inductively, by flowing towards this root vertex. This argument occurs frequently in the tropical enumerative geometry of rational curves, see e.g. \cite[Lemma~3.4]{BNR2}.\smallskip

\noindent
\emph{Step 2: Ghost sheaf and stratification.} Given a point of $\Orb_\Lambda(\Acal_r)$, the previous step associates a tropical type $\uptau$ of twisted punctured map to $(\Rplus,r)$. This defines a strictly convex rational polyhedral cone, with a dual toric monoid. Specialisations between types give rise to face maps between cones.

Associating to each point its toric monoid, we obtain a constructible sheaf of monoids on $\Orb_\Lambda(\Acal_r)$ in the smooth topology. We denote it by $\Mbar$ and take it to be the ghost sheaf of our putative logarithmic structure.\smallskip

\noindent
\emph{Step 3: Divisors associated to generators.} Our goal is to produce a logarithmic structure $M \to \OO_{\Orb_\Lambda(\Acal_r)}$ whose associated ghost sheaf is $\Mbar$. We employ the Olsson--Borne--Vistoli formalism~\cite{OlssonLogStacks,BorneVistoli} which allows us to define a logarithmic structure as a map of symmetric monoidal stacks
\begin{equation} \label{eqn: map ghost sheaf to divisor sheaf for Orb} \Mbar \to \mathfrak{Div}_{\Orb_\Lambda(\Acal_r)},\end{equation}
where $\mathfrak{Div}_{\Orb_\Lambda(\Acal_r)}$ is the stack of generalised effective Cartier divisors, i.e. pairs $(L,s)$ of a line bundle and a section. This is a sheaf of symmetric monoidal categories on the smooth site of $\Orb_\Lambda(\Acal_r)$.

The ghost sheaf $\Mbar$ stratifies $\Orb_\Lambda(\Acal_r)$. Recall from \cite[Definition~3.39]{HMOP} the notion of a \textbf{nuclear} open set $U \subseteq \Orb_\Lambda(\Acal_r)$. Restricting to a nuclear open ensures that there is a unique minimal stratum, and that the sheaf $\Mbar$ is constant. The nuclear opens form a base for the topology on $\Orb_\Lambda(\Acal_r)$ \cite[Lemma~3.40]{HMOP}. It thus suffices to construct \eqref{eqn: map ghost sheaf to divisor sheaf for Orb} on each nuclear open, and check that these maps are compatible under restriction of nuclear opens. We write a nuclear open as
\[ U_\uptau \subseteq \Orb_\Lambda(\Acal_r) \]
where $\uptau$ is the tropical type of map to $(\Rplus,r)$ which indexes the unique minimal stratum intersecting $U_\uptau$. The monoid $\Mbar(U_\uptau)$ is then the tropical monoid dual to the cone $\uptau$. This arises as a quotient of
\begin{equation} \label{eqn: prequotient monoid} \N^{E(\Gamma)} \oplus \N^{V_+(\Gamma)} \end{equation}
where $V_+(\Gamma) \subseteq V(\Gamma)$ are the vertices mapping to $\R_{>0}$. To each of the generators above, we will associate a generalised effective Cartier divisor on $U_\uptau$. We will then verify that this association respects the relations in the quotient monoid $\Mbar(U_\uptau)$.

First consider the $\N^{E(\Gamma)}$ factor. This is generated by the gerby edge lengths $\tilde{\ell}_e$. For each edge $e \in E(\Gamma)$ we consider the source graph obtained from $\Gamma$ by contracting every edge except $e$. This determines a Cartier divisor in $\Mfrak^{\tw}$ and we let
\[ Z_{e} \in \mathfrak{Div}_{\Orb_\Lambda(\Acal_r)}(U_\uptau) \]
denote the pullback of this Cartier divisor to $U_\uptau$. We define this to be the generalised effective Cartier divisor associated to $\tilde{\ell}_e$ under \eqref{eqn: map ghost sheaf to divisor sheaf for Orb}.

We now consider the $\N^{V_+(\Gamma)}$ factor. This is generated by the gerby vertex positions $\tilde\ftrop(v)$ for $v \in V_+(\Gamma)$. Each point $\upxi \in U_\uptau$ has an associated tropical type $\uptau(\upxi)$ with a specialisation of types $\uptau \rightsquigarrow \uptau(\upxi)$ obtained by setting some edge lengths and vertex positions to zero. The vertex $v \in V(\Gamma)$ specialises to a vertex $v \rightsquigarrow v(\upxi)$ and the latter corresponds to a component
\[ \Ccal_{v(\upxi)} \subseteq \Ccal|_{\upxi} \]
in the fibre of the universal curve $\Ccal \to U_\uptau$. We now choose a section
\[ p_v \colon U_\uptau \to \Ccal \]
such that $p_v(\upxi) \in \Ccal_{v(\upxi)}$ for all $\upxi \in U_\uptau$ and $p_v$ is disjoint from all special points. Such sections always exist locally on $U_\uptau$, which is sufficient for the construction.

We define the generalised effective Cartier divisor associated to $\tilde\ftrop(v)$ as:
\[ Z_v \colonequals p_v^\star (\Lcal,s) \in \mathfrak{Div}_{\Orb_\Lambda(\Acal_r)}(U_\uptau) \]
where $(\Lcal,s)$ is the universal line bundle and section on $\Ccal$. The following observation is crucial:
\begin{equation} \label{eqn: vanishing locus of xvs} \text{$p_v^\star (s)$ vanishes at  $\upxi$ if and only if $s$ vanishes identically on $\Ccal_{v(\upxi)}$}. \end{equation}
Clearly if $s$ vanishes identically on $\Ccal_{v(\upxi)}$ then $p_v^\star(s)$ vanishes at $\upxi$. Suppose conversely that $s$ does not vanish identically on $\Ccal_{v(\upxi)}$. We claim that $\Ccal_{v(\upxi)}$ cannot contain any high age markings or nodes. Otherwise, since $r$ is large, $\Lcal$ will push forward to a line bundle of negative degree on the coarse curve, and hence will have no sections (see e.g. \cite[Section~2.1.4]{JohnsonThesis}). It follows that $\Ccal_{v(\upxi)}$ only contains low age markings and nodes, and hence pushes forward to the trivial bundle on the coarse curve. In particular it has a unique nontrivial section cutting out the special points. Since $p_v$ is disjoint from the special points, it follows that $p_v^\star(s)$ does not vanish at $\upxi$.

The observation \eqref{eqn: vanishing locus of xvs} shows that the generalised effective Cartier divisor $Z_v = p_v^\star (\Lcal,s)$ is, up to isomorphism, independent of the choice of $p_v$.\smallskip

\noindent
\emph{Step 4: Consistency with relations.} We must now show that the above associations respect the relations in $\Mbar(U_\uptau)$. These relations are indexed by oriented edges $\vec{e} \in \vec{E}(\Gamma)$ with $\widetilde{m}_{\vec{e}} \geqslant 0$, and are given by
\[ \tilde\ftrop(v_1) + \widetilde{m}_{\vec{e}} \, \tilde\ell_e = \tilde\ftrop(v_2).\]
where $v_1$ and $v_2$ are the tail and head of $\vec{e}$. Here $\widetilde{m}_{\vec{e}} \in \N$ whilst $\tilde\ftrop(v_1), \tilde\ell_e, \tilde\ftrop(v_2)$ are generators of the monoid \eqref{eqn: prequotient monoid} (if $\upsigma_{v_i}=0$ then we set $\tilde\ftrop(v_i)=0$). We must show that
\[ Z_{v_1} + \widetilde{m}_{\vec{e}} Z_{e} = Z_{v_2} \]
on $U_\uptau$. We first note that for every point $\upxi \in U_\uptau \setminus Z_{e}$ we have $v_1(\upxi) = v_2(\upxi)$. It follows from \eqref{eqn: vanishing locus of xvs} that $Z_{v_1}=Z_{v_2}$ on $U_\uptau \setminus Z_e$. We conclude that
\[ Z_{v_1} + k  Z_e = Z_{v_2} \]
for some $k \in \Z$.\footnote{A priori, instead of $k Z_e$ we have an arbitrary weighted sum of the irreducible components of $Z_{e}$. However, the upcoming proof that $k=\widetilde{m}_{\vec{e}}$ applies equally to each of these components, so \emph{a posteriori} we find that the weights are uniform.} It remains to show that $k=\widetilde{m}_{\vec{e}}$. 

Restrict to a nuclear neighbourhood $U_\upepsilon$ of the generic point of $Z_e$. Since $Z_e$ is a divisor, the associated tropical type $\upepsilon$ is one-dimensional. From $\widetilde{m}_{\vec{e}} \geqslant 0$ and $\dim \upepsilon=1$ we obtain $\tilde\ftrop(v_1)=0$ and hence $p_{v_1}^\star (s)$ never vanishes on $U_\upepsilon$ by \eqref{eqn: vanishing locus of xvs}. It follows that $Z_{v_1}=0$ on $U_\upepsilon$ and we obtain:
\[ k Z_e = Z_{v_2}.\]
To find $k$ we must compare the vanishing orders of the node smoothing parameter and $p_{v_2}^\star (s)$. For this we further restrict from $U_\upepsilon$ to a curve in the moduli space
\[ \Spec \kfield \llbracket w \rrbracket \to U_\upepsilon \]
which intersects $Z_{e}$ transversely at the closed point. The universal curve pulls back to a family
\[ \uppi \colon \Ccal \to \Spec \kfield \llbracket w \rrbracket \]
whose total space is a smooth orbifold surface. We now examine relations between monomials on $\Ccal$. We ignore units throughout.

Restrict to a neighbourhood of the node $q_e=\Ccal_{v_1} \cap \Ccal_{v_2}$ in the surface $\Ccal$. On $\Ccal$ there are equations $z_1,z_2$ for the central fibre components $\Ccal_{v_1},\Ccal_{v_2}$ satisfying
\[ z_1 z_2 = \uppi^\star(w) \]
where $w$ is the equation of $Z_e$. Since $\tilde\ftrop(v_1) = 0$ we have $s = z_2^m$ for some $m \in \N$. Since $q_e \in \Ccal_{v_1}$ has gerby tangency order $\widetilde{m}_{\vec{e}}$ we have
\[ \widetilde{m}_{\vec{e}}/s_e = \Ccal_{v_1} \cdot \mathbb{V}(s) = \Ccal_{v_1} \cdot m \Ccal_{v_2} = m/s_e \]
which gives $m=\widetilde{m}_{\vec{e}}$ and thus: 
\[ s = z_2^{\widetilde{m}_{\vec{e}}}.\]
Recall that $p_{v_2}$ is a section passing through $\Ccal_{v_2}$ and disjoint from the special points. In particular $p_{v_2}^\star (z_1)$ is invertible, and so modulo units we have:
\[ p_{v_2}^\star(s) = p_{v_2}^\star(z_2^{\widetilde{m}_{\vec{e}}}) = p_{v_2}^\star (z_1^{\widetilde{m}_{\vec{e}}}z_2^{\widetilde{m}_{\vec{e}}}) = p_{v_2}^\star \uppi^\star (w^{\widetilde{m}_{\vec{e}}}) = w^{\widetilde{m}_{\vec{e}}}\]
which gives precisely the desired relation: $Z_{v_2} = \widetilde{m}_{\vec{e}}\, Z_{e}$.\smallskip

\noindent
\emph{Step 5: Compatibility under restriction.} We have successfully constructed the map \eqref{eqn: map ghost sheaf to divisor sheaf for Orb} on nuclear opens. We must now show that these maps are compatible under restriction. Consider an inclusion
\[ U_{\uptau^\prime} \subseteq U_\uptau\]
of nuclear opens. This induces a specialisation of tropical types $\uptau \rightsquigarrow \uptau^\prime$. We need to check that the following square is $2$-commutative:
\[
\begin{tikzcd}
\Mbar(U_\uptau) \ar[r] \ar[d] & \mathfrak{Div}(U_\uptau) \ar[d] \\
\Mbar(U_{\uptau^\prime}) \ar[r] & \mathfrak{Div}(U_{\uptau^\prime}).	
\end{tikzcd}
\]
It is sufficient to check this on generators of $\Mbar(U_\uptau)$. For the generators $\tilde\ell_e$ for $e \in E(\Gamma)$ this is clear. Consider $v \in V_+(\Gamma)$ specialising to $v^\prime \in V(\Gamma^\prime)$ and examine the associated generalised effective Cartier divisors $Z_v,Z_{v^\prime}$. We must show that
\begin{equation} \label{eqn: restriction of open sets} Z_v|_{U_{\uptau^\prime}} = Z_{v^\prime} \end{equation}
i.e. that the sections have the same scheme-theoretic vanishing locus. Note that we may have $v^\prime \not\in V_+(\Gamma^\prime)$, in which case $\tilde\ftrop(v^\prime)=0$ and $Z_{v^\prime} =0 $. This does not affect the following argument.

Recall that $Z_v = p_v^\star (\Lcal,s)$. By \eqref{eqn: vanishing locus of xvs}, the vanishing locus of $p_v^\star (s)$ is a union of closed strata in $U_\uptau$, corresponding to those tropical types for which the vertex $v(\upxi)$ is mapped to $\R_{>0}$. The same description applies to $p_{v^\prime}^\star (s)$ and we immediately obtain \eqref{eqn: restriction of open sets}.\medskip

\noindent
\emph{Step 6: Factorisation through puncturing substack.} We have completed the construction of the logarithmic structure on $\Orb_\Lambda(\Acal_r)$. This gives a morphism:
\begin{equation} \label{eqn: morphism Orb to Artin fan} \Orb_\Lambda(\Acal_r) \to \Acal(\Tsf_r). \end{equation}
We now argue that this factors through the puncturing substack $\Vcal(\Tsf_r)$. It suffices to prove this on each nuclear open $U_\uptau$. Given a puncture $p \in P$ the puncturing offset morphism
\[ \Acal(\Tsf_r) \to \Acal \]
gives the generalised effective Cartier divisor corresponding to the piecewise linear function $\tilde\ftrop(v)$ where $v \in V(\Gamma)$ is the vertex supporting the puncture $p$. We must show that the corresponding section $p_v^\star(s)$ vanishes identically. By \eqref{eqn: vanishing locus of xvs} it is equivalent to show that the section $s$ vanishes identically on the curve component $\Ccal_v$.

Here we use that $r$ is large and that $p$ is a marking of high age. As in Step 3 above, this implies that the line bundle $\Lcal|_{\Ccal_v}$ pushes forward to a line bundle on the coarse curve of negative degree, and hence has no sections. It follows that $s$ must vanish identically, as required. We conclude that \eqref{eqn: morphism Orb to Artin fan} factors through the closed substack
\[ \Orb_\Lambda(\Acal_r) \to \Vcal(\Tsf_r)\]
as required. We obtain a diagram
\bcd
\Orb_\Lambda(\Acal_r) \ar[rd,dashed] \ar[rrd,bend left] \ar[ddr, bend right] & & \\
& \PunctOrb_\Lambda(\Acal_r|\Dcal_r) \ar[r] \ar[d] \ar[rd,phantom,"\square"] & \Mfrak^{\tw} \ar[d] \\
& \Vcal(\mathsf{T}_r) \ar[r] & \Acal(\MTropOrb)
\ecd
with the induced map $\Orb_\Lambda(\Acal_r) \to \PunctOrb_\Lambda(\Acal_r|\Dcal_r)$ giving an inverse to $\upomega$. We conclude that $\upomega$ is an isomorphism.
\end{proof}

\subsection{Identifying virtual classes}\label{sec: comparing obstruction theories}

The above \Cref{thm: spaces same} establishes an identification of moduli spaces:
\[ \PunctOrb_\Lambda(\Acal_r|\Dcal_r) = \Orb_\Lambda(\Acal_r).\]
When we wish to avoid privileging one perspective over the other, we will use the notation:
\[ \Kup_\Lambda(\Acal_r|\Dcal_r) \colonequals \PunctOrb_\Lambda(\Acal_r|\Dcal_r) = \Orb_\Lambda(\Acal_r).\]
This space carries two Chow classes of the same dimension: the refined punctured class defined in \Cref{sec: refined punctured class}, and the orbifold virtual fundamental class. In this section we identify these (\Cref{thm: classes same}). Combined with \Cref{thm: punctured orbifold vs punctured}, this establishes \Cref{thm: correspondence smooth pairs} for universal targets.
\begin{theorem} \label{thm: classes same} There is an equality of classes
\begin{equation} \label{eqn: equality of classes} [\PunctOrb_\Lambda(\Acal_r|\Dcal_r)]^{\refined} = [\Orb_\Lambda(\Acal_r)]^{\virt}	\end{equation}
in the Chow homology of $\Kup_\Lambda(\Acal_r|\Dcal_r)$.
\end{theorem}

\subsubsection{Positivised numerical data} \label{sec: positivised numerical data} We begin by removing the tangency conditions at the punctures. This produces an auxiliary moduli space of stable logarithmic maps with positive tangency, which we use to relate the two classes.

\begin{definition} The \textbf{positivised numerical data} $\Lambda_+$ is obtained from the numerical data $\Lambda$ by making the following modifications:
\begin{itemize}
	\item For each puncture $p \in P$ we replace the contact order $\upalpha_p$ by $0$ and the rooting index $s_p$ by $1$.
	\item We replace the total degree $d$ by $d - \Sigma_{p \in P} \upalpha_p$.
\end{itemize}
Notice that the total degree increases. Here is an $r=1$ example, writing $\Lambda$ as $(g,d,(\upalpha_1,\ldots,\upalpha_n))$:
\[ \Lambda=(0,1,(4,-1,-2)) \rightsquigarrow \Lambda_+ = (0,4,(4,0,0)).\]
\end{definition}
The positivised numerical data defines a tropical moduli space as in \Cref{sec: tropical twisted punctured maps}. We temporarily denote this by $\mathsf{T}_r^+$.

\begin{lemma} \label{lem: positivised tropical moduli equals original tropical moduli} There is a natural isomorphism of cone stacks $\mathsf{T}_r=\mathsf{T}_r^+$.\end{lemma}
\begin{proof} Consider a tropical type indexing a cone in $\mathsf{T}_r$. We associate a tropical type indexing a cone in $\mathsf{T}_r^+$ by subtracting from each gerby vertex degree the sum of gerby tangencies at adjacent punctures divided by rooting parameters:
\[ \tilde{d}_v \mapsto \tilde{d}_v - \sum_{v \in p} \frac{\widetilde{\upalpha}_p}{s_p}.\]
Since $\widetilde{\upalpha}_p/s_p=\upalpha_p/r$, the new gerby degree still belongs to $\tfrac{1}{r}\Z$. The inverse association is defined analogously. An example with $r=1$ is illustrated below, with vertex degrees and edge slopes indicated in blue:
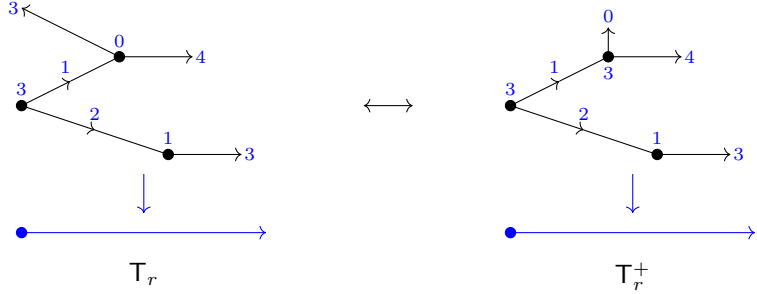
\[
\begin{tikzpicture}[scale=1.3]


\draw[fill=black] (-2,-1) circle[radius=1.5pt];
\draw[blue] (-2,-1) node[above]{\tiny$3$};

\draw[->] (-2,-1) -- (-1.5,-0.75);
\draw (-1.5,-0.75) -- (-1,-0.5);x
\draw[blue] (-1.55,-0.775) node[above]{\tiny$1$};

\draw[fill=black] (-1,-0.5) circle[radius=1.5pt];
\draw[blue] (-1,-0.5) node[above]{\tiny$0$};

\draw[->] (-1,-0.5) -- (-0.25,-0.5);
\draw[blue] (-0.325,-0.5) node[right]{\tiny$4$};

\draw[->] (-1,-0.5) -- (-2,0);
\draw[blue] (-1.925,0) node[left]{\tiny$3$};

\draw[->] (-2,-1) -- (-1.25,-1.25);
\draw (-1.25,-1.25) -- (-0.5,-1.5);
\draw[blue] (-1.25,-1.25) node[above]{\tiny$2$};

\draw[fill=black] (-0.5,-1.5) circle[radius=1.5pt];
\draw[blue] (-0.5,-1.5) node[above]{\tiny$1$};

\draw[->] (-0.5,-1.5) -- (0.25,-1.5);
\draw[blue] (0.175,-1.5) node[right]{\tiny$3$};

\draw[fill=blue,blue] (-2,-2.3) circle[radius=1.5pt];
\draw[blue,->] (-2,-2.3) -- (0.5,-2.3);

\draw[blue,->] (-0.75,-1.7) -- (-0.75,-2.1);

\draw (-0.75,-2.5) node[below]{\small$\Tsf_r$};


\draw[<->] (1.5,-1) -- (2,-1);



\draw[fill=black] (3,-1) circle[radius=1.5pt];
\draw[blue] (3,-1) node[above]{\tiny$3$};

\draw[->] (3,-1) -- (3.5,-0.75);
\draw (3.5,-0.75) -- (4,-0.5);
\draw[blue] (3.45,-0.775) node[above]{\tiny$1$};

\draw[fill=black] (4,-0.5) circle[radius=1.5pt];
\draw[blue] (4,-0.5) node[below]{\tiny$3$};

\draw[->] (4,-0.5) -- (4.75,-0.5);
\draw[blue] (4.675,-0.5) node[right]{\tiny$4$};

\draw[->] (4,-0.5) -- (4,-0.2);
\draw[blue] (4,-0.25) node[above]{\tiny$0$};

\draw[->] (3,-1) -- (3.75,-1.25);
\draw (3.75,-1.25) -- (4.5,-1.5);
\draw[blue] (3.75,-1.25) node[above]{\tiny$2$};

\draw[fill=black] (4.5,-1.5) circle[radius=1.5pt];
\draw[blue] (4.5,-1.5) node[above]{\tiny$1$};

\draw[->] (4.5,-1.5) -- (5.25,-1.5);
\draw[blue] (5.175,-1.5) node[right]{\tiny$3$};

\draw[fill=blue,blue] (3,-2.3) circle[radius=1.5pt];
\draw[blue,->] (3,-2.3) -- (5.5,-2.3);

\draw[blue,->] (4.25,-1.7) -- (4.25,-2.1);

\draw (4.25,-2.5) node[below]{\small$\Tsf_r^+$};


\end{tikzpicture}
\]
This establishes a bijection between tropical types, compatible with the specialisation ordering. This bijection extends to a natural isomorphism between cones, giving the desired isomorphism of cone stacks.
\end{proof}

\subsubsection{Negative inside positive}\label{sec: minus-inside-plus} The isomorphism of cone stacks established in \Cref{lem: positivised tropical moduli equals original tropical moduli} produces an isomorphism $\Acal(\mathsf{T}_r^+) = \Acal(\mathsf{T}_r)$. Crucially, since $\Lambda_+$ involves no punctures, the puncturing substack coincides with the Artin fan:
\[ \Vcal(\mathsf{T}_r^+) = \Acal(\mathsf{T}_r^+).\]
Combining with \Cref{lem: LogOrb as fibre product over twisted curves}, and noting that \Cref{thm: spaces same} applies equally to $\Lambda_+$ as to $\Lambda$, we obtain a cartesian square:
\begin{equation} \label{eqn: positive tangency auxiliary space}
\begin{tikzcd}
\Kup_{\Lambda_+}(\Acal_r|\Dcal_r) \ar[r] \ar[d] \ar[rd,phantom,"\square"] & \Mfrak^{\tw} \ar[d] \\
\Acal(\mathsf{T}_r) \ar[r] & \Acal(\MTropOrb).
\end{tikzcd}
\end{equation}
From this we obtain the following diagram
\[
\begin{tikzcd}
\Kup_\Lambda(\Acal_r|\Dcal_r) \ar[r,hook,"\uptheta"] \ar[d,"\upmu"] \ar[rd,phantom,"\text{A}"] & \Kup_{\Lambda_+}(\Acal_r|\Dcal_r) \ar[r] \ar[d,"\upnu"] \ar[rd,phantom,"\text{B}"] & \Mfrak^{\tw} \ar[d] \\
\Vcal(\mathsf{T}_r) \ar[r,hook] \ar[d,"\uplambda"] \ar[rd,phantom,"\square"] & \Acal(\mathsf{T}_r) \ar[r] \ar[d] & \Acal(\MTropOrb) \\
\Bcal \Gm^{P} \ar[r,hook,"\upiota"] & \Acal^{P}
\end{tikzcd}
\]
and since both B and A+B are cartesian, it follows that A is cartesian.

\begin{remark} \label{rmk: negative inside positive} The closed embedding
\[ \uptheta \colon \Kup_\Lambda(\Acal_r|\Dcal_r) \hookrightarrow \Kup_{\Lambda_+}(\Acal_r|\Dcal_r) \]
realises the negative contact space as a union of closed strata in the positive contact space. This union of closed strata is precisely the locus where every ``former'' puncture belongs to an irreducible component of the source curve on which the section vanishes identically. The closed embedding is given explicitly by retaining the marked source curve and modifying the line bundle as follows
\[ \Lcal \mapsto \Lcal \left(-\Sigma_{p \in P} \widetilde{\upalpha}_p \tilde{p} \right), \]
where $\tilde{p} \subseteq \Ccal$ are the gerby punctures (see \Cref{sec: description of POTs} below). The section lifts uniquely, since it vanishes identically on every component containing a puncture. This closed embedding of moduli spaces only exists at the level of universal targets.
\end{remark}

The vertical arrows $\upmu, \upnu$ are smooth since $\Mfrak^{\tw} \to \Acal(\MTropOrb)$ is smooth. It follows that $\Kup_{\Lambda_+}(\Acal_r|\Dcal_r)$ is irreducible, since it is smooth over an irreducible stack. From the definition of the refined punctured class we obtain:
\[
[\PunctOrb_\Lambda(\Acal_r|\Dcal_r)]^{\refined} = \upmu^\star \upiota^! [\Acal(\mathsf{T}_r)] = \upiota^! \upnu^\star [\Acal(\mathsf{T}_r)] = \upiota^! [\Kup_{\Lambda_+}(\Acal_r|\Dcal_r)].
\]
The regular embedding $\upiota$ naturally carries a perfect obstruction theory. This pulls back to a perfect obstruction theory for $\uptheta$ given by:
\begin{equation} \label{eqn: Etheta} \EE_\uptheta = \upmu^\star \uplambda^\star N_{\upiota}^\vee[1].\end{equation}
The associated virtual pullback coincides with the refined Gysin pullback $\upiota^!$ and we obtain:
\begin{equation} \label{eqn: pullback from positive to negative space} [\PunctOrb_\Lambda(\Acal_r|\Dcal_r)]^{\refined} = \uptheta^! [\Kup_{\Lambda_+}(\Acal_r|\Dcal_r)].\end{equation}
Now consider the triangle of spaces
\[
\begin{tikzcd}
\Kup_\Lambda(\Acal_r|\Dcal_r) \ar[r,hook,"\uptheta"] \ar[rr,bend right=15pt,"\upeta"] & \Kup_{\Lambda_+}(\Acal_r|\Dcal_r) \ar[r,"\uprho"] & \Mfrak^{\tw}(\Bcal \mathbb{G}_{\mathrm{m},r})
\end{tikzcd}
\]
where the morphisms to $\Mfrak^{\tw}(\Bcal \mathbb{G}_{\mathrm{m},r})$ are obtained from the morphism constructed in \Cref{lem: map LogOrb to Orb}. Viewing the first two spaces as moduli spaces of orbifold maps, both $\upeta$ and $\uprho$ carry perfect obstruction theories, such that
\begin{align}
\label{eqn: orb class as pullback from Pic} [\Orb_\Lambda(\Acal_r)]^{\virt} & = \upeta^![\Mfrak^{\tw}(\Bcal \mathbb{G}_{\mathrm{m},r})], \\
\label{eqn: orb positivised class as pullback from Pic} [\Orb_{\Lambda_+}(\Acal_r)]^{\virt} & = \uprho^![\Mfrak^{\tw}(\Bcal \mathbb{G}_{\mathrm{m},r})].
\end{align}
However, since the numerical data $\Lambda_+$ is positive and the genus is zero, the moduli space $\Orb_{\Lambda_+}(\Acal_r)$ is irreducible and contains the locus of maps with smooth source curve and nontrivial section as a dense open \cite[Lemma~4.2.3]{AbramovichCadmanWise}. On this locus the morphism
\[ \Orb_{\Lambda_+}(\Acal_r) \to \Mfrak^{\tw}(\Bcal \mathbb{G}_{\mathrm{m},r}) \]
is smooth, and the perfect obstruction theory coincides with the relative cotangent bundle. We conclude that the virtual fundamental class agrees with the fundamental class:
\[ [\Orb_{\Lambda_+}(\Acal_r)]^{\virt} = [\Kup_{\Lambda_+}(\Acal_r|\Dcal_r)].\]
Combining this with \eqref{eqn: pullback from positive to negative space} and \eqref{eqn: orb positivised class as pullback from Pic} we obtain
\begin{equation} \label{eqn: punctured class as pullback from Pic} [\PunctOrb_\Lambda(\Acal_r|\Dcal_r)]^{\refined} = \uptheta^!\uprho^! [\Mfrak^{\tw}(\Bcal \mathbb{G}_{\mathrm{m},r})].\end{equation}
From \eqref{eqn: orb class as pullback from Pic} and \eqref{eqn: punctured class as pullback from Pic}, we see that \Cref{thm: classes same} follows if we can show $\uptheta^! \uprho^! = \upeta^!$. We will achieve this by forming a compatible triple of obstruction theories \cite[Theorem~4.8]{ManolachePull}.

\subsubsection{Obstruction theories} \label{sec: description of POTs} We first describe $\EE_\upeta$ and $\EE_\uprho$. These are given by 
\begin{align} \label{eqn: POT for eta}
\EE_\upeta & = (\Rder \uppi_\star \Lcal)^\vee \qquad \\ 
\label{eqn: POT for rho} \uptheta^\star \EE_\uprho &= (\Rder \uppi_\star \Lcal_+)^\vee
\end{align}
where $\Lcal_+$ is the universal line bundle for $\Kup_{\Lambda_+}(\Acal_r|\Dcal_r)$. The bundles $\Lcal$ and $\Lcal_+$ are related as follows. Given a puncture $p \in P$ with gerby tangency $\widetilde\upalpha_p \in \Z$, we define
\[ \tilde{c}_p \colonequals -\widetilde\upalpha_p > 0\] 
and refer to this as the \textbf{multiplicity} of the puncture. We let $\tilde{p} \subseteq \Ccal$ denotes the gerby puncture on the universal curve. Consider the effective divisor
\[ \Dcal \colonequals \sum_{p \in P} \tilde{c}_p \tilde{p}.\]
Then we have
\[ \Lcal_+ = \Lcal(\Dcal)\]
from which we obtain a short exact sequence
\[ 0 \to \Lcal \to \Lcal_+ \to \Lcal_+|_\Dcal \to 0.\]
Applying $\Rder\uppi_\star$ we obtain the following exact triangle
\[ \Rder\uppi_\star \Lcal \to \Rder \uppi_\star \Lcal_+ \to \Rder \uppi_\star \Lcal_+|_D \xrightarrow{[1]} \]
The final term is supported on a zero-dimensional stack, and so
\begin{equation} \label{eqn: exact triangle pushforward}  \Rder \uppi_\star \Lcal_+|_D = \uppi_\star \Lcal_+|_D = \bigoplus_{p \in P} \uppi_\star ( \Lcal_+|_{\tilde{c}_p \tilde{p}}).\end{equation}

We now consider $\EE_\uptheta$. This is given by \eqref{eqn: Etheta}, and the vector bundle $\upmu^\star \uplambda^\star N_{\upiota}$ splits as a direct sum of line bundles associated to the gerby puncturing offsets, defined in \Cref{sec: tropical twisted punctured maps}. For simplicity write $S \colonequals \PunctOrb_\Lambda(\Acal_r|\Dcal_r)$.

\begin{proposition} \label{prop: line bundle associated to gerby puncturing offset} The line bundle associated to the gerby puncturing offset $\tilde\ftrop(p)$ is:
	\[ \OO_{S}(\, \tilde\ftrop(p)) = \uppi_\star (\Lcal_+|_{\tilde{p}}). \]
\end{proposition}
\begin{proof} 
We employ a technique for manipulating line bundles on logarithmic curves which has been used before in other contexts --- see e.g. \cite[Construction~5.4]{BarrottNabijou} and \cite[Theorem~4.2]{BNRGenus1} --- and that we refer to as the \textbf{there and back again} technique. The idea is to pull back the piecewise linear function to the universal curve, manipulate it there, and then pull or push it back down to the base of the family. 

We use the logarithmic structure on the universal twisted punctured curve (see \Cref{construction: universal family PunctOrb}). Consider the tropicalisation of the universal family:
\[
\begin{tikzcd}
\sqC \ar[d,"\ptrop"] \ar[r,"\tilde{\ftrop}"] & \Rplus \\
\Tsf_r.	
\end{tikzcd}
\]
We work locally on a nuclear open substack of $\PunctOrb_\Lambda(\Acal_r|\Dcal_r)$. This corresponds to restricting to a cone in $\Tsf_r$ indexed by a tropical type of twisted punctured map. Let $v \in V(\Gamma)$ be the vertex supporting the puncturing $p$. Consider the piecewise linear function
\[ \ptrop^\star \tilde{\ftrop}(p) - \tilde{\ftrop}^\star(1) \]
on $\sqC$. This vanishes identically on $v$ and has slope $-\widetilde{\upalpha}_i$ along every adjacent marking, and $-\widetilde{m}_{\vec{e}}$ along every adjacent edge. Passing to the associated line bundles, we obtain (see e.g. \cite[Proposition~2.4.1]{RanganathanSantosParkerWise1})
\[ (\uppi^\star \OO_S(\, \tilde\ftrop(p)) \otimes \Lcal^{-1})|_{\Ccal_v} = \OO_{\Ccal_v} \left(\sum_{v\in i} (-\widetilde{\upalpha}_i) \tilde{p}_i + \sum_{v \in e} (-\widetilde{m}_{\vec{e}}) \tilde{q}_e \right)
\]
where the sum is over markings and edges adjacent to $v$. Rearranging and restricting to the puncturing divisor $\tilde{p}$ we obtain:
\begin{equation} \label{eqn: gerby puncturing offset relation in universal curve} \uppi^\star \OO_S(\, \tilde\ftrop(p))|_{\tilde{p}} = \Lcal(\tilde{c}_p \tilde{p})|_{\tilde{p}} = \Lcal_+|_{\tilde{p}}.\end{equation}
Finally the composite $\tilde{p} \hookrightarrow \Ccal \to S$ is a gerbe banded by $\upmu_{s_p}$. Given a line bundle $L$ on $S$ we have $\operatorname{age}_{\tilde{p}} \uppi^\star L=0$ and so $\uppi_\star \uppi^\star (L|_{\tilde{p}})=L$. In our setting we have $L=\OO_S(\tilde\ftrop(p))$. The vanishing of the age can also be verified directly, as follows:
\begin{align*} \operatorname{age}_{\tilde{p}} \Lcal_+ & = \operatorname{age}_{\tilde{p}} \Lcal + \operatorname{age}_{\tilde{p}}(\OO_\Ccal(D)) && \text{(mod $1$)} \\
&= \operatorname{age}_{\tilde{p}} \Lcal + \operatorname{age}_{\tilde{p}} \OO_\Ccal(\tilde{c}_p \tilde{p}) && \text{(mod $1$)} \\
& = (1 - \tilde{c}_p/s_p) + \tilde{c}_p/s_p && \text{(mod $1$)} \\
& = 0 && \text{(mod $1$).}
\end{align*}
Applying $\uppi_\star$ to \eqref{eqn: gerby puncturing offset relation in universal curve} we obtain the desired identity:
\[\OO_S(\, \tilde\ftrop(p)) = \uppi_\star (\uppi^\star \OO_S(\, \tilde\ftrop(p)))|_{\tilde{p}} = \uppi_\star (\Lcal_+|_{\tilde{p}}).\qedhere\]
\end{proof}

From the previous proposition and \eqref{eqn: Etheta} we obtain:
\begin{equation} \label{eqn: Etheta in terms of PL functions}  \EE_\uptheta = \bigoplus_{p \in P} \uppi_\star(\Lcal_+|_{\tilde{p}})^\vee[1].\end{equation}
We now relate this to \eqref{eqn: exact triangle pushforward}. We require the following:

\begin{lemma} We have $\uppi_\star (\Lcal_+|_{\tilde{c}_p \tilde{p}}) = \uppi_\star (\Lcal_+|_{\tilde{p}})$.\end{lemma}
\begin{proof}
We argue locally. Restrict to an affine open $\Spec R \subseteq S$. Local to the gerby divisor $\tilde{p}$, the universal curve $\Ccal \to \Spec R$ is given by
\[ \Ccal = [\Aaff^{\!1}_R/\upmu_{s_p}] \]
where the action $\upmu_{s_p} \acts \Aaff^{\! 1}_R = \Aaff^{\! 1}_\Z \times \Spec R$ is the standard action on the first factor and the trivial action on the second factor. Letting $t$ denote the coordinate on the affine line, we have
\begin{align*}
\tilde{p} & = [\operatorname{Spec}R/\upmu_{s_p}], \\
\tilde{c}_p \tilde{p} & = [\left(\operatorname{Spec}R[t]/(t^{s_p}=0)\right)/\upmu_{s_p}].
\end{align*}
We view $\Lcal_+|_{\tilde{c}_p \tilde{p}}$ and $\Lcal_+|_{\tilde{p}}$ as $\upmu_{s_p}$-equivariant modules over the appropriate coordinate rings. The pushforward is obtained by calculating the invariant sections. We saw in the proof of \Cref{prop: line bundle associated to gerby puncturing offset} that the action on $\Lcal_+|_{\tilde{p}}$ is trivial. On $\tilde{c}_p \tilde{p}$ the coordinate ring splits into one-dimensional eigenspaces generated by the monomials
\[ t^0, t^1, \ldots, t^{\tilde{c}_p-1}.\]
Since we assume $r > |\upalpha_p|=-\upalpha_p$ (\Cref{def: rooting data}), it follows that $s_p > -\upalpha_p s_p/r = -\tilde{\upalpha_p} = \tilde{c}_p$. Hence the action is nontrivial on all monomials except $t^0=1$. It follows that the invariant sections over $\tilde{c}_p \tilde{p}$ and $\tilde{p}$ are both generated by $1$.
\end{proof}
From the previous lemma and \eqref{eqn: exact triangle pushforward} we obtain the exact triangle:
\begin{equation*} \Rder\uppi_\star \Lcal \to \Rder \uppi_\star \Lcal_+ \to \bigoplus_{p \in P} \uppi_\star (\Lcal_+|_{\tilde{p}}) \xrightarrow{[1]}. \end{equation*}
Dualising, shifting, and combining with \eqref{eqn: POT for eta}, \eqref{eqn: POT for rho}, \eqref{eqn: Etheta in terms of PL functions}, we obtain the compatible triple
\[ \uptheta^\star \EE_\uprho \to \EE_\upeta \to \EE_\uptheta \xrightarrow{[1]} \]
as required. This completes the proof of \Cref{thm: classes same}.

\subsection{General targets} \label{sec: correspondence smooth pairs general targets} Theorems~\ref{thm: punctured orbifold vs punctured}, \ref{thm: spaces same}, \ref{thm: classes same} establish \Cref{thm: correspondence smooth pairs} for universal targets. A straightforward comparison of obstruction theories now establishes the result for general targets:

\begin{proof}[Proof of \Cref{thm: correspondence smooth pairs}] By definition there is a cartesian square
\[
\begin{tikzcd}
\PunctOrb_\Lambda(X_r|D_r) \ar[r] \ar[d,"\uppsi_r"] \ar[rd,phantom,"\square"] & \Punct_\Lambda(X|D) \ar[d,"\uppsi"] \\
\PunctOrb_\Lambda(\Acal_r|\Dcal_r) \ar[r,"\upalpha"] & \Punct_\Lambda(\Acal|\Dcal).
\end{tikzcd}
\]
The vertical morphisms carry compatible perfect obstruction theories. By the commutativity of virtual pullbacks and proper pushforwards \cite[Theorem~4.1]{ManolachePull}, and the comparison for universal targets (\Cref{thm: punctured orbifold vs punctured}), we obtain
\[ \upalpha_\star[\PunctOrb_\Lambda(X_r|D_r)]^{\refined} = r^{-|P|}  [\Punct_\Lambda(X|D)]^{\refined}  \]
as required. For the other identity, first consider the diagram
\[
\begin{tikzcd}
\PunctOrb_\Lambda(X_r|D_r) \ar[r] \ar[d,"\uppsi_r"] \ar[rd,phantom,"\text{A}"] & \Punct_\Lambda(X|D) \ar[d,"\uppsi"] \ar[r] \ar[rd,phantom,"\text{B}"] & \Msf_{0,n,\upbeta}(X) \ar[d] \\
\PunctOrb_\Lambda(\Acal_r|\Dcal_r) \ar[r,"g"] & \Punct_\Lambda(\Acal|\Dcal) \ar[r] & \Mfrak(\Acal)
\end{tikzcd}
\]
where the moduli spaces in the third column parametrise maps with no logarithmic or orbifold structure. The square A is cartesian by definition, and B is cartesian by inspection. We conclude that A+B is cartesian. On the other hand, the following square is also cartesian by inspection:
\[
\begin{tikzcd}
\Orb_\Lambda(X_r) \ar[r] \ar[d,"\uptheta_r"] \ar[rd,phantom,"\square"] & \Msf_{0,n,\upbeta}(X) \ar[d] \\
\Orb_\Lambda(\Acal_r) \ar[r] & \Mfrak(\Acal).	
\end{tikzcd}
\]
By \cite[Proposition~3.1.1]{AbramovichMarcusWise} and \cite[Section~5.2]{AbramovichCadmanWise} the vertical morphisms carry compatible perfect obstruction theories given by $(\Rder \uppi_\star f^\star T_{X|D})^\vee$ and satisfying 
\[ [\Orb_\Lambda(X_r)]^{\virt} = \uptheta_r^![\Orb_\Lambda(\Acal_r)]^{\virt}.\]
Combining these cartesian diagrams with \Cref{thm: spaces same} we obtain
\[
\begin{tikzcd}
\PunctOrb_\Lambda(X_r|D_r) \ar[r,"\cong"] \ar[d,"\uppsi_r"] \ar[rd,phantom,"\square"] & \Orb_\Lambda(X_r) \ar[d,"\uptheta_r"] \\
\PunctOrb_\Lambda(\Acal_r|\Dcal_r) \ar[r,"\cong"] & \Orb_\Lambda(\Acal_r).	
\end{tikzcd}
\]
The vertical morphisms carry obstruction theories, and by \eqref{eqn: obstruction theory chimera} we see that these are compatible and hence induce the same virtual pullback. The desired identity
\[ [\PunctOrb_\Lambda(X_r|D_r)]^{\refined} = [\Orb_\Lambda(X_r)]^{\virt} \]
follows immediately from the result for universal targets (\Cref{thm: classes same}).
\end{proof}

\section{Proof: normal crossings pairs} \label{sec: proof snc}

\noindent We now establish the comparison for simple normal crossings pairs, completing the main theorem.

\begin{theorem}[\Cref{thm: main snc pairs introduction}] \label{thm: correspondence snc} Consider a simple normal crossings pair $(X|D)$ and numerical data $\Lambda$ for a moduli space of stable punctured maps in genus zero. There is a diagram of moduli spaces
\[
\begin{tikzcd}[column sep=small]
\Punct_\Lambda(X|D) \ar[rd,"\uptheta" below] && \NPunctOrb_\Lambda(X_{\bm{r}}|D_{\bm{r}}) \ar[ld,"\upalpha"] \ar[rd,"\upomega" below] && \Orb_\Lambda(X_{\bm{r}}) \ar[ld,"\upeta"] \\
& \NPunct_\Lambda(X|D) && \NOrb_\Lambda(X_{\bm{r}})	
\end{tikzcd}
\]
with all arrows proper, such that $\upalpha$ is virtually of degree $\Pi_{j=1}^k r_j^{-n_j}$ and $\upomega$ and $\upeta$ are virtually birational. If moreover $(X|D)$ is $\Lambda$-sensitive, then $\uptheta$ is virtually birational.
\end{theorem}
We employ rank reduction to deduce this result from the case of smooth pairs established in the previous section. The majority of the proof is an elaborate diagram chase. The novel geometric content is \Cref{thm: naive punctured versus punctured}, which hinges on the tropical insights of \cite{BNR2}.

In Sections~\ref{sec: snc proof universal space} and \ref{sec: snc proof punctured versus naive punctured} we establish the result for universal targets, and in Section~\ref{sec: snc proof general targets} we extrapolate to general targets. We freely employ the notation established in Sections~\ref{sec: new theory of negative tangency} and \ref{sec: correspondence}.

\subsection{Universal naive spaces} \label{sec: snc proof universal space} This section closely mirrors \cite[Section~2]{BNR2}. Fix a simple normal crossings pair $(X|D=D_1+\ldots+D_k)$ and numerical data $\Lambda$ for a moduli space of stable punctured maps. This induces numerical data $\Lambda_j$ for each of the smooth pairs $(X|D_j)$. As in \Cref{sec: new theory of negative tangency} we employ $(\Acal^k | \partial \Acal^k)$ as a convenient proxy for the universal target. We define universal naive spaces by fibring over the moduli space of curves:
\[
\begin{tikzcd}
\NPunct_\Lambda(\Acal^k|\partial \Acal^k) \ar[r] \ar[d] \ar[rd,phantom,"\square" left] & \prod_{j=1}^k \Punct_{\Lambda_j} (\Acal|\Dcal) \ar[d] & \NOrb_\Lambda(\Acal^k_{\bm{r}}) \ar[r] \ar[d] \ar[rd,phantom,"\square" left] & \prod_{j=1}^k \Orb_{\Lambda_j}(\Acal_{r_j}) \ar[d] \\
\Mfrak \ar[r] & \Mfrak^k & \Mfrak \ar[r] & \Mfrak^k \\
& \NPunctOrb_\Lambda(\Acal_{\bm{r}}^k|\partial \Acal_{\bm{r}}^k) \ar[r] \ar[d] \ar[rd,phantom,"\square" left] & \prod_{j=1}^k \PunctOrb_{\Lambda_j}(\Acal_{r_j}|\Dcal_{r_j}) \ar[d] \\
& \Mfrak \ar[r] & \Mfrak^k
\end{tikzcd}
\]
Each universal space is equipped with a virtual class, obtained by applying diagonal pullback to the external product of virtual classes on the spaces of maps to the smooth pair. A diagram chase produces cartesian towers
\[
\begin{tikzcd}[column sep=tiny]
\NPunctOrb_\Lambda(\Acal_{\bm{r}}^k|\partial \Acal_{\bm{r}}^k) \ar[r] \ar[d,"\upalpha"] \ar[rd,phantom,"\square"] & \prod_{j=1}^k \PunctOrb_{\Lambda_j}(\Acal_{r_j}|\Dcal_{r_j}) \ar[d] & \NPunctOrb_\Lambda(\Acal_{\bm{r}}^k|\partial \Acal_{\bm{r}}^k) \ar[r] \ar[d,"\upomega"] \ar[rd,phantom,"\square"] & \prod_{j=1}^k \PunctOrb_{\Lambda_j}(\Acal_{r_j}|\Dcal_{r_j}) \ar[d] \\
\NPunct_\Lambda(\Acal^k|\partial \Acal^k) \ar[r] \ar[d] \ar[rd,phantom,"\square" left] & \prod_{j=1}^k \Punct_{\Lambda_j}(\Acal|\Dcal) \ar[d] & \NOrb_\Lambda(\Acal_{\bm{r}}^k) \ar[r] \ar[d] \ar[rd,phantom,"\square" left] & \prod_{j=1}^k \Orb_{\Lambda_j}(\Acal_{r_j}) \ar[d] \\
\Mfrak \ar[r] & \Mfrak^k & \Mfrak \ar[r] & \Mfrak^k	
\end{tikzcd}
\]
From Theorems~\ref{thm: punctured orbifold vs punctured} and \ref{thm: classes same}, and the commutativity of Gysin pullback with proper pushforward, we see that $\upalpha$ has virtual degree $\Pi_{j=1}^k r_j^{-n_j}$ and that $\upomega$ is virtually birational. On the other hand, the product formula for orbifold targets shows that the morphism
\[ \upeta \colon \Orb_\Lambda(\Acal_{\bm{r}}^k) \to \NOrb_\Lambda(\Acal_{\bm{r}}^k) \]
is virtually birational \cite{AJTProducts} (the product formula is typically stated for Deligne--Mumford targets and fibring over stable curves, but the arguments apply unchanged to our setting).

\subsection{Naive punctured versus punctured} \label{sec: snc proof punctured versus naive punctured} It remains to compare the punctured theory with the naive punctured theory. This is where $\Lambda$-sensitivity enters. There is a natural proper morphism
\[ \uptheta \colon \Punct_\Lambda(\Acal^k|\partial \Acal^k) \to \NPunct_\Lambda(\Acal^k|\partial \Acal^k) \]
and Theorem~\ref{thm: correspondence snc} for universal targets follows from:
\begin{theorem} \label{thm: naive punctured versus punctured} Suppose that $(\Acal^k|\partial \Acal^k)$ is $\Lambda$-sensitive. Then $\uptheta$ is virtually birational.
\end{theorem}
The special case of positive tangency constitutes the key rank reduction step in \cite{BNR2}. By positivising the numerical data, we will deduce the general case from this one:
\begin{theorem}[\!{\cite[Theorem~5.1]{BNR2}}] \label{thm: rank reduction positive tangency} Suppose that $(\Acal^k|\partial \Acal^k)$ is $\Lambda$-sensitive and that $\Lambda$ includes no punctures. Then $\uptheta$ is a birational morphism between irreducible stacks.
\end{theorem}

We build to the proof of \Cref{thm: naive punctured versus punctured}. For the rest of this section we assume that $(\Acal^k |\partial \Acal^k)$ is $\Lambda$-sensitive. As in \Cref{sec: tropical maps}, the numerical data $\Lambda, \Lambda_1, \ldots, \Lambda_k$ define tropical moduli spaces, which we denote $\Tsf_\Lambda, \Tsf_{\Lambda_1}, \ldots, \Tsf_{\Lambda_k}$ respectively. We denote and define the \textbf{naive Artin fan}:
\[ \NAcal(\Tsf_\Lambda) \colonequals \left( \prod_{j=1}^k \Acal(\Tsf_{\Lambda_j}) \right) \times_{\Acal(\MTrop)^k} \Acal(\MTrop). \]
Directly examining the tropical moduli, we find an isomorphism of cone stacks
\[ \Tsf_\Lambda = \left( \prod_{j=1}^k \Tsf_{\Lambda_j} \right) \times_{\MTrop^k} \MTrop\]
from which we obtain a morphism of logarithmic algebraic stacks
\begin{equation} \label{eqn: map Artin fan to fibre product of Artin fans} \Acal(\Tsf_\Lambda) \to \NAcal(\Tsf_\Lambda). \end{equation}

\begin{proposition} \label{prop: Artin and naive Artin birational} The morphism \eqref{eqn: map Artin fan to fibre product of Artin fans} is proper and birational.
\end{proposition}

\begin{proof} By construction this morphism is the integralisation and saturation of the codomain \cite{Molcho}. We will prove that it is surjective. This implies that the codomain is integral, and it follows that the morphism is at worst a saturation. In particular, it is proper and birational.

To see that \eqref{eqn: map Artin fan to fibre product of Artin fans} is surjective, we pass to the positivised numerical data as in \Cref{sec: positivised numerical data}. This does not change the Artin fans by \Cref{lem: positivised tropical moduli equals original tropical moduli}, but it does change the moduli spaces of maps. Consider:
\[
\begin{tikzcd}
\NLog_{\Lambda_+}(\Acal^k | \partial \Acal^k) \ar[r] \ar[d] \ar[rd,phantom,"\text{A}"] & \prod_{j=1}^k \Log_{\Lambda_{j+}}(\Acal|\Dcal) \ar[r] \ar[d] \ar[rd,phantom,"\text{B}"] & \prod_{j=1}^k \Acal(\Tsf_{\Lambda_j}) \ar[d] \\
\Mfrak \ar[r] & \Mfrak^k \ar[r] & \Acal(\MTrop)^k.
\end{tikzcd}
\]
Here A is cartesian by definition of the naive space, while B is cartesian by \eqref{eqn: positive tangency auxiliary space}. We conclude that A+B is cartesian. Now consider:
\[
\begin{tikzcd}
\NLog_{\Lambda_+}(\Acal^k | \partial \Acal^k) \ar[r] \ar[d] \ar[rd,phantom,"\text{A'}"] & \NAcal(\Tsf_\Lambda) \ar[r] \ar[d] \ar[rd,phantom,"\text{B'}"] & \prod_{j=1}^k \Acal(\Tsf_{\Lambda_j}) \ar[d] \\
\Mfrak \ar[r] & \Acal(\MTrop) \ar[r] & \Acal(\MTrop)^k.
\end{tikzcd}
\]
Here A'+B'=A+B is cartesian by the above, while B' is cartesian by definition of the naive Artin fan. We conclude that A' is cartesian. Finally, consider:
\[
\begin{tikzcd}
\Log_{\Lambda_+}(\Acal^k | \partial \Acal^k) \ar[r,"\uptheta_+"] \ar[d] \ar[rd,phantom,"\text{A''}"] & \NLog_{\Lambda_+}(\Acal^k | \partial \Acal^k) \ar[r] \ar[d,"\upphi"] \ar[rd,phantom,"\text{B''}"] & \Mfrak \ar[d,"\uppsi"] \\
\Acal(\Tsf) \ar[r,"\upmu"] & \NAcal(\Tsf_\Lambda) \ar[r] & \Acal(\MTrop).
\end{tikzcd}
\]
Here B''=A' is cartesian by the above, while A''+B'' is cartesian by \eqref{eqn: positive tangency auxiliary space}. We conclude that A'' is cartesian. The morphism $\upphi$ is smooth and surjective, since the same is true for $\uppsi$.

On the other hand, the assumption that $(\Acal^k | \partial \Acal^k)$ is $\Lambda$-sensitive implies that it is also $\Lambda_+$-sensitive, as the collection of slopes appearing in the tropical moduli spaces are identical. \Cref{thm: rank reduction positive tangency} applies, and we conclude that $\uptheta_+$ is surjective (this is the only place in this section where we use the slope-sensitivity assumption). It follow that $\upphi \circ \uptheta_+$ is surjective, and hence $\upmu$ is surjective as required.
\end{proof}

Recall the total puncturing rank $k_P$ from \Cref{def: puncturing rank}, and consider the following diagram. The \textbf{naive puncturing substack} $\NVcal(\Tsf_\Lambda) \hookrightarrow \NAcal(\Tsf_\Lambda)$ is defined in order to make B cartesian.
\begin{equation} \label{eqn: diagram all puncturing substacks}
\begin{tikzcd}
\Vcal(\Tsf_\Lambda) \ar[r,hook] \ar[d,"\upnu"] \ar[rd,phantom,"\text{A}"] & \Acal(\Tsf_\Lambda) \ar[d,"\upalpha"] \\
\NVcal(\Tsf_\Lambda) \ar[r,hook] \ar[d] \ar[rd,phantom,"\text{B}"] & \NAcal(\Tsf_\Lambda) \ar[r] \ar[d] \ar[rd,phantom,"\text{D}"] & \Acal(\MTrop) \ar[d,"\updelta"] \\
\prod_{j=1}^k \Vcal(\Tsf_{\Lambda_j}) \ar[r,hook] \ar[d] \ar[rd,phantom,"\text{C}" left] & \prod_{j=1}^k \Acal(\Tsf_{\Lambda_j})  \ar[r] \ar[d]  & \Acal(\MTrop)^k \\
\Bcal \Gm^{k_P} \ar[r,hook,"\upiota"] & \Acal^{k_P}
\end{tikzcd}
\end{equation}
Here B and D are cartesian by the definitions of $\NVcal(\Tsf_\Lambda)$ and $\NAcal(\Tsf_\Lambda)$. We conclude that B+D is cartesian. Moreover D is a cartesian diagram of irreducible stacks, and $\updelta$ is a regular embedding of codimension zero which pulls back to a closed embedding, also of codimension zero. We conclude:
\begin{equation} \label{eqn: pullback product of Artin to naive Artin}
[\NAcal(\Tsf_\Lambda)] = \updelta^! \big(\Pi_{j=1}^k [\Acal(\Tsf_{\Lambda_j})]\big).
\end{equation}
We also have C cartesian by the definition of the $\Vcal(\Tsf_{\Lambda_j})$. We conclude that B+C is cartesian. On the other hand A+B+C is cartesian by the definition of $\Vcal(\Tsf_\Lambda)$. We conclude that A is cartesian, and:
\[
\upnu_\star[\Vcal(\Tsf_\Lambda)]^{\refined} = \upnu_\star \upiota^! [\Acal(\Tsf_\Lambda)] = \upiota^! \upalpha_\star [ \Acal(\Tsf_\Lambda)].
\]
Combining \Cref{prop: Artin and naive Artin birational} and \eqref{eqn: pullback product of Artin to naive Artin} gives $\upalpha_\star [ \Acal(\Tsf_\Lambda)] = [\NAcal(\Tsf_\Lambda)]=\updelta^! \big(\Pi_{j=1}^k [\Acal(\Tsf_{\Lambda_j})]\big)$. We obtain:
\begin{equation} \label{eqn: pushforward puncturing substack as diagonal pullback}
\upnu_\star[\Vcal(\Tsf_\Lambda)]^{\refined} = \upiota^! \updelta^! \big( \Pi_{j=1}^k [\Acal(\Tsf_{\Lambda_j})] \big) = \updelta^! \upiota^!\big( \Pi_{j=1}^k [\Acal(\Tsf_{\Lambda_j})] \big) = \updelta^! \big( \Pi_{j=1}^k [\Vcal(\Tsf_{\Lambda_j})]^{\refined} \big).
\end{equation}

\begin{proof}[Proof~of~\Cref{thm: naive punctured versus punctured}] Consider the diagram:
\[
\begin{tikzcd}
\Punct_\Lambda(\Acal^k | \partial \Acal^k) \ar[r,"\uptheta"] \ar[dd,"\upeta"] \ar[rdd,phantom,"\text{E}"] & \NPunct_\Lambda(\Acal^k|\partial \Acal^k) \ar[rr] \ar[dd,"\upzeta"] \ar[rd] & & \prod_{j=1}^k \Punct_{\Lambda_j}(\Acal|\Dcal) \ar[dd] \ar[rd] \\
& & \Mfrak \ar[rr,crossing over,"\Delta" {xshift=-15pt}] & & \Mfrak^k \ar[dd,"\upphi"] \\
\Vcal(\Tsf_\Lambda) \ar[r,"\upnu"] & \NVcal(\Tsf_\Lambda) \ar[rr] \ar[rd] & & \prod_{j=1}^k \Vcal(\Tsf_{\Lambda_j}) \ar[rd] \\
& & \Acal(\MTrop) \ar[from=uu, crossing over,"\uppsi" {xshift=0pt,yshift=15pt}] \ar[rr,"\updelta"] & & \Acal(\MTrop)^k.
\end{tikzcd}
\]
In the cube, the following faces are cartesian:
\begin{itemize}
\item \textbf{Top.} By the definition of $\NPunct_\Lambda(\Acal^k | \partial \Acal^k)$. 
\item \textbf{Right.} By \eqref{eqn: cartesian square punctured maps over prestable curves}.
\item \textbf{Bottom.} By B+D in \eqref{eqn: diagram all puncturing substacks}.
\end{itemize}
We conclude that the left face in the cube is also cartesian. The composition of the left face with E is cartesian by \eqref{eqn: cartesian square punctured maps over prestable curves}. It follows that E is cartesian. The vertical morphisms are smooth, giving
\[ \uptheta_\star[\Punct_\Lambda(\Acal^k | \partial \Acal^k)]^{\refined} = \uptheta_\star \upeta^\star [\Vcal(\Tsf_\Lambda)]^{\refined} = \upzeta^\star \upnu_\star [ \Vcal(\Tsf_\Lambda)]^{\refined} \]
and applying \eqref{eqn: pushforward puncturing substack as diagonal pullback} we obtain
\[ \uptheta_\star[\Punct_\Lambda(\Acal^k | \partial \Acal^k)]^{\refined} = \upzeta^\star \updelta^! \big( \Pi_{j=1}^k [\Vcal(\Tsf_{\Lambda_j})]^{\refined} \big).\]
Consider the front face of the cube. This is not cartesian, but each morphism is lci and hence induces a refined pullback on any overlying cartesian square \cite[Section~6.6]{FultonBig}. These pullbacks commute:
\[ \uppsi^! \circ \updelta^! = \Delta^! \circ \upphi^!.\]
The morphisms $\upphi$ and $\uppsi$ in particular are smooth, so their refined pullbacks coincide with flat pullbacks after base change \cite[Proposition~6.6(b)]{FultonBig}. We obtain
\begin{align*} \uptheta_\star[\Punct_\Lambda(\Acal^k | \partial \Acal^k)]^{\refined} & = \upzeta^\star \updelta^! \big( \Pi_{j=1}^k [\Vcal(\Tsf_{\Lambda_j})]^{\refined} \big) = \uppsi^! \updelta^! \big( \Pi_{j=1}^k [\Vcal(\Tsf_{\Lambda_j})]^{\refined} \big) \\
& = \Delta^! \upphi^! \big( \Pi_{j=1}^k [\Vcal(\Tsf_{\Lambda_j})]^{\refined} \big) = \Delta^! \big( \Pi_{j=1}^k [\Punct_{\Lambda_j}(\Acal|\Dcal)]^{\refined} \big) \\
& = [\NPunct_\Lambda(\Acal^k| \partial \Acal^k)]^{\refined}
\end{align*}
as required.
\end{proof}

\subsection{General targets} \label{sec: snc proof general targets} Sections~\ref{sec: snc proof universal space} and \ref{sec: snc proof punctured versus naive punctured} establish \Cref{thm: correspondence snc} for the universal target $(\Acal^k | \partial \Acal^k)$. We now deduce the result for the general target $(X|D)$.

\begin{proof}[Proof~of~\Cref{thm: correspondence snc}]
 There are fibre squares:
\[
\begin{tikzcd}
\Punct_\Lambda(X|D) \ar[r] \ar[d] \ar[rd,phantom,"\square"] & \Msf_{0,n,\upbeta}(X) \ar[d] & \Orb_\Lambda(X_{\bm{r}}) \ar[r] \ar[d] \ar[rd,phantom,"\square"] & \Msf_{0,n,\upbeta}(X) \ar[d] \\
\Punct_\Lambda(\Acal^k | \partial \Acal^k) \ar[r] & \Mfrak(\Acal^k), & \Orb_\Lambda(\Acal^k_{\bm{r}}) \ar[r] & \Mfrak(\Acal^k).
\end{tikzcd}
\]
In each case, the vertical morphisms carry compatible obstruction theories given by $(\Rder \uppi_\star f^\star T_{X|D})^\vee$ such that virtual pullback of the refined, respectively virtual, class gives the refined, respectively virtual, class (see \cite[Proposition~3.1.1]{AbramovichMarcusWise}, \cite[Section~4.2]{PuncturedMaps}, \cite[Section~4.2]{AbramovichCadmanWise}). We use analogous fibre squares to define the naive spaces
\[ \NPunct_\Lambda(X|D), \quad \NPunctOrb_\Lambda(X_{\bm{r}}|D_{\bm{r}}), \quad \NOrb_\Lambda(X_{\bm{r}}) \]
and their virtual classes. Chasing through cartesian towers results in a diagram
\[
\begin{tikzcd}[column sep=small]
\Punct_\Lambda(X|D) \ar[rd] \ar[dd] && \NPunctOrb_\Lambda(X_{\bm{r}}|D_{\bm{r}}) \ar[ld] \ar[dd] \ar[rd] && \Orb_\Lambda(X_{\bm{r}}) \ar[ld] \ar[dd] \\
& \NPunct_\Lambda(X|D) \ar[dd] && \NOrb_\Lambda(X_{\bm{r}}) \ar[dd] 	\\
\Punct_\Lambda(\Acal^k | \partial \Acal^k) \ar[rd] && \NPunctOrb_\Lambda(\Acal^k_{\bm{r}} | \partial \Acal^k_{\bm{r}}) \ar[ld] \ar[rd] && \Orb_\Lambda(\Acal^k_{\bm{r}}) \ar[ld] \\
& \NPunct_\Lambda(\Acal^k | \partial \Acal^k) && \NOrb_\Lambda(\Acal^k_{\bm{r}})	
\end{tikzcd}
\]
in which all square are cartesian, and all vertical arrows carry mutually compatible perfect obstruction theories. \Cref{thm: correspondence snc} for $(X|D)$ now follows from \Cref{thm: correspondence snc} for $(\Acal^k | \partial \Acal^k)$, and the commutativity of virtual pullback with proper pushforward.
\end{proof}

\section{Birational invariance} \label{sec: birational invariance}

\noindent We investigate birational invariance of the refined punctured theory. This is more subtle than in the positive tangency setting, see \Cref{sec: birational invariance introduction}. We first identify a system of iterated blowups which stabilises the virtual dimension (\Cref{sec: stabilisation of virtual dimension}) We then show via a counterexample that this alone is insufficient to guarantee birational invariance (\Cref{sec: birational counterexample}). We leave a comprehensive study of the behaviour of the refined punctured theory under birational modifications to future work.

\subsection{Stabilisation of virtual dimension} \label{sec: stabilisation of virtual dimension} The dimension of the puncturing substack $\Vcal(\Tsf)$ changes under birational modification, in a manner which is difficult to predict. In contrast, we show that the dimension of the refined virtual class $[\Vcal(\Tsf)]^{\refined}$ stabilises after an initial sequence of blowups.

\begin{proposition} \label{lem: blowup to get constant puncturing rank} Fix a pair $(X|D)$ and numerical data $\Lambda$.
\begin{enumerate}
\item There exists an iterated blowup $(X^\prime|D^\prime) \to (X|D)$ and numerical data $\Lambda^\prime$ pushing forward to $\Lambda$, such that the puncturing rank of each marking is either zero or one:
\[ k_i^\prime  = \begin{cases} 0 \qquad & \text{if $i \in O^\prime$,} \\ 1 \qquad & \text{if $i \in P^\prime$.} \end{cases} \]
Moreover the sets of ordinary markings and punctures are unchanged: $O^\prime=O$ and $P^\prime=P$.
\item The above system is stable under further blowups: given any iterated blowup $(X^{\prime\prime}|D^{\prime\prime}) \to (X^\prime|D^\prime)$ there is a lift of the numerical data with the same properties.
\end{enumerate}
For any member of the above system, the class $[\Vcal(\Tsf)]^{\refined}$ has pure dimension $-k_P^\prime = -|P|$.
\end{proposition}

We delay the proof for a preliminary discussion concerning the behaviour of numerical data under iterated blowups. Given an iterated blowup, there are typically multiple lifts of the numerical data (\Cref{rmk: contact data lift not unique}). We begin by identifying a \textbf{faithful lift}, which will be used in the proof of \Cref{lem: blowup to get constant puncturing rank}.

Given a nonempty $J \subseteq [k]$ denote the corresponding stratum $D_J=\cap_{j \in J} D_j$ and consider the associated stratum blowup $(X^\prime|D^\prime) \to (X|D)$. We write $D_0^\prime$ for the exceptional divisor and $D_1^\prime,\ldots,D_k^\prime$ for the strict transforms. There is a lattice map recording pullbacks of boundary divisors
\begin{align}
\nonumber \Z^{k} & \longleftarrow \Z^{k+1} \\
\label{eqn: lattice map blowup} a_j & \mapsto \begin{cases} a_j^\prime + a_0^\prime \quad & \text{if $j \in J$} \\ a_j^\prime \qquad\qquad & \text{if $j \in [k] \setminus J$} \end{cases}
\end{align}
where $(a_j)_{j=1}^k$ and $(a_j^\prime)_{j=0}^k$ are the standard coordinates.

\begin{construction}[Faithful lift of the numerical data] \label{construction: faithful lift} Fix a marking $i \in [n]$. We construct a \textbf{faithful lift} of its tangencies along the lattice map \eqref{eqn: lattice map blowup}:
\[ (\upalpha_{ij}^\prime)_{j=0}^k \mapsto (\upalpha_{ij})_{j=1}^k.\]
Such a lift\footnote{Setting $\upalpha_{i0}^\prime \colonequals 0$ and $\upalpha_{ij}^\prime \colonequals \upalpha_{ij}$ for $j \in [k]$ is typically invalid: in the blowup we have $\cap_{j \in J} D_j^\prime = \emptyset$ and hence require $\upalpha_{ij}^\prime = 0$ for some $j \in J$.} will preserve the sets of ordinary markings and punctures, not increase the puncturing rank of any marking, and decrease the puncturing rank in a controlled manner. These properties are not automatic: carelessly chosen lifts can increase the puncturing rank, and transform an ordinary marking into a puncture. To construct a faithful lift, we distinguish two cases.\smallskip

\noindent \emph{Case 1.} Suppose there exists $j \in J$ such that $\upalpha_{ij} \geqslant 0$. Choose $l \in J$ such that
\[ \upalpha_{il} = \min \{ \upalpha_{ij} : j \in J \text{ and } \upalpha_{ij} \geqslant 0\}.\]
Then define lifted contact orders as follows:
\begin{align} \label{eqn: lifted contact order}
\upalpha_{ij}^\prime \colonequals \begin{cases} \upalpha_{il} \qquad & \text{if $j=0$} \\ \upalpha_{ij} - \upalpha_{il} \qquad & \text{if $j \in J$} \\ \upalpha_{ij} \qquad & \text{if $j \in [k] \setminus J$}.	
 \end{cases}
\end{align}
We have $(\upalpha_{ij}^\prime)_{j=0}^k \mapsto (\upalpha_{ij})_{j=1}^k$ and the lift is valid since $\upalpha_{il}^\prime=0$. The puncturing rank of $i$ does not change. In particular, if $i$ was an ordinary marking then it remains an ordinary marking after lifting, and if $i$ was a puncture then it remains a puncture after lifting. Note that the lift does not depend on the choice of $l$. This completes Case 1.\smallskip

\noindent \emph{Case 2.} It remains to consider the case $\upalpha_{ij} < 0$ for all $j \in J$. In this case we choose $l \in J$ such that
\[ \upalpha_{il} = \max \{ \upalpha_{ij} : j \in J \} \]
and then define lifted contact orders exactly as in \eqref{eqn: lifted contact order}. As in the previous case, this is a valid lift of the contact orders and does not depend on the choice of $l$. Unlike in the previous case, the puncturing rank can fall. Precisely, we have
\[ k_i^\prime = k_i + 1 - \left| \{ j \in J : \upalpha_{ij} = \upalpha_{il} \} \right|. \]
The final term is at least $1$ since $l \in J$. We conclude that $k_i^\prime \leqslant k_i$. On the other hand, $\upalpha_{i0}^\prime = \upalpha_{il} < 0$ implies $k_i^\prime \geqslant 1$. Summarising, we have
\[ 1 \leqslant k_i^\prime \leqslant k_i.\]
In particular $i$ remains a puncture after lifting, and if $k_i=1$ then $k_i^\prime=1$. While the puncturing rank is not forced to decrease, there is a finer quantity which must. We define the \textbf{puncturing multiplicity} $m_i$ to be the sum of the absolute values of the negative contact orders:
\[ m_i \colonequals -\sum_{\substack{j \in [k] \\ \upalpha_{ij} < 0}} \upalpha_{ij} \in \N.\]
It follows from \eqref{eqn: lifted contact order} that after lifting the contact orders we have
\begin{equation} \label{eqn: puncturing multiplicity change} m_i^\prime = m_i + (|J|-1)\upalpha_{il}. \end{equation}
Since $\upalpha_{il}<0$ we see that $|J| \geqslant 2$ implies $m_i^\prime < m_i$. This completes Case 2.

This lifts the contact orders. The curve class $\upbeta^\prime$ is determined by the global balancing equation
\[ D_0^\prime \cdot \upbeta^\prime = \Sigma_{i=0}^n \upalpha_{i0}^\prime \]
and the fact that $\upbeta^\prime \mapsto \upbeta$. We thus obtain the faithful lift $\Lambda^\prime$ of the numerical data $\Lambda$. \qed
\end{construction}

\begin{proof}[Proof~of~\Cref{lem: blowup to get constant puncturing rank}] 
We proceed inductively. At each step we select a puncture $p \in P$ with $k_p \geqslant 2$ and take
\[ J \colonequals \{ j \in [k] : \upalpha_{pj} < 0 \}.\]
We perform the associated stratum blowup and take the faithful lift of the tangencies. This preserves the sets of ordinary markings and punctures, and does not increase any puncturing rank. The chosen puncture $p$ belongs to Case 2 above, and $|J|=k_p \geqslant 2$ combined with \eqref{eqn: puncturing multiplicity change} ensures that the puncturing multiplicity decreases.

Repeating, we eventually obtain $|J^\prime| = k_p^\prime=1$ since the puncturing multiplicity cannot decrease indefinitely. Repeating for all punctures, we arrive at an iterated blowup with lifted numerical data, establishing (i). Given any further blowup, we then choose the faithful lift of the contact orders: this ensures that the desired properties continue to hold, establishing (ii).
\end{proof}

\subsection{Counterexample} \label{sec: birational counterexample}
Following \Cref{lem: blowup to get constant puncturing rank}, while initial blowups can change the puncturing rank and hence the virtual dimension of the refined punctured class, these quantities eventually stabilise as long as we lift the numerical data carefully. In this stable range, at least, we might hope for birational invariance.

However, this is too optimistic: we now present an example in the stable range which does not satisfy birational invariance. 

\begin{example} Let $X=\mathbb{F}_1$, $D_1$ a fibre of the bundle projection, and $D_2$ the $(-1)$-curve. Take $\upbeta=2D_2$ so that $D_1 \cdot \upbeta = 2$ and $D_2 \cdot \upbeta = -2$, and take $n=2$ with contact orders $\upalpha_1=\upalpha_2=(1,-1)$. This defines numerical data for stable punctured maps to $(X|D_1+D_2)$. The puncturing rank of each puncture is $1$, as in \Cref{lem: blowup to get constant puncturing rank}.

Fix a tropical type arising from a stable punctured map to $(X|D_1+D_2)$. The curve class must be entirely concentrated on the vertices mapping to $D_2$. This reduces to a small number of tropical types, and as in \Cref{example: higher rank regular embedding} we replace $\Tsf$ by the subcomplex spanned by these types. The resulting finite cone complex is illustrated in \Cref{fig: T}.
\begin{figure}
\begin{tikzpicture}[scale=1.3]

\draw[fill=gray,gray] (0,0) circle[radius=1.5pt];
\draw[gray,thick,->] (0,0) -- (4,0);
\draw[gray] (4,0) node[right]{$Z_3$};
\draw[gray,thick,->] (0,0) -- (0,4);
\draw[gray] (0,4) node[above]{$Z_1$};
\draw[gray,thick,->] (0,0) -- (-4,0);
\draw[gray] (-4,0) node[left]{$Z_2$};

\draw[black,fill=black] (-0.5,4.75) circle[radius=1pt];
\draw[black,->] (-0.5,4.75) -- (0.5,4.75);
\draw (0.4,4.75) node[right]{\tiny$D_1$};
\draw[black,->] (-0.5,4.75) -- (-0.5,5.75);
\draw (-0.5,5.71) node[above]{\tiny$D_2$};
\draw[blue,fill=blue] (-0.5,5.4) circle[radius=1pt];
\draw[blue] (-0.43,5.4) node[left]{\tiny$2D_2$};
\draw[blue,-stealth] (-0.5,5.435) -- (0.185,4.75);
\draw[blue,-stealth] (-0.5,5.365) -- (0.115,4.75);

\draw[black,fill=black] (-3.2,2.45) circle[radius=1pt];
\draw[black,->] (-3.2,2.45) -- (-2.2,2.45);
\draw (-2.3,2.45) node[right]{\tiny$D_1$};
\draw[black,->] (-3.2,2.45) -- (-3.2,3.45);
\draw (-3.2,3.41) node[above]{\tiny$D_2$};
\draw[blue,fill=blue] (-3.2,3.1) circle[radius=1pt];
\draw[blue,fill=blue] (-2.8,2.7) circle[radius=1pt];
\draw[blue] (-3.2,3.1) -- (-2.8,2.7);
\draw[blue] (-2.9,3) node{\tiny$2$};
\draw[blue] (-3.13,3.1) node[left]{\tiny$2D_2$};
\draw[blue,-stealth] (-2.8,2.735) -- (-2.515,2.45);
\draw[blue,-stealth] (-2.8,2.665) -- (-2.585,2.45);

\draw[black,fill=black] (-5.8,-0.25) circle[radius=1pt];
\draw[black,->] (-5.8,-0.25) -- (-4.8,-0.25);
\draw (-4.9,-0.25) node[right]{\tiny$D_1$};
\draw[black,->] (-5.8,-0.25) -- (-5.8,0.75);
\draw (-5.8,0.71) node[above]{\tiny$D_2$};
\draw[blue,fill=blue] (-5.8,0.4) circle[radius=1pt];
\draw[blue,fill=blue] (-5.15,-0.25) circle[radius=1pt];
\draw[blue] (-5.8,0.4) -- (-5.15,-0.25);
\draw[blue] (-5.4,0.2) node{\tiny$2$};
\draw[blue] (-5.73,0.4) node[left]{\tiny$2D_2$};
\draw[blue,-stealth] (-5.15,-0.215) -- (-4.95,-0.415);
\draw[blue,-stealth] (-5.15,-0.285) -- (-4.95,-0.485);

\draw[black,fill=black] (1.8,2.45) circle[radius=1pt];
\draw[black,->] (1.8,2.45) -- (2.8,2.45);
\draw (2.7,2.45) node[right]{\tiny$D_1$};
\draw[black,->] (1.8,2.45) -- (1.8,3.45);
\draw (1.8,3.41) node[above]{\tiny$D_2$};
\draw[blue,fill=blue] (1.8,3.1) circle[radius=1pt];
\draw[blue,-stealth] (1.8,3.135) -- (2.4855,2.45);
\draw[blue,-stealth] (1.8,3.065) -- (2.415,2.45);
\draw[blue,fill=blue] (1.5,3.3) circle[radius=1pt];
\draw[blue] (1.55,3.3) node[left]{\tiny$D_2$};
\draw[blue] (1.554,2.9) node[left]{\tiny$D_2$};
\draw[blue,fill=blue] (1.5,2.9) circle[radius=1pt];
\draw[blue] (1.5,3.3) -- (1.5,2.9);
\draw[blue,-stealth] (1.5,3.3) -- (1.7,3.3);
\draw[blue,-stealth] (1.5,2.9) -- (1.7,2.9);
\draw[blue,densely dashed] (1.3,3.1) circle[radius=9pt];
\draw[blue,densely dashed] (1.5,2.85) -- (1.8,3.1);
\draw[blue,densely dashed] (1.5,3.35) -- (1.8,3.1);

\draw[black,fill=black] (5.4,-0.25) circle[radius=1pt];
\draw[black,->] (5.4,-0.25) -- (6.4,-0.25);
\draw (6.3,-0.25) node[right]{\tiny$D_1$};
\draw[black,->] (5.4,-0.25) -- (5.4,0.75);
\draw (5.4,0.71) node[above]{\tiny$D_2$};
\draw[blue,fill=blue] (5.4,-0.25) circle[radius=1pt];
\draw[blue,-stealth] (5.4,-0.285) -- (5.6,-0.485);
\draw[blue,-stealth] (5.4,-0.215) -- (5.6,-0.415);
\draw[blue,fill=blue] (5.1,-0.05) circle[radius=1pt];
\draw[blue] (5.15,-0.05) node[left]{\tiny$D_2$};
\draw[blue] (5.154,-0.45) node[left]{\tiny$D_2$};
\draw[blue,fill=blue] (5.1,-0.45) circle[radius=1pt];
\draw[blue] (5.1,-0.05) -- (5.1,-0.45);
\draw[blue,-stealth] (5.1,-0.05) -- (5.3,-0.05);
\draw[blue,-stealth] (5.1,-0.45) -- (5.3,-0.45);
\draw[blue,densely dashed] (4.9,-0.25) circle[radius=9pt];
\draw[blue,densely dashed] (5.1,-0.5) -- (5.4,-0.25);
\draw[blue,densely dashed] (5.1,0) -- (5.4,-0.25);
\end{tikzpicture}
\caption{The cone complex $\Tsf$.}	
\label{fig: T}
\end{figure}

The total puncturing rank is $k_P=2$ and the two puncturing offsets are identical, each giving the divisor $Z_1$. The puncturing substack is $\Vcal(\Tsf) = Z_1 \cap Z_1 = Z_1$ and we conclude (see \Cref{sec: regularly embedded class}):
\begin{equation} \label{eqn: birational invariance counterexample refined class 1} [\Vcal(\Tsf)]^{\refined} = \left\{ \dfrac{(1+Z_1)^2}{1+Z_1} \right\}^1 \cap Z_1 = Z_1^2.\end{equation}
Now blowup the target in $D_1 \cap D_2$ to obtain a pair $(X^\prime|D^\prime)$. Here $D^\prime=D_0^\prime+D_1^\prime+D_2^\prime$ where $D_0^\prime$ is the exceptional divisor and $D_1^\prime,D_2^\prime$ are the strict transforms. The faithful lift of the tangencies given by \Cref{construction: faithful lift} is:
\[ (\upalpha_{ij}^\prime)_{j=0}^2 \colonequals (1,0,-2) \qquad \text{for $i=1,2$.}\]
The lifted curve class is $\upbeta^\prime = 2D_2^\prime$. The associated cone stack $\Tsf^\prime$ is the modification of $\Tsf$ illustrated in \Cref{fig: Tprime}. Geometrically, the logarithmic modification $\Acal(\Tsf^\prime) \to \Acal(\Tsf)$ is the blowup in $Z_1 \cap Z_2$. The cone spanned by $Z_0^\prime$ and $Z_2^\prime$ is redundant since the corresponding stratum does not intersect the puncturing substack, but  we may retain it without affecting the refined class.
\begin{figure}
\begin{tikzpicture}[scale=1.3]

\draw[fill=gray,gray] (0,0) circle[radius=1.5pt];
\draw[gray,thick,->] (0,0) -- (4,0);
\draw[gray] (4,0) node[right]{$Z_3^\prime$};
\draw[gray,thick,->] (0,0) -- (0,4);
\draw[gray] (0,4) node[above]{$Z_1^\prime$};
\draw[gray,thick,->] (0,0) -- (-4,0);
\draw[gray] (-4,0) node[left]{$Z_2^\prime$};
\draw[gray,thick,->] (0,0) -- (-3.5,3.5);
\draw[gray] (-3.7,3.7) node{$Z_0^\prime$};

\draw[black,fill=black] (-0.5,4.75) circle[radius=1pt];
\draw[black,->] (-0.5,4.75) -- (0.5,4.75);
\draw (0.4,4.75) node[right]{\tiny$D_1^\prime$};
\draw[black,->] (-0.5,4.75) -- (-0.5,5.75);
\draw (-0.5,5.71) node[above]{\tiny$D_2^\prime$};
\draw[black,->] (-0.5,4.75) -- (0.3,5.55);
\draw (0.475,5.6) node{\tiny$D_0^\prime$};
\draw[blue,fill=blue] (-0.5,5.4) circle[radius=1pt];
\draw[blue] (-0.43,5.4) node[left]{\tiny$2D_2^\prime$};
\draw[blue,-stealth] (-0.5,5.435) -- (-0.1575,5.0925);
\draw[blue,-stealth] (-0.5,5.365) -- (-0.1925,5.0575);


\draw[black,fill=black] (-3.8,1.45) circle[radius=1pt];
\draw[black,->] (-3.8,1.45) -- (-2.8,1.45);
\draw (-2.9,1.45) node[right]{\tiny$D_1^\prime$};
\draw[black,->] (-3.8,1.45) -- (-3.8,2.45);
\draw (-3.8,2.41) node[above]{\tiny$D_2^\prime$};
\draw[black,->] (-3.8,1.45) -- (-3,2.25);
\draw (-3.1,2.3) node[right]{\tiny$D_0^\prime$};
\draw[blue,fill=blue] (-3.8,2.1) circle[radius=1pt];
\draw[blue,fill=blue] (-3.475,1.775) circle[radius=1pt];
\draw[blue,fill=blue] (-3.3,1.6) circle[radius=1pt];
\draw[blue] (-3.8,2.1) -- (-3.3,1.6);
\draw[blue] (-3.73,2.1) node[left]{\tiny$2D_2^\prime$};
\draw[blue,-stealth] (-3.3,1.635) -- (-3.115,1.45);
\draw[blue,-stealth] (-3.3,1.565) -- (-3.185,1.45);

\draw[black,fill=black] (-2.1,3.15) circle[radius=1pt];
\draw[black,->] (-2.1,3.15) -- (-1.1,3.15);
\draw (-1.2,3.15) node[right]{\tiny$D_1^\prime$};
\draw[black,->] (-2.1,3.15) -- (-2.1,4.15);
\draw (-2.1,4.11) node[above]{\tiny$D_2^\prime$};
\draw[black,->] (-2.1,3.15) -- (-1.3,3.95);
\draw (-1.4,4) node[right]{\tiny$D_0^\prime$};
\draw[blue,fill=blue] (-2.1,3.8) circle[radius=1pt];
\draw[blue,fill=blue] (-1.9,3.6) circle[radius=1pt];
\draw[blue] (-2.1,3.8) -- (-1.9,3.6);
\draw[blue] (-2.03,3.8) node[left]{\tiny$2D_2^\prime$};
\draw[blue,-stealth] (-1.9,3.635) -- (-1.7575,3.4925);
\draw[blue,-stealth] (-1.9,3.565) -- (-1.7925,3.4575);

\draw[black,fill=black] (-4.3,4.15) circle[radius=1pt];
\draw[black,->] (-4.3,4.15) -- (-3.3,4.15);
\draw (-3.4,4.15) node[right]{\tiny$D_1^\prime$};
\draw[black,->] (-4.3,4.15) -- (-4.3,5.15);
\draw (-4.3,5.11) node[above]{\tiny$D_2^\prime$};
\draw[black,->] (-4.3,4.15) -- (-3.5,4.95);
\draw (-3.6,5) node[right]{\tiny$D_0^\prime$};
\draw[blue,fill=blue] (-4.3,4.8) circle[radius=1pt];
\draw[blue,fill=blue] (-3.975,4.475) circle[radius=1pt];
\draw[blue] (-4.3,4.8) -- (-3.975,4.475);
\draw[blue] (-4.23,4.8) node[left]{\tiny$2D_2^\prime$};
\draw[blue,-stealth] (-3.975,4.515) -- (-3.84,4.38);
\draw[blue,-stealth] (-3.975,4.435) -- (-3.875,4.335);

\draw[black,fill=black] (-5.8,-0.25) circle[radius=1pt];
\draw[black,->] (-5.8,-0.25) -- (-4.8,-0.25);
\draw (-4.9,-0.25) node[right]{\tiny$D_1^\prime$};
\draw[black,->] (-5.8,-0.25) -- (-5.8,0.75);
\draw (-5.8,0.71) node[above]{\tiny$D_2^\prime$};
\draw[black,->] (-5.8,-0.25) -- (-5,0.55);
\draw (-5.1,0.6) node[right]{\tiny$D_0^\prime$};
\draw[blue,fill=blue] (-5.8,0.4) circle[radius=1pt];
\draw[blue,fill=blue] (-5.475,0.075) circle[radius=1pt];
\draw[blue,fill=blue] (-5.15,-0.25) circle[radius=1pt];
\draw[blue] (-5.8,0.4) -- (-5.15,-0.25);
\draw[blue] (-5.73,0.4) node[left]{\tiny$2D_2^\prime$};
\draw[blue,-stealth] (-5.15,-0.215) -- (-4.95,-0.415);
\draw[blue,-stealth] (-5.15,-0.285) -- (-4.95,-0.485);

\draw[black,fill=black] (1.8,2.45) circle[radius=1pt];
\draw[black,->] (1.8,2.45) -- (2.8,2.45);
\draw (2.7,2.45) node[right]{\tiny$D_1^\prime$};
\draw[black,->] (1.8,2.45) -- (1.8,3.45);
\draw (1.8,3.41) node[above]{\tiny$D_2^\prime$};
\draw[black,->] (1.8,2.45) -- (2.6,3.25);
\draw (2.75,3.3) node{\tiny$D_0^\prime$};
\draw[blue,fill=blue] (1.8,3.1) circle[radius=1pt];
\draw[blue,-stealth] (1.8,3.135) -- (2.14275,2.79275);
\draw[blue,-stealth] (1.8,3.065) -- (2.1075,2.7575);
\draw[blue,fill=blue] (1.5,3.3) circle[radius=1pt];
\draw[blue] (1.55,3.3) node[left]{\tiny$D_2^\prime$};
\draw[blue] (1.554,2.9) node[left]{\tiny$D_2^\prime$};
\draw[blue,fill=blue] (1.5,2.9) circle[radius=1pt];
\draw[blue] (1.5,3.3) -- (1.5,2.9);
\draw[blue,-stealth] (1.5,3.3) -- (1.7,3.3);
\draw[blue,-stealth] (1.5,2.9) -- (1.7,2.9);
\draw[blue,densely dashed] (1.3,3.1) circle[radius=9.5 4pt];
\draw[blue,densely dashed] (1.5,2.85) -- (1.8,3.1);
\draw[blue,densely dashed] (1.5,3.35) -- (1.8,3.1);

\draw[black,fill=black] (5.4,-0.25) circle[radius=1pt];
\draw[black,->] (5.4,-0.25) -- (6.4,-0.25);
\draw (6.3,-0.25) node[right]{\tiny$D_1^\prime$};
\draw[black,->] (5.4,-0.25) -- (5.4,0.75);
\draw (5.4,0.71) node[above]{\tiny$D_2^\prime$};
\draw[black,->] (5.4,-0.25) -- (6.2,0.55);
\draw (6.375,0.6) node{\tiny$D_0^\prime$};
\draw[blue,fill=blue] (5.4,-0.25) circle[radius=1pt];
\draw[blue,-stealth] (5.4,-0.285) -- (5.6,-0.485);
\draw[blue,-stealth] (5.4,-0.215) -- (5.6,-0.415);
\draw[blue,fill=blue] (5.1,-0.05) circle[radius=1pt];
\draw[blue] (5.15,-0.05) node[left]{\tiny$D_2^\prime$};
\draw[blue] (5.154,-0.45) node[left]{\tiny$D_2^\prime$};
\draw[blue,fill=blue] (5.1,-0.45) circle[radius=1pt];
\draw[blue] (5.1,-0.05) -- (5.1,-0.45);
\draw[blue,-stealth] (5.1,-0.05) -- (5.3,-0.05);
\draw[blue,-stealth] (5.1,-0.45) -- (5.3,-0.45);
\draw[blue,densely dashed] (4.9,-0.25) circle[radius=9.5pt];
\draw[blue,densely dashed] (5.1,-0.5) -- (5.4,-0.25);
\draw[blue,densely dashed] (5.1,0) -- (5.4,-0.25);

\end{tikzpicture}
\caption{The cone complex $\Tsf^\prime$.}
\label{fig: Tprime}
\end{figure}

Again the two puncturing offsets are identical, each giving the divisor $Z_1^\prime$. We see that $\Vcal(\Tsf^\prime)=Z_1^\prime \cap Z_1^\prime = Z_1^\prime$ is the strict transform of $\Vcal(\Tsf)=Z_1$ (note that this simple relationship between puncturing substacks does not hold for other lifts of the numerical data). The refined class is:
\[ [\Vcal(\Tsf^\prime)]^{\refined} = \left\{ \dfrac{(1+Z_1^\prime)^2}{1+Z_1^\prime} \right\}^1 \cap Z_1^\prime = (Z_1^\prime)^2.\]
Consider the morphism $\upnu \colon \Acal(\Tsf^\prime) \to \Acal(\Tsf)$. We have $\upnu^\star Z_1 = Z_1^\prime + Z_0^\prime$ from which we obtain
\[ [\Vcal(\Tsf^\prime)]^{\refined} = (\upnu^\star Z_1 - Z_0^\prime)^2 = \upnu^\star(Z_1^2) - 2 \upnu^\star(Z_1) Z_0^\prime + (Z_0^\prime)^2.\]
We then compute:
\[ \upnu_\star [\Vcal(\Tsf^\prime)]^{\refined} = Z_1^2 - Z_1 Z_2.\]
Comparing with \eqref{eqn: birational invariance counterexample refined class 1} we see that $\upnu_\star[\Vcal(\Tsf^\prime)]^\refined \neq [\Vcal(\Tsf)]^{\refined}$. Similar calculations show that in fact no choice of lifted numerical data achieves birational invariance in this case.
\end{example}

\footnotesize
\bibliographystyle{alpha}
\bibliography{Bibliography.bib}

\noindent Luca Battistella, University of Bologna, \href{mailto:luca.battistella2@unibo.it}{luca.battistella2@unibo.it}.

\noindent Navid Nabijou, Queen Mary University of London, \href{mailto:n.nabijou@qmul.ac.uk}{n.nabijou@qmul.ac.uk}.

\noindent Dhruv Ranganathan, University of Cambridge, \href{mailto:dr508@cam.ac.uk}{dr508@cam.ac.uk}.\smallskip

\noindent \textbf{Funding.} L.B. was supported by: Deutsche Forschungsgemeinschaft (DFG, German Research Foundation) under Germany's Excellence Strategy EXC 2181/1-390900948 (the Heidelberg STRUCTURES Excellence Cluster); ERC Advanced Grant SYZYGY of the European Research Council (ERC) under the European Union Horizon 2020 Research and Innovation Program (grant agreement No. 834172); the European Union - NextGenerationEU under the National Recovery and Resilience Plan (PNRR) - Mission 4 Education and Research - Component 2 From Research to Business - Investment 1.1 Notice Prin 2022 - DD N.104 del 2/2/2022, from title "Symplectic varieties: their interplay with Fano manifolds and derived categories", proposal code 2022PEKYBJ – CUP J53D23003840006. L.B. is a member of INdAM group GNSAGA. N.N. was supported by the Herchel~Smith~Fund. D.R. was supported by EPSRC New Investigator Grant EP/V051830/1. 

\end{document}